\documentclass{svjour3}
\usepackage[utf8]{inputenc} 
\usepackage[T1]{fontenc}
\usepackage[english]{babel}
\usepackage{enumerate}
\usepackage{graphicx}   
\usepackage{subcaption}
\usepackage{booktabs}
\usepackage{mathtools}
\usepackage{amsmath,bm}
\usepackage{amssymb}
\usepackage{amsfonts}
\usepackage{fancyhdr}
\usepackage{fourier} 
\usepackage{pgfplots} 
\usepackage{cancel}
\usepackage{cellspace}
\usepackage{epstopdf}

\ifpdf
\DeclareGraphicsExtensions{.eps,.pdf,.png,.jpg}
\else
\DeclareGraphicsExtensions{.eps} 
\fi

\usepackage[square,sort,comma,numbers]{natbib}                            
\usepackage{tikz- cd}

\numberwithin{equation}{section}					
\numberwithin{figure}{section}						
\numberwithin{table}{section}						

\usepackage{listings}								
\usepackage[top=3cm, bottom=3cm,
			inner=3cm, outer=3cm]{geometry}			
\usepackage{eso-pic}								

\usepackage{float} 									
\usepackage{parskip}								
\usepackage{tcolorbox}

\newcommand{\RR}{\mathbb{R}}

\newcommand{\RT}{\mathbf{RT}}
\newcommand{\BDM}{\mathbf{BDM}}

\newcommand{\FF}{\mathcal{F}}

\DeclareMathOperator{\dive}{div}
\DeclareMathOperator{\crl}{curl}

\newcommand{\jump}[1]{\llbracket #1 \rrbracket}
\newcommand{\bfV}{\mathbf{V}}

\newcommand{\bfH}{\mathbf{H}}

\newcommand{\bfx}{\mathbf{x}}

\newcommand{\bfu}{\pmb{u}}
\newcommand{\bfv}{\pmb{v}}
\newcommand{\bff}{\pmb{f}}
\newcommand{\bfg}{\pmb{g}}
\newcommand{\bfn}{\pmb{n}}
\newcommand{\bft}{\pmb{t}}

\newcommand{\mesh}{\mathcal{T}_h}
\newcommand{\mcT}{\mathcal{T}}
\newcommand{\mcG}{\mathcal{G}}
\newcommand{\mcM}{\mathcal{M}}
\newcommand{\mcE}{\mathcal{E}}
\newcommand{\mcF}{\mathcal{F}}

\newcommand{\lagp}{\alpha}
\newcommand{\uB}{\pmb{g}}

\newcommand{\kp}{\hat{k}}

\newcommand{\mcTT}{\mcT_{h}(T_L)} 
\newcommand{\mcTL}{\mcT_{H}} 
\newcommand{\mcFM}{\mcF_{h}(M)} 

\newcommand{\diff}{\mathop{ }\mathopen{ }\mathrm{d}}
\newcommand{\restreinta}[1]{\mathclose{}|\mathopen{}_{#1}}
\newcommand{\Wk}{W_{k+1,h}}
\newcommand{\Vk}{\bfV_{k,h}}
\newcommand{\Vkg}{\bfV_{k,h,{\pmb{g}}}}
\newcommand{\Vkz}{\bfV_{k,h,{0}}}

\newcommand{\Qk}{Q_{\kp,h}}
\newcommand{\Qkk}{Q_{k,h}}
\newcommand{\Qkkg}{Q_{\tilde{k},h}^{\partial\Omega}}
\newcommand{\Af}{\mathcal{A}}
\newcommand{\Bf}{\mathcal{B}}

\newcommand{\Omdh}{\Omega_{\mcT_{h}}}

\newcommand{\nDj}{D^j_{\bfn_F}}
\newcommand{\avgut}{\{\mu\nabla(\bfu_h\cdot \pmb{t})\cdot\bfn\}}
\newcommand{\avgvt}{\{\mu\nabla(\bfv_h\cdot \pmb{t})\cdot\bfn\}}
\newcommand{\jumput}{\jump{\bfu_h\cdot \pmb{t}}}
\newcommand{\jumpvt}{\jump{\bfv_h\cdot \pmb{t}}}

\spnewtheorem{prop}{Proposition}{\bf}{\it}
\spnewtheorem{thm}{Theorem}{\bf}{\it}

\begin{document} 

\title{Divergence-free cut finite element methods for Stokes flow \thanks{This research was supported by the Swedish Research Council Grant No. 2018-05262 and the Wallenberg Academy Fellowship KAW 2019.0190.}}
\author{Thomas Frachon \and Erik Nilsson * \and Sara Zahedi}
\institute{
     Department of Mathematics,  KTH Royal Institute of Technology,  SE-100 44 Stockholm, Sweden.  \\
     *\email{erikni6@kth.se} 
    }
\date{31/08/2024}

\maketitle

\begin{abstract}
We develop two unfitted finite element methods for the Stokes equations based on $\bfH^{\dive}$-conforming finite elements. Both cut finite element methods exhibit optimal convergence order for the velocity, pointwise divergence-free velocity fields, and well-posed linear systems, independently of the position of the boundary relative to the computational mesh. The first method is a cut finite element discretization of the Stokes equations based on the Brezzi-Douglas-Marini (BDM) elements and involves interior penalty terms to enforce tangential continuity of the velocity at interior edges in the mesh. The second method is a cut finite element discretization of a three-field formulation of the Stokes problem involving the vorticity, velocity, and pressure and uses the Raviart-Thomas (RT) space for the velocity. We present mixed ghost penalty stabilization terms for both methods so that the resulting discrete problems are stable and the divergence-free property of the $\bfH^{\dive}$-conforming elements is preserved also on unfitted meshes. In both methods boundary conditions are imposed weakly. We show that imposing Dirichlet boundary conditions weakly introduces additional challenges; 1) The divergence-free property of the RT and the BDM finite elements may be lost depending on how the normal component of the velocity field at the boundary is imposed. 2) Pressure robustness is affected by how well the boundary condition is satisfied and may not hold even if the incompressibility condition holds pointwise. We study two approaches of weakly imposing the normal component of the velocity at the boundary; we either use a penalty parameter and Nitsche's method or a Lagrange multiplier method. We show that appropriate conditions on the velocity space has to be imposed when Nitsche's method or penalty is used. Pressure robustness can hold with both approaches by reducing the error at the boundary but the price we pay is seen in the condition numbers of the resulting linear systems, independent of if the mesh is fitted or unfitted to the boundary. 

\keywords{mass conservation \and mixed finite element methods \and unfitted, Stokes equations \and cut finite element method \and pressure robustness \and vorticity-velocity-pressure formulation}
\subclass{65N30 \and 65N85 \and 65N22}

\end{abstract}

\section{Introduction}\label{sec: intro}
Computational methods for simulations of incompressible fluids that fail to produce pointwise divergence-free velocity approximations may result in numerical solutions with large errors and instabilities \cite{VolkerDivConstr2017}.  In this paper our aim is to present unfitted finite element discretizations of the Stokes equations that produce pointwise divergence-free velocity approximations, independently of how the boundary is positioned relative to the computational mesh.

Let $\Omega \subset \RR^d$ be an open bounded domain with a piecewise smooth boundary $\partial\Omega$, occupied by an incompressible fluid. Denote by $\bfn$ the normal vector to $\partial \Omega$. Let $\mathbf{L}^2(\Omega)$ be the space of vector-valued functions with each component in $L^2(\Omega)$. Given a source $\pmb{f}\in \mathbf{L}^2(\Omega)$, boundary data $\uB\in \mathbf{L}^{2}(\Omega)$ such that $\int_{\partial \Omega} \uB \cdot \bfn =0$, and fluid viscosity $\mu>0$, we seek fluid velocity $\bfu: \Omega \rightarrow \RR^d$ and pressure $p: \Omega \rightarrow \RR$ satisfying the Stokes equations:
\begin{alignat}{2}
	-\dive(\mu \nabla\bfu - p\pmb{I}) & = \pmb{f} 
	&& \text{ in } \Omega, \label{eq:momentum-balance} \\
	\hfill \dive\bfu & =  0
	&&\text{ in } \Omega, \label{eq:conservation} \\
	\hfill \bfu & = \pmb{g}
	&&\text{ on } \partial\Omega. \label{eq:BCu} 
\end{alignat}

Unfitted discretizations can be an alternative to standard fitted finite element discretizations when the domain $\Omega$ is complicated to mesh. In applications where the boundary $\partial \Omega$ or interfaces (internal boundaries) are evolving discretizations that can avoid remeshing processes are of great interest, see e.g. \cite{FrZa19, FraZah23} for unfitted discretizations of two-phase flow problems.

For the Stokes equations Cut Finite Element Methods (CutFEM) with optimal rates of convergence of the approximate velocity and pressure have been developed, see e.g.  \cite{BecBurHan09, HaLaZa14, KirGroReu16}  for the Stokes interface problem and see e.g. \cite{BurHan14, MaLaLoRo14} for the Stokes fictitious domain problem.  A popular strategy for ensuring that the discretization is robust and the discrete problem is well-posed independent of the position of the boundary or the interface relative to the computational mesh, is to add ghost penalty terms \cite{Bu10} in the weak form, see for example the works mentioned above. The ghost penalty terms also ensure that the condition number of the resulting matrix from the cut finite element discretization scales with mesh size as for the fitted finite element discretization. An alternative to adding stabilization terms in the weak formulation is to agglomerate elements (merge elements). In that case one needs to have methods that can deal with the general shape that agglomerated elements can have, see e.g. \cite{BurDelErn20}.

Existing cut finite element discretizations for the Stokes problem are mostly based on popular inf-sup stable element pairs such as the Taylor-Hood elements see~\cite{KirGroReu16}, mini-elements \cite{BurHan14}, or the P1 iso P2-P1 elements~\cite{HaLaZa14} and do not result in pointwise divergence-free approximations. Recently, unfitted discretizations for the Stokes problem based on finite element spaces that result in pointwise divergence-free velocity approximations and are conforming in $\bfH^1$ have been developed and analyzed in \cite{LiuNeiOls23, BurHanLarst23}. In~\cite{LiuNeiOls23} a cut finite element discretization based on the Scott-Vogelius pair is proposed. However, the approximate velocity is pointwise divergence-free only outside a band around the unfitted boundary. In \cite{BurHanLarst23} unfitted discretizations are developed using a low order divergence-free element. The proposed discretizations in \cite{BurHanLarst23}  do inherit the pointwise divergence-free property of the underlying element but fail to be pressure robust. It is unclear if the condition numbers for the linear systems associated to the unfitted discretizations in \cite{BurHanLarst23} are controlled and scale with mesh size as for the fitted method.

The construction of efficient divergence-free mixed finite element discretizations of the Stokes problem is not trivial, especially not on general triangulations and in three space dimensions. Several strategies exists, with their pros and cons, see e.g. \cite{VolkerDivConstr2017} and references therein. The Scott-Vogelius finite element pair which was used in~\cite{LiuNeiOls23}, where continuous piecewise polynomials of degree $k$ are used for approximating the velocity and discontinuous piecewise polynomials of degree $(k-1)$ are used for the pressure, has the advantage that on special triangulations the element pair is inf-sup stable in $\RR^d$ for $k\geq d$. On unfitted meshes these triangulations are not complicated to contruct and thus this element pair can provide a simple and good base for an inf-sup stable and divergence-free CutFEM for the Stokes problem. Our experience is that using the stabilization terms we propose in \cite{FraHaNilZa22} instead of the ones in~\cite{LiuNeiOls23} it is possible to modify the discretization in~\cite{LiuNeiOls23} to retain optimal divergence-free velocity fields in the entire computational domain using the Scott-Vogelius pair. However, to obtain optimal accuracy a high order representation of the boundary and high order integration on cut elements are required. In this work we choose to study in more detail two first and second order accurate discretizations of the Stokes problem and focus on studying how to guarantee unfitted discretizations to inherit the divergence-free property of the underlying element and being pressure robust. In this paper, we work in two space dimensions, i.e., $d=2$. However, we believe our main findings and conclusions are valid also in three space dimensions and are of interest also when other divergence-free elements are used. 

A way to bypass the challenge of constructing conforming, inf-sup stable, and divergence-free elements is to develop discretizations based on $\bfH^{\dive}$-conforming elements. The Brezzi-Douglas-Marini (BDM) space and Raviart-Thomas (RT) space are two $\bfH^{\dive}$-conforming finite element spaces \cite{BrFo12}. Together with their appropriate pressure spaces they form inf-sup stable pairs for the Stokes equations. However, these spaces are non-conforming with respect to $\bfH^1$. One approach of handling the lack of smoothness of these spaces has been to modify the weak formulation using interior penalty terms as in discontinuous Galerkin methods to enforce tangential continuity across interior edges of the mesh~\cite{WangHdiv2007, CocKanSch07}. We combine this strategy (presented for fitted meshes) with ghost penalty stabilization terms we developed in \cite{FraHaNilZa22} and present a cut finite element discretization for the Stokes problem based on the BDM space for the velocity with the following properties: 
\begin{enumerate}
\item Optimal rates of convergence of the approximate velocity
  and pressure;
\item Well-posed linear systems where the condition number of the system matrix scales as for the fitted finite element discretization;
\item Pointwise divergence-free velocity approximations;
\end{enumerate}

We also propose an unfitted discretization based on the vorticity-velocity-pressure formulation of the Stokes equations. For some boundary conditions this formulation is beneficial \cite{Dub02,Hanot2021} since discretizations that fall into an appropriate de Rham complex have been developed for fitted meshes~\cite{Hanot2021} using the Raviart-Thomas (RT) space for the velocity. We develop a cut finite element discretization for this formulation based on the Raviart-Thomas space for the velocity field and propose new ghost penalty stabilization terms. We study two strategies for imposing Dirichlet boundary conditions weakly. The tangential component of the velocity field is imposed as a natural boundary condition but the normal component, $\bfu \cdot \bfn$,  is weakly imposed either via a penalty parameter or via a Lagrange multiplier variable.

Our challenge has been to develop unfitted discretizations that have all the three properties above. We have chosen to utilize $\bfH^{\dive}$-conforming elements in order to be able to also obtain high order approximations for the Stokes equations. In order to guarantee pointwise divergence-free approximations we show that it is necessary to use proper stabilization terms and to guarantee that the approximate velocity $\bfu_h$ satisfies $\int_{\partial \Omega} \bfu_h \cdot \bfn=0$. 

The third property above usually implies pressure robustness for standard mixed finite element methods \cite{VolkerDivConstr2017}. Our numerical results indicate that this is only the case when the normal component $\bfu \cdot \bfn$ at the boundary is prescribed strongly. Thus, when the boundary condition is imposed weakly pressure robustness can be lost due to the error at the boundary. 
We do not analyze the proposed methods in this paper but we provide a careful numerical study that illustrates challenges for unfitted discretizations. In particular, we discuss errors due to weak imposition of essential boundary conditions and study how to reduce these errors.

The outline of the paper is as follows. In Section \ref{sec:Nummeth} we introduce the computational mesh and the finite element spaces which our discretizations will be built upon. In Section \ref{sec:method_wang}  we present a discretization for the standard velocity-pressure formulation and in Section \ref{sec:method_vort} we present two dicretizations of the vorticity-velocity-pressure formulation of the Stokes problem. In Section~\ref{sec: the divfree prop} we show that the proposed discretizations results in velocity approximations that are pointwise divergence-free. We discuss the influence on the divergence-free property of the approximate velocity when Dirichlet boundary conditions are enforced weakly in Section~\ref{sec:lagmult}. Efficient implementations of the stabilization terms are discussed in Section~\ref{sec:equivalent_stab}. We present results from numerical experiments in Section \ref{sec:numex} and finally, in Section \ref{sec:conclusion} we summarize our findings.

\section{The finite element mesh and spaces} \label{sec:Nummeth}
We start by introducing the computational mesh and the finite element spaces we need in order to discuss our numerical methods.

\subsection{Mesh}\label{sec:mesh_struct}
Introduce a computational domain $\Omega_0\supseteq\Omega$ with polygonal boundary $\partial \Omega_0$. Let $\{ \mcT_{0,h} \}_{h}$ be a quasi-uniform family of simplicial meshes of $\Omega_0.$ We denote by $h$ the piecewise constant function that on element $T\in \mcT_{0,h}$ is equal to $h_T>0$, the diameter of $T$. Let $\max_{T\in\mcT_{0,h}} h_T < h_0$ with $h_0<<1$ small. Note that 
\begin{align}
\Omega_0=\bigcup_{T\in\mcT_{0,h}}T.
\end{align} 
We call $\mcT_{0,h}$ the background mesh, and from it we construct the active mesh 
\begin{align}
	\mesh:= \{T\in\mcT_{0,h} : |T\cap \Omega | > 0 \}
\end{align}	
on which we will define our finite element spaces. 
The active mesh $\mesh$ constitutes the active domain 
\begin{align}
	\Omdh := \bigcup_{T\in\mesh} T.
\end{align}	
Denote by $\mcE_h$ the set of edges that are shared by two elements in $\mesh$. 
The corresponding intersection with $\Omega$ is defined by
\begin{equation}
  \mcE_{h,\Omega} := \{F = E\cap \Omega:  E \in \mcE_h \}. 
\end{equation}

The background mesh $\mcT_{0,h}$ conforms to the fixed polygonal boundary $\partial \Omega_0$ but the active mesh $\mesh$ does not need to conform to $\partial\Omega$. 
We assume that $\partial\Omega$ intersects the boundary $\partial T$ of an element $T \in \mcT_h$ exactly twice and each (open) edge at most once. 
Let $\mcG_h$ denote the set of elements intersected by $\partial \Omega$, 
\begin{align}
  \mcG_h := \left\{T \in \mcT_h : | \partial \Omega \cap \overline{T}| > 0 \right\}.
\end{align}
Let $\mcF_h$ denote the set of edges in $\mcG_h$ that are shared by two elements in $\mesh$, and denote by $\mcF_{h,\partial\Omega}$ the set of interior edges in $\mcG_h$. These sets are illustrated in Figure \ref{fig:static active mesh}.
\begin{figure}[h!]
\centering
	\begin{subfigure}[b]{0.3 \textwidth}  	 	 	
	\centering	
	\includegraphics[scale=0.15]{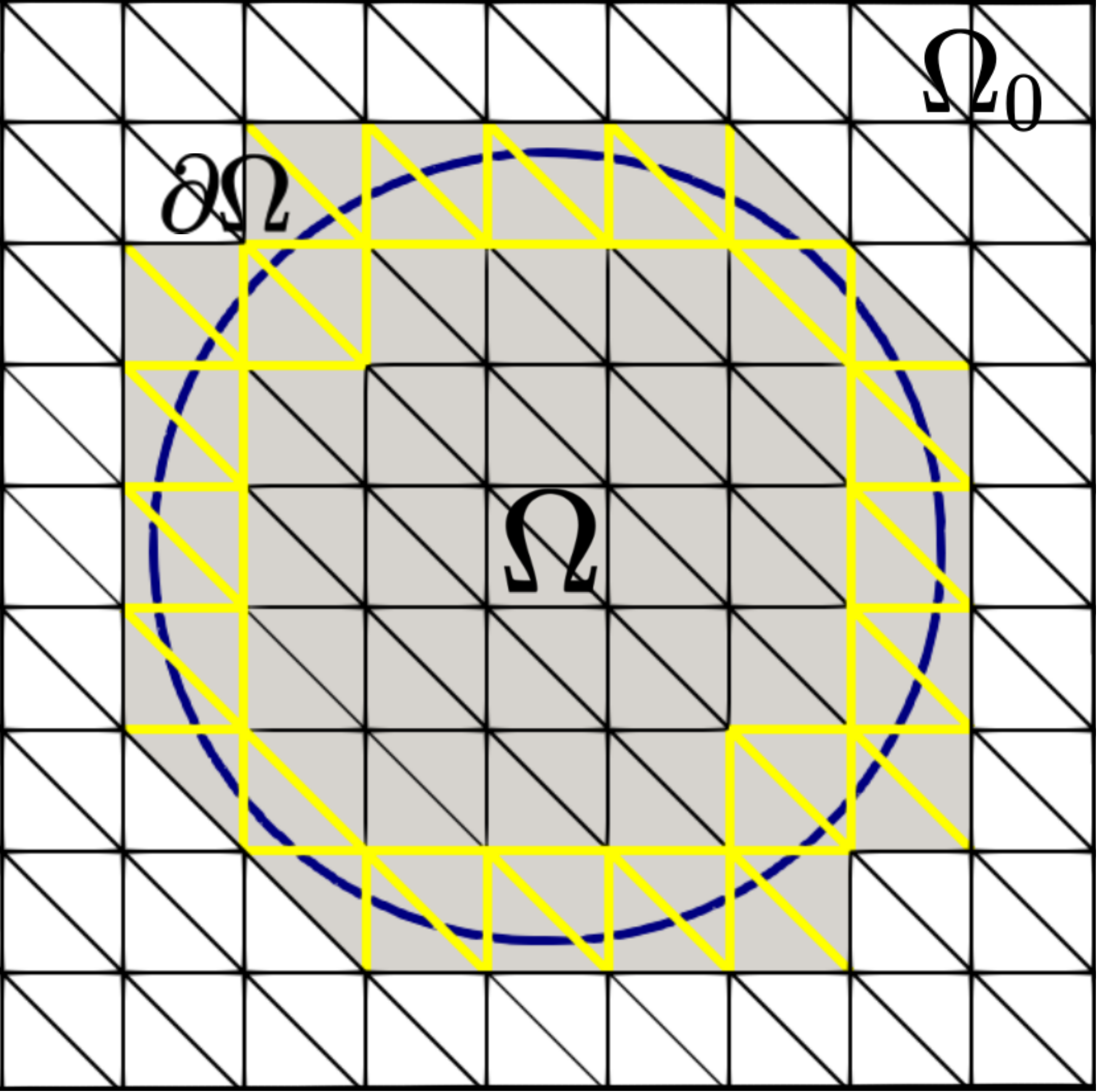} 
	\end{subfigure}
	\begin{subfigure}[b]{0.3 \textwidth}  	 	 	
	\centering	
	\includegraphics[scale=0.15]{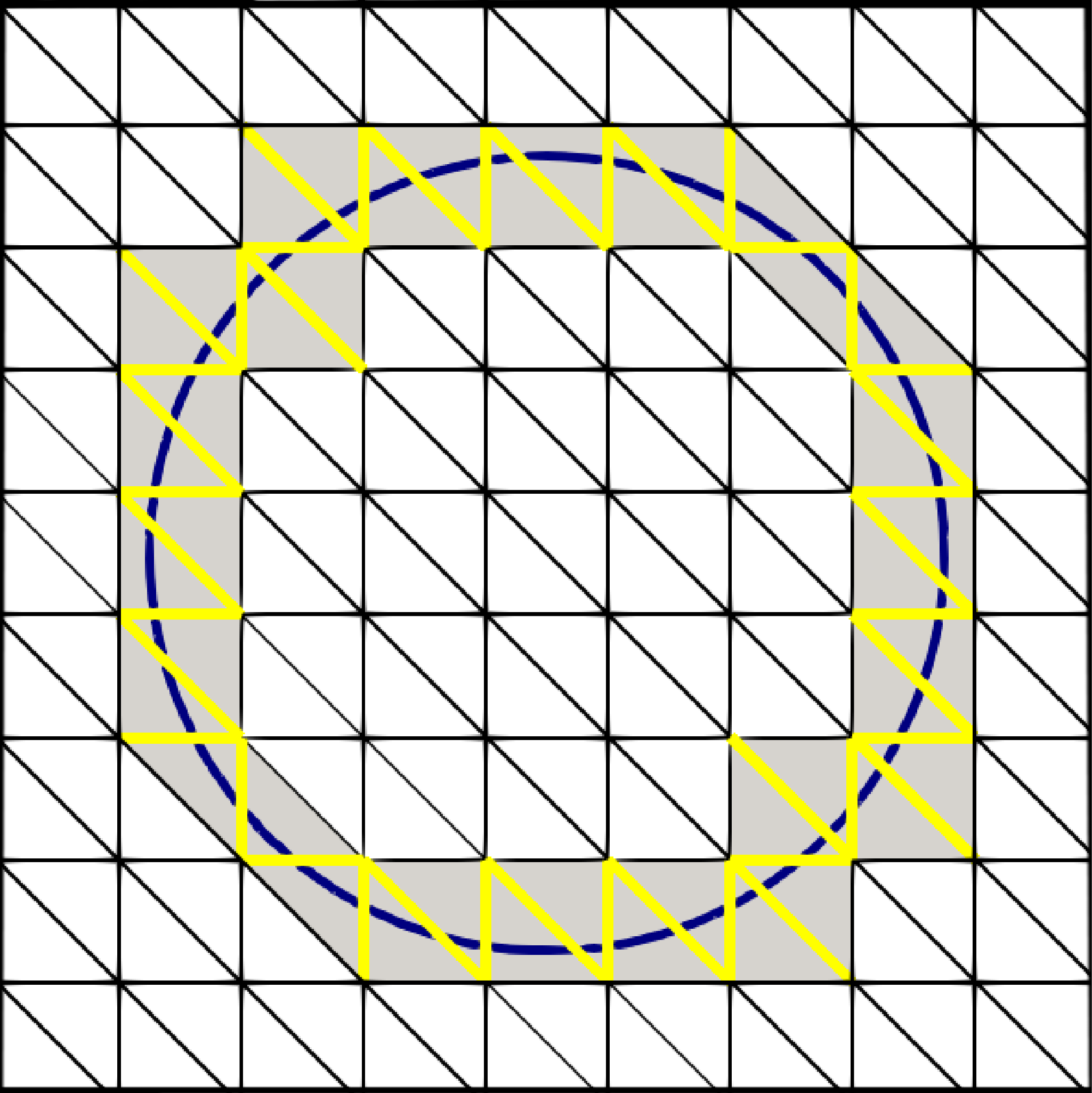} 
	\end{subfigure}
      \caption{Illustration of the active mesh. Left: The grey triangles are the elements in the set $\mesh$, which constitutes the active mesh $\Omega_{\mcT_h}$ and the yellow edges are edges in $\mcF_{h}$. Right: The grey triangles are in the set $\mcG_h$ and the yellow edges are edges in $\mcF_{h,\partial\Omega}$.}  
        \label{fig:static active mesh}      	
\end{figure}

\subsection{Notation}
Let $D\subset\RR^2$. We will work with the Sobolev spaces 
\begin{align}
  \bfH^{1}(D) &:= \left\{\bfv\in\mathbf{L}^2(D) : \nabla\bfv\in \mathbf{L}^2(D)\right\}, \\
	\bfH^{\dive}(D) &:= \left\{\bfv\in\mathbf{L}^2(D) : \dive\bfv\in L^2(D)\right\}, \label{eq:Hdiv}\\
	H^{1}(D) &:= \{v\in L^2(D) : \nabla v\in \mathbf{L}^2(D) \}.
\end{align}
We use boldface letters for vector-valued functions and their associated spaces. For a vector field $\bfu=\begin{bmatrix}
	u_1 \\ u_2
\end{bmatrix}$ we set $\crl\bfu:=\partial_2 u_1 - \partial_1 u_2$, and for a scalar $p$ we set $\crl p = \begin{bmatrix}-\partial_2 p\\ \partial_1 p\end{bmatrix}$. Note that $H^{\crl}(D):=\{p\in L^2(D) : \crl p \in L^2(D)\}\cong H^1(D),$ where $\cong$ denotes isomorphism between Hilbert spaces.

We use the notation 
\begin{equation} \label{eq:innerp}
	(v,w)_{D}=\int_{D} v(\bfx) w(\bfx) \diff \bfx
\end{equation}
for the $L^2$-inner product on the domain $D$. The induced norm is denoted by $\|\cdot \|_{D}$. 
Regarding normal vectors, we take the convention that the normal vector $\bfn$ to $\partial\Omega$ is outward pointing. We also introduce a unique normal $\bfn$ to each edge $F=\partial T \cap \partial T'$, where $T, T' \in\mesh$. 
We define the tangent vector $\bft$ as the counter-clockwise rotation by $\pi/2$ of $\bfn$. 
We define the jump and average operators of a scalar function $p:\Omdh\to\RR$ at $F=\partial T \cap \partial {T'}$ by
\begin{align}
	\jump{p} &= (p|_T - p|_{T'})\restreinta{F} \label{eq:jump}, \\
	\{p \} &= 1/2 (p|_T + p|_{T'})\restreinta{F}. \label{eq:aver}
\end{align} 
We use identical notations for vector-valued functions.

For two functions $a,b$ we write $a\lesssim b$ if and only if $a\leq c b$ for some constant $c$ independent of $h$ and how the boundary $\partial\Omega$ cuts the mesh $\mesh$. Similarly we write $a\gtrsim b$.

\subsection{Discrete spaces}\label{sec:spaces}
Given an element $T\in\mesh$, let $Q_k(T)$ denote the space of polynomial functions on $T$ of degree less than or equal to $k\geq 0$. Define the following finite element spaces on the mesh $\mesh$:
\begin{align} \label{eq:Vi}
	\Wk &:= \left\{\omega_h\in H^{1}(\Omega_{\mcT_{h}}): \ \omega_h\lvert_T\ \in Q_{k+1}(T), \ \forall T\in \mcT_{h} \right\}, \\
	\Vk & :=\left\{\bfv_h\in \bfH^{\dive}(\Omega_{\mcT_{h}}): \ \bfv_h\lvert_T\ \in \bfV_k(T), \ \forall T\in \mcT_{h} \right\}, \\
	\Qk & := \bigoplus_{T \in \mcT_{h}} Q_{\kp}(T).
\end{align}
Here $\Wk$ is the standard space of continuous piecewise polynomials of degree less than or equal to $k+1\geq 1$,  $\bfV_k(T)$ is either the Raviart-Thomas space $\RT_{k}(T)$ or the Brezzi-Douglas-Marini space $\BDM_{k}(T)$, and $\Qk$ is the space of discontinuous piecewise polynomial functions of degree less than or equal to $\kp$. 
Note that having chosen $\Vk$ then $\Qk$ is implicitly chosen. That is 
when $\bfV_k=\RT_{k}$, we take $\kp=k\geq 0$ and when $\bfV_k=\BDM_k$, we take $\kp=k-1 \geq 0$.  For the definition of these spaces we refer to for example Chapter III.3 in \cite{BrFo12} or to \cite{BofBreFor13}. 
The space $\Vk$ is the $\bfH^{\dive}$-conforming finite element space we consider for the approximation space of the velocity and $\Qk$ is the approximation space for the pressure.

In case of Dirichlet boundary conditions on $\partial \Omega$ the pressure is only determined up to a constant. We therefore introduce the following 
subspace of $\Qk$
\begin{align}
	\Qk^0 & := \left\{q_h\in\Qk: \int_{\Omega} q_h = 0 \right\}.\label{eq:defQ0}
\end{align}
When the boundary condition is imposed weakly using a penalty parameter or Nitsche's method we will need the following subspace of our discrete space $\Vk$ 
\begin{align}
  \Vk^0 & :=\left\{\bfv_h\in\Vk: \int_{\partial \Omega} \bfv_h\cdot\bfn = 0 \right\}.
\end{align}
We will discuss why we need $\Vk^0$ in Section \ref{sec:lagmult}. 


  
\section{A non-conforming mixed method - Method NC}
\label{sec:method_wang}
Functions in $\Vk$ are tangentially discontinuous vector-valued functions, meaning that for $\bfv \in \Vk$, the tangential component $\bfv\cdot\pmb{t}$ is a discontinuous function across interior edges in the mesh. The space $\Vk$ lacks the smoothness to be conforming with respect to $\bfH^1(\Omdh)$, thus $\Vk\not\subset\bfH^1(\Omdh)$. One approach of handling the lack of smoothness is to modify the weak formulation using interior penalty terms to enforce tangential continuity across all interior edges of the mesh, see \cite{WangHdiv2007, WangRobust2009}. Following this approach we present a cut finite element discretization of the Stokes problem~\eqref{eq:momentum-balance}-\eqref{eq:BCu}. Find $(\bfu_h,p_h)\in \Vk \times \Qk^0$ such that 
\begin{alignat}{2}
    \Af(\bfu_h,\bfv_h) - \Bf(\bfv_h,p_h) &= \FF (\bfv_h),\ &&\text{ for all } \bfv_h\in\Vk^0, \label{eq:mNC_discreteStokes} \\
     \Bf_0(\bfu_h,q_h) &= 0, &&\text{ for all } q_h\in\Qk. \label{eq:mNC_conservation}
\end{alignat}
Here 
\begin{align}
   &\Af(\bfu_h,\bfv_h) := a(\bfu_h,\bfv_h) + s_{a}(\bfu_h,\bfv_h), \label{eq:Af} \\
   & \Bf(\bfu_h,q_h) := b(\bfu_h,q_h)+s_b(\bfu_h,q_h), \label{eq:Bf} \\
   & \Bf_0(\bfu_h,q_h) := b_0(\bfu_h,q_h)+s_b(\bfu_h,q_h), \label{eq:Bf0}
\end{align}

\begin{align}
  a(\bfu_h,\bfv_h) &:=\sum_{T\in\mesh} (\mu \nabla\bfu_h,\nabla\bfv_h)_{T \cap \Omega} \nonumber \\
 &\quad +\sum_{T\in \mcG_h} \left( -(\mu\nabla\bfu_h\cdot\bfn,\bfv_h)_{T\cap \partial\Omega} + (\bfu_h,\mu\nabla\bfv_h\cdot\bfn)_{T \cap \partial\Omega}+(\lambda_{\bfu} h^{-1}\bfu_h,\bfv_h)_{T\cap \partial\Omega}\right)
  +t(\bfu_h,\bfv_h), \label{eq:a}\\
	t(\bfu_h,\bfv_h) &:= \sum_{F\in \mcE_{h,\Omega}} \left( -(\avgut,\jumpvt)_F + (\jumput,\avgvt)_F + (\lambda_t h^{-1}\jumput,\jumpvt)_F \right), \label{eq:t}\\
	b_0(\bfv_h,p_h) &:= (\dive \bfv_h,p_h)_{\Omega}, \label{eq:b0} \\
    b(\bfv_h,p_h) &:=b_0(\bfv_h,p_h) -\sum_{T\in \mcG_h}(\bfv_h\cdot\bfn,p_h)_{T \cap \partial\Omega}, \label{eq:b} \\
	\FF(\bfv_h) &:= (\pmb{f},\bfv_h)_{\Omega}
	+ \sum_{T\in \mcG_h}\left( (\lambda_{\bfu} h^ {-1}\uB,\bfv_h)_{T \cap \partial\Omega}+(\uB,\mu\nabla\bfv_h\cdot\bfn)_{T \cap \partial\Omega} \right),
	\label{eq:f0}
\end{align}
where $\lambda_t$ and $\lambda_{\bfu}$ are penalty parameters and chosen to be positive constants and we use the mixed ghost penalty stabilization terms developed in \cite{FraHaNilZa22}
\begin{align}
  s_{a}(\bfu_h,\bfv_h) &:=\sum_{F \in \mcF_h}\sum_{j=0}^{k+1}
	\tau_{a} h^{2j+\gamma_a}  (\jump{\nDj \bfu_{h} }, \jump{\nDj \bfv_{h}})_{F}, \label{eq:stab-u}\\
	s_b(\bfu_h,q_h) &:=\sum_{F \in \mcF_h}\sum_{j=0}^{\hat{k}}
	\tau_b h^{2j+1}  ( \jump{D^j (\dive \bfu_{h})}, \jump{D^j q_{h}} )_{F}. \label{eq:stab-b}
\end{align}
Note that $\jump{D^j q_{h}}$ denotes the jump in the generalised derivative of order $j$ across the edge $F$. Similarly, $\jump{D^j_{\bfn_F} \bfv_{h}}$ denotes the jump in the normal component of the generalised derivative, with $\jump{D^0_{\bfn_F} \bfv_{h}} = \jump{\bfv_{h}}$.  The stabilization parameters $\tau_{a}>0, \tau_{b}>0,$ and we choose $\gamma_a=-1$ for this method.

To derive this method we multiplied the Stokes problem with test functions $(\bfv_h, q_h)\in \Vk^0 \times \Qk$ (we clarify why we choose these test spaces and not for example $\Vk \times \Qk^0$ in Section \ref{sec:lagmult}), applied integration by parts elementwise, and added terms in the weak form to enforce tangential continuity and the Dirichlet boundary condition weakly following Nitsche's method. We have chosen to make the bilinear form $a(\bfu_h,\bfv_h)$ anti-symmetric. The advantage of the non-symmetric version of the interior penalty method and Nitsche's method is that the choice of the penalty parameter is less sensitive than in the symmetric version, see e.g. \cite{FreSte95, BoiBur16}. Important to note is that the presented method is non-symmetric due to the difference between $\Bf$ and $\Bf_0$. We use $\Bf_0$ since we do not want to perturb the divergence condition, see Theorem~\ref{thm:divfree}. Enforcing the boundary conditions with a Lagrange multiplier (instead of Nitsche's method), as in \cite{BurHanLarst23}, would lead to a symmetric method. See also Method C2 in Section~\ref{sec:method_vort}.

Since the proposed scheme is based on the boundary-fitted scheme of \cite{WangHdiv2007,WangRobust2009}, the convergence order of the pressure variable is expected to be optimal with respect to the larger pair $\BDM_{k+1}\times Q_k$ but not with respect to the pair $\RT_{k}\times Q_k$. For the Raviart-Thomas space the method is expected to achieve a convergence in the pressure of order $\kp$. See the a priori analysis in \cite{WangHdiv2007,VolkerDivConstr2017} for details.

\section{A conforming mixed vorticity method - Method C}\label{sec:method_vort}
We start by rewriting the velocity-pressure formulation \eqref{eq:momentum-balance}-\eqref{eq:BCu} of the Stokes equations by introducing the vorticity. We then present two conforming numerical discretizations of the vorticity-velocity-pressure formulation of the Stokes equations. In the first discretization a penalty parameter is utilized to enforce the boundary condition and we refer to this method as Method C1 while in the second discretization, Method C2 a Lagrange multiplier is used.

We introduce a new variable, the vorticity
\begin{equation}
  \omega := \mu\crl\bfu.
\end{equation}
We make use of the incompressibility $\dive \bfu=0$ to write $-\dive(\mu\nabla\bfu) = \crl(\mu\crl \bfu) -\nabla(\mu\dive\bfu)=\crl\omega.$ The Stokes problem \eqref{eq:momentum-balance}-\eqref{eq:BCu} becomes \cite{Dub02,Hanot2021}
\begin{align}\label{eq:curlStokes}
	&\left\{
	\begin{aligned}
		\mu^{-1}\omega - \crl\bfu &= 0,&&\text{ in } \Omega, \\
		\crl\omega+\nabla p &= \bff,&&\text{ in } \Omega, \\
		\dive\bfu & =  0,&&\text{ in } \Omega, \\
		\bfu & = \pmb{g},&&\text{ on } \partial\Omega.
	\end{aligned}
	\right.
\end{align}
Note that the following exact sequence holds
\begin{align}
	0 &\hookrightarrow H^1(\Omega) \overset{\crl}{\to} \bfH^{\dive}(\Omega) \overset{\dive}{\to} L^2(\Omega) \to 0 \label{eq:2Dcomplex}
\end{align}
That the sequence above is exact means that the kernel of a given operator in the sequence above is the image of the preceding operator in the same sequence. In particular, the divergence of the curl of a scalar field in two space dimensions is $0$ by virtue of the commutativity of $x-$ and $y$-derivatives. 

Multiplying the second equation of \eqref{eq:curlStokes} with a test function $\bfv\in\bfH^{\dive}(\Omega)$ and integrating by parts we get 
\begin{equation}
	\int_{\Omega}\crl\omega\cdot\bfv + \int_{\partial\Omega}p\bfv\cdot\bfn - \int_{\Omega}p\dive \bfv = \int_{\Omega}\bff\cdot\bfv. \label{eq:curlw}
\end{equation}
Multiplying the first equation of \eqref{eq:curlStokes} with $\phi\in\mathcal{H}(\Omega):=H^1(\Omega)$, and integrating by parts we have 
\begin{equation}
	\int_{\Omega}\mu^{-1}\omega\phi
	 - \int_{\Omega} \bfu\cdot\crl\phi = \int_{\partial\Omega}\bfg\cdot \bft \phi.
\end{equation}
Since the term $-(\bfu\cdot\bft,\phi)_{\partial\Omega}$ appears naturally in the integration by parts, we handle the boundary condition \eqref{eq:BCu} such that the tangential component $\bfu\cdot\bft$ is treated as a natural boundary condition. The normal component $\bfu\cdot\bfn$ can be enforced strongly or weakly. We define the space $\bfH^{\dive}_{\bfg}(\Omega)$ as the space of vector fields in $\bfH^{\dive}(\Omega)$ (see \eqref{eq:Hdiv}) that satisfy (in the distributional sense) the boundary condition $\bfu\cdot\bfn=\bfg\cdot\bfn$ on $\partial\Omega$. 
The weak formulation of \eqref{eq:curlStokes} when  $\bfu\cdot\bfn$ is imposed strongly, i.e. in the space, reads as follows. Find $(\omega,\bfu,p)\in \mathcal{H}(\Omega)\times\bfH^{\dive}_{\bfg}(\Omega)\times L^2(\Omega)$ such that
\begin{alignat}{2}
	(\mu^{-1}\omega,\phi)_{\Omega} - (\bfu,\crl\phi)_{\Omega} &= (\bfg\cdot\bft,\phi)_{\partial\Omega}, \ &&\forall \phi\in \mathcal{H}(\Omega), \\
	(\crl\omega,\bfv)_{\Omega} + (p,\bfv\cdot\bfn)_{\partial\Omega} - (p,\dive\bfv)_{\Omega} &= (\bff,\bfv)_{\Omega},\ &&\forall \bfv\in\bfH^{\dive}_0(\Omega), \\
	(\dive\bfu,q) &= 0, \ &&\forall q\in L^2(\Omega).
\end{alignat}
In the next section we develop two unfitted finite element schemes where $\bfu\cdot\bfn$ is imposed weakly, either via a penalty parameter or via a Lagrange multiplier.

\subsection{The finite element method}
We choose conforming finite element subspaces $\Wk\subset \mathcal{H}(\Omega), \Vk\subset\bfH^{\dive}(\Omdh)$ with $\bfV_k=\RT_k$, and $\Qkk\subset L^2(\Omdh)$. These are the piecewise continuous polynomial functions of degree less than or equal to $k+1$, Raviart-Thomas elements, and piecewise discontinuous polynomial functions of degree less than or equal to $k$. 

Here we present two discretizations. The first, which we refer to as Method C1, uses penalty and the second, which we refer to as Method C2, uses a Lagrange multiplier in order to enforce Dirichlet boundary conditions on unfitted boundaries. See Section \ref{sec:mC_3D} for details on the discretization with penalty (Method C1) in 3D. 
Method C2 requires an extra variable which needs to be stabilized, but results in a symmetric formulation and in our numerical experiments the method performed better for Dirichlet boundary conditions. 

\subsubsection{Method C1: enforcing the essential boundary condition with penalty}
Find $(\omega_h,\bfu_h,p_h)\in \Wk\times \Vk\times \Qkk^0$ such that 
\begin{alignat}{2} 
	m(\omega_h,\phi_h) - c(\phi_h,\bfu_h) &= (\bfg\cdot\bft,\phi_h)_{\partial\Omega}, \ &&\forall \phi_h\in \Wk, \label{eq:mC_curleq} \\
	\check{a}(\bfu_h,\bfv_h) + c(\omega_h,\bfv_h) - \Bf(\bfv_h,p_h) &= (\bff,\bfv_h)_{\Omega} + (\lambda_{\bfu_n}h^{-1}\bfg\cdot\bfn,\bfv_h\cdot\bfn)_{\partial\Omega},\ \quad &&\forall \bfv_h\in\Vk^0, \label{eq:mC_momentumbalance} \\
	\Bf_0(\bfu_h,q_h) &= 0, \ &&\forall q_h\in \Qkk. \label{eq:mC_conservation} 
\end{alignat}
Here $\Bf$ and $\Bf_0$ are as in \eqref{eq:Bf} and \eqref{eq:Bf0}, and 
\begin{align}
	m(\omega_h,\phi_h) &:=(\mu^{-1}\omega,\phi)_{\Omega}  + s_m(\omega_h,\phi_h), \label{eq:m} \\
	s_m(\omega_h,\phi_h) &:=\sum_{F \in \mcF_h}\sum_{j=1}^{k}
	\tau_m h^{2j+1}  ( \jump{\nDj \omega_h}, \jump{\nDj \phi_h} )_{F}, \label{eq:stab-w} \\
	c(\phi_h,\bfu_h) &:= (\crl\phi_h, \bfu_h)_{\Omega} + s_c(\phi_h,\bfu_h),\\
        s_c(\phi_h,\bfu_h) &:=\sum_{F \in \mcF_h}\sum_{j=0}^{k}
        \tau_c h^{2j+1}  ( \jump{\nDj (\crl \phi_h)}, \jump{\nDj \bfu_h} )_{F}. \label{eq:stab-m} \\
	\check{a}(\bfu,\bfv) &:=\sum_{T\in\mcG_h} (\lambda_{\bfu_n}h^{-1}\bfu\cdot\bfn,\bfv\cdot\bfn)_{T\cap \partial\Omega} +s_a(\bfu_h,\bfv_h), \label{eq:MCpenalty}
\end{align}
where $\lambda_{\bfu_n}$ is a positive sufficiently large constant and $s_a$ is as in equation \eqref{eq:stab-u} but with $\gamma_a=1$. 
The stabilization terms are chosen so that we retain the structure of the original problem. The $s_b$ stabilization is always included hence $\tau_b$ is always a positive constant. The parameters $\tau_m$, $\tau_c$, and $\tau_a$ are either positive or zero. The choice of the three parameters gives different numerical methods which we label as $(\tau_m,\tau_c,\tau_a)$ for different values of $\tau_m,\tau_c,\tau_a\geq 0$, we shall focus on $(\tau_m,0,\tau_a)$ and $(0,\tau_c,0)$.

\subsubsection{Method C2:  enforcing the essential boundary condition with a Lagrange multiplier}
We now present a scheme where we prescribe the normal component $\bfu \cdot \bfn$  at the boundary using a Lagrange multiplier method. We introduce a new variable residing in the Lagrange multiplier space: 
\begin{align}
	Q_{\tilde{k},h}^{\partial\Omega} := \bigoplus_{T\in\mcG_h} Q_{\tilde{k},h}(T).
\end{align}
We choose $\tilde{k}\geq k\geq0$ and formulate a variant of Method C where $\bfu \cdot \bfn$ on unfitted parts of the boundary are imposed via a Lagrange multiplier variable $\xi_h$ as follows: Find $(\omega_h,\bfu_h,p_h,\xi_h)\in \Wk\times \Vk\times \Qkk^0\times\Qkkg$ such that 
\begin{alignat}{2} 
	m(\omega_h,\phi_h) - c(\phi_h,\bfu_h) &= (\bfg\cdot\bft,\phi_h)_{\partial\Omega}, \ &&\forall \phi_h\in \Wk, \label{eq:MClagmulteq1} \\
	c(\omega_h,\bfv_h) - \Bf_0(\bfv_h,p_h) + (\xi_h,\bfv_h\cdot\bfn)_{\partial\Omega} &= (\bff,\bfv_h)_{\Omega},\ \quad &&\forall \bfv_h\in\Vk, \\
	\Bf_0(\bfu_h,q_h) &= 0, \ &&\forall q_h\in \Qkk^0, \label{eq:MClagmult_divcond} \\
	(\bfu_h\cdot\bfn,\chi_h)_{\partial\Omega}-s(\xi_h,\chi_h) &= (\bfg\cdot\bfn,\chi_h)_{\partial\Omega}, \ &&\forall \chi_h\in \Qkkg. \label{eq:MClagmult}
\end{alignat}
Here $s(\cdot,\cdot)$ is the following stabilization term~\cite{BurHanLarMas17} 
\begin{align}\label{eq:stablag}
	s(\xi_h,\chi_h) := 
	\sum_{F \in \mcF_{h,\partial \Omega}} \sum_{j=0}^{\tilde{k}}
	\tau_{\xi} h^{2j}  ( \jump{\nDj \xi_h}, \jump{\nDj \chi_h} )_{F},
\end{align}
where $\mcF_{h,\partial\Omega}$ is the set of interior edges in $\mcG_h$, see Figure \ref{fig:static active mesh}, and $\tau_{\xi}$ is a positive constant. A Lagrange multiplier method is used in \cite{BurHanLarst23} to prescribe $\bfu|_{\partial \Omega}=\bfg$ with $\tilde{k}=0$. 
We show by numerical experiments in Section \ref{sec:numex} that choosing $\tilde{k}>k$ may improve the pressure robustness of our discretization. We also refer to \cite{Burman2020Dirichlet} for the Lagrange multiplier method. For higher order elements than linear, i.e., $k>1$, other stabilization terms than the ghost penalty terms in \eqref{eq:stablag} need to be included, see~\cite{Olshans2017Trace, LarZah19}.

The method \eqref{eq:MClagmulteq1}-\eqref{eq:MClagmult} is an alternative formulation to  \eqref{eq:mC_curleq}-\eqref{eq:mC_conservation}. Note that in this version of Method C when $\bfu_h \cdot \bfn|_{\partial \Omega}=\bfg \cdot \bfn$ is imposed via a Lagrange multiplier method on $\partial\Omega$, we only have $\Bf_0$ (not $\Bf$), the space $\Vk^0$ is not needed,  and $p_h, q_h\in \Qkk^0$. In Section~\ref{sec:lagmult} we discuss why we need the space $\Vk^0$ when the Dirichlet boundary condition is imposed weakly via penalty but not when we use the Lagrange multiplier method.  The numerical examples in Section \ref{sec:numex} show that version \eqref{eq:MClagmulteq1}-\eqref{eq:MClagmult} of Method C performs better.

\begin{remark}[Mixed stabilization terms]
We note that the stabilization terms $s_b$ and $s_c$ that we use in the bilinear forms $B_0(\cdot,\cdot)$, $B(\cdot,\cdot)$, and $c(\cdot,\cdot)$ are not symmetric. In the design of these terms, we followed the discrete version of the exact sequence \eqref{eq:2Dcomplex}, making use of the property that $\dive (\Vk)\subset \Qkk$ and $\crl (W_{k+1,h})\subset \Vk$.
\end{remark}

\begin{remark}[Boundary conditions] \label{rem:otherbc}
The discrete problem with Dirichlet boundary conditions \eqref{eq:BCu} fails to fall into an appropriate de Rham complex. 
In \cite{Hanot2021}, this is suggested to impact the theoretical convergence order of the pressure variable. We observe this in our numerical experiments in Section \ref{sec:numex} only when the Dirichlet boundary condition is imposed via a penalty parameter as in the discretization \eqref{eq:mC_curleq}-\eqref{eq:mC_conservation} but not when we use a Lagrange multiplier as in \eqref{eq:MClagmulteq1}-\eqref{eq:MClagmult}. 

According to \cite{Hanot2021}, the convergence order is provably optimal for fitted FEM with respect to either of the following boundary conditions 
\begin{align}
	&\omega\bft = \omega_0\bft,\ \bfu\cdot\bfn = u_0, \text{ or } \label{eq:altBC0}\\
	&\bfu\cdot\bft = u_0,\ p=p_0. \label{eq:altBC}
\end{align}
We show by numerical experiments in Section \ref{sec:numex} that for the boundary condition \eqref{eq:altBC} it is better to choose $(0,\tau_c,0)$. In this case we don't need to prescribe $\bfu \cdot \bfn$ and the most natural formulations reads:
Find $(\omega_h,\bfu_h,p_h)\in \Wk\times \Vk\times \Qkk$ such that 
\begin{alignat}{2}
	(\mu^{-1}\omega_h,\phi_h)_{\Omega}  - c(\phi_h,\bfu_h) &= (u_0,\phi_h)_{\partial\Omega}, \ &&\forall \phi_h\in \Wk, \label{eq:mC_variant10}\\
	c(\omega_h,\bfv_h) - \Bf_0(\bfv_h,p_h) &= (\bff,\bfv_h)_{\Omega} - (p_0,\bfv_h\cdot\bfn)_{\partial\Omega},\ \quad &&\forall \bfv_h\in\Vk, \\
	\Bf_0(\bfu_h,q_h) &= 0, \ &&\forall q_h\in \Qkk.\label{eq:mC_variant01}
\end{alignat}
\end{remark}

\section{The divergence-free property} \label{sec: the divfree prop}
We now show that all the cut finite element discretizations presented in Section \ref{sec:method_wang}-\ref{sec:method_vort} produce pointwise divergence-free approximations of solenoidal velocity fields, as fitted standard FEM does with the considered element pairs. We follow the proof in \cite{FraHaNilZa22}. Define the standard ghost penalty term for the pressure:
\begin{align}
  s_p(p_h,q_h) &:=\sum_{F \in \mcF_h}\sum_{j=0}^{\hat{k}}
	\tau_b h^{2j+1}  ( \jump{D^j p_{h}}, \jump{D^j q_{h}} )_{F}. \label{eq:stab-p}
\end{align}



\begin{thm} [\textbf{The divergence-free property}] \label{thm:divfree} 
  Assume $\bfu_h \in\Vk$ satisfies $\Bf_0(\bfu_h,q_h) = 0$, for all $q_h\in\Qk$ then $\dive\bfu_{h} =0$.
\end{thm} 
\begin{proof}
We have that $\bfu_h\in\Vk$ satisfies 
	\[0 = \Bf_0(\bfu_h,q_h) = \int_{\Omega} \dive\bfu_h q_h + \sum_{F \in \mcF_h}\sum_{j=0}^{\hat{k}}
	\tau_b h^{2j+1}  ( \jump{D^j (\dive \bfu_{h})}, \jump{D^j q_{h}} )_{F},  \quad \forall  q_h\in\Qk.  \]
	We may choose $q_{h}=\dive\bfu_{h}$ since $\dive \Vk \subset \Qk$. Then we have 
	\begin{equation}
		0=\|\dive\bfu_{h} \|_{\Omega}^2+ s_p(\dive\bfu_{h}, \dive\bfu_{h}) 
		\gtrsim \|\dive\bfu_{h} \|_{\Omdh}^2\geq 0,
	\end{equation}
	where we used that for any $q_{h} \in\Qk$, $\|q_{h} \|_{\Omega}^2 + s_p(q_{h}, q_{h}) \gtrsim \| q_{h} \|^2_{\Omdh}$. See Lemma 3.8 in \cite{HaLaZa14} or Lemma 5.1 in  \cite{MaLaLoRo14}. We can thus conclude that $\dive\bfu_{h} =0$ in $\Omdh$. 
\end{proof}

We see from this theorem that it is important to not disturb the condition $\Bf_0(\bfu_h,q_h) =0$. Next we show that prescribing the mean in the pressure test space will perturb the incompressibility condition when we impose boundary conditions by Nitsche's method or penalty, and motivate our choice of using the spaces $(\Vk^0, \Qk)$ as test spaces. 

\subsection{Justification and implementation of $\Vk^0$ }\label{sec:lagmult}
We know that the pressure is only determined up to a constant when Dirichlet boundary conditions are imposed.  Determining this unknown constant is usually handled by modifying the space $\Qk$ by setting the mean value:
\begin{equation}\label{eq:Qk0}
	\Qk^0=\left\{q_h \in \bigoplus_{T \in \mcT_{h}} Q_{\kp}(T): \ \int_{\Omega} q_h=0 \right\}.
\end{equation}
The space \eqref{eq:Qk0} is difficult to work with numerically (see \cite[Section 5.2.5]{BofBreFor13}), so equivalently one adds a Lagrange multiplier $\lagp\in\RR$ such that $\lagp \int_{\Omega} q_h$ is added to the weak formulation. 

The resulting discrete linear system for Method NC would be: Find solutions $(\bfu_h,p_h,\lagp)\in \Vk\times\Qk\times\RR$ to
\begin{alignat}{2} 
	\Af(\pmb{u}_h,\pmb{v}_h) - \Bf(\pmb{v}_h,p_h) &= \FF (\pmb{v}_h),  \ &&\text{ for all } \pmb{v}_h\in \Vk,  \label{eq:Astandard_lagr} \\
	\Bf_0(\pmb{u}_h,q_h) + (\lagp,q_h)_{\Omega} &= 0,  &&\text{ for all } q_h\in \Qk, \label{eq:Bstandard_lagr} \\
	(\sigma,p_h)_{\Omega} &= 0, &&\sigma\in\RR. \label{eq:standard_lagr}
\end{alignat}

From \eqref{eq:Bstandard_lagr} we have for $\lagp \in \RR$ and $q_h=1$ that  
\begin{align}
0 = \int_{\partial\Omega} \bfu_h\cdot\bfn + \int_{\Omega} \lagp &\implies |\Omega| \lagp = -\int_{\partial\Omega} \bfu_h \cdot\bfn, 
\end{align}
where we have used integration by parts. Imposing the boundary condition $\bfu_h\cdot\bfn |_{\partial\Omega} = \bfg\cdot\bfn$ strongly,
with $\int_{\partial\Omega} \bfg\cdot\bfn=0$ as a compatibility criterion on the datum $\bfg$, implies $\lagp = 0$ for a well-posed problem. By choosing $\chi_h=1$ in~\eqref{eq:MClagmult} we see that when $\bfu_h\cdot\bfn |_{\partial\Omega} = \bfg\cdot\bfn$ is imposed weakly via a Lagrange multiplier method the condition 
\begin{align}
	\int_{\partial\Omega} \bfu_h\cdot\bfn=\int_{\partial\Omega} \bfg\cdot\bfn=0
\end{align}
is satisfied.  However, prescribing the normal component $\bfu \cdot \bfn$ weakly via Nitsche's method or penalty yields $\int_{\partial\Omega} \bfu_h\cdot\bfn\approx\int_{\partial\Omega} \bfg\cdot\bfn$. It is easy to see that the approximate solution $\bfu_h$ will then satisfy
\begin{equation}\label{eq:modded_divcond}
	\dive\bfu_h = -\lagp.
\end{equation}
In the numerical experiments we see that $\lagp$ is typically not $0$ in this case and the divergence of the approximate velocity is incorrect by exactly the value of $\lagp$, see Section \ref{sec:num_lagr}. Thus, to remedy this issue we introduced the space $\Vk^0$.  

We avoid the Lagrange multiplier in \eqref{eq:Bstandard_lagr} by pairing $\int_{\Omega} p_h=0$ with a linear equation on the DOFs of the test space $\Vk$; namely $\int_{\partial\Omega} \bfv_h\cdot\bfn = 0$. This condition is enforced by a Lagrange multiplier in the momentum equation instead. In this way we avoid perturbing the incompressibility condition. We have already written our methods with this in mind by using $\Vk^0$ as test space, see equation \eqref{eq:mNC_discreteStokes}-\eqref{eq:mNC_conservation} for Method NC. We note that since the incompressibility condition holds pointwise, i.e., $\dive \bfu_h=0$, we have $\int_{\Omega} \dive \bfu_h=\int_{\partial\Omega} \bfu_h\cdot\bfn = 0$. Hence $\bfu_h \in \Vk^0$. 

\subsubsection{Implementation of the space $\Vk^0$ and $\Qk^0$}\label{sec:mixedL}
We now show how the resulting linear systems associated to Method NC and Method C become.
Method NC reads: Find $(\bfu_h,p_h,\lagp)\in \Vk\times\Qk\times\RR$ for which \begin{alignat}{2}\label{eq:mNC_lagr2}
	\Af(\pmb{u}_h,\pmb{v}_h) - \Bf(\pmb{v}_h,p_h) + (\lagp,\bfv_h\cdot\bfn)_{\partial\Omega} &= \FF (\pmb{v}_h), \ &&\text{ for all } \pmb{v}_h\in \Vk, \\
	\Bf_0(\pmb{u}_h,q_h) &= 0, &&\text{ for all } q_h\in \Qk, \nonumber \\
	(\sigma,p_h)_{\Omega} &= 0, &&\sigma\in\RR. \nonumber
\end{alignat}
Expanding trial and test functions in their finite element basis, let $\mathbf{\hat{f}}$ be the data vector corresponding to $\FF$ and let $\begin{bmatrix}
	\mathbf{\hat{u}} \\
	\mathbf{\hat{p}} 
\end{bmatrix}$ denote the unknown degree of freedom (DOF) values. The linear system of equations associated with \eqref{eq:mNC_lagr2} can be written as
\begin{equation}\label{eq:structure1}
	\begin{bmatrix}
		\mathbf{A} & -\mathbf{B}^T & (1,\bfv_h\cdot\bfn)_{\partial\Omega} \\
		\mathbf{B}_0 & \mathbf{0} & \mathbf{0} \\
		0 & (1,\sigma)_{\Omega} & 0
	\end{bmatrix} 
	\begin{bmatrix}
		\mathbf{\hat{u}} \\
		\mathbf{\hat{p}} \\
		\lagp
	\end{bmatrix} 
	=\begin{bmatrix}
		\mathbf{\hat{f}} \\
		\mathbf{0} \\
		0
	\end{bmatrix}.
\end{equation}

Method C reads: Find $(\omega_h,\bfu_h,p_h,\lagp)\in \Wk\times\Vk\times\Qkk\times\RR$ for which
\begin{alignat}{2}\label{eq:mC_lagr2}
	m(\omega_h,\phi_h) - c(\phi_h,\bfu_h) &= (\bfg\cdot\bft,\phi_h)_{\partial\Omega}, \ &&\forall \phi_h\in \Wk, \\
	\check{a}(\bfu_h,\bfv_h) + c(\omega_h,\bfv_h) - \Bf(\bfv_h,p_h) + (\lagp,\bfv_h\cdot\bfn)_{\partial\Omega} &= (\bff,\bfv_h)_{\Omega} + (\lagp_{\bfu}h^{-1}\bfg\cdot\bfn,\bfv_h\cdot\bfn)_{\partial\Omega},\ \quad &&\forall \bfv_h\in\Vk, \nonumber \\
	\Bf_0(\bfu_h,q_h) &= 0, \ &&\forall q_h\in \Qkk, \nonumber \\
	(\sigma,p_h)_{\Omega} &= 0,\ &&\quad\ \sigma\in\RR. \nonumber
\end{alignat}
The linear system of equations associated with \eqref{eq:mC_lagr2} can be written as
\begin{align}\label{eq:structure2}
	\begin{bmatrix}
		\mathbf{M} & -\mathbf{C}^T & \mathbf{0} & \mathbf{0} \\
		\mathbf{C} & \check{\mathbf{A}} & -\mathbf{B}^T & (1,\bfv_h\cdot\bfn)_{\partial\Omega} \\
		\mathbf{0} & \mathbf{B_0} & \mathbf{0} & \mathbf{0} \\
		0 & 0 & (1,\sigma)_{\Omega} & 0		
	\end{bmatrix}
	\begin{bmatrix}
		\hat{\omega} \\
		\mathbf{\hat{u}} \\
		\mathbf{\hat{p}} \\
		\lagp
	\end{bmatrix} 
	=\begin{bmatrix}
		\mathbf{\hat{g}} \\
		\mathbf{\hat{f}} \\
		\mathbf{0} \\
		0
	\end{bmatrix}.
\end{align}

Note that the divergence condition isn't disturbed in either method, so nothing changes in the proof of Theorem \ref{thm:divfree}.

\begin{remark} \textbf{(On lowest order discretizations)}
	Even though functions in $\Vk^0$ do not satisfy $\int_{\partial\Omega} p_h\bfv_h\cdot\bfn = 0$ for all $p_h\in\Qk^0$, our numerical studies indicate that the lowest order methods of \eqref{eq:mNC_lagr2} and \eqref{eq:mC_lagr2} perform well when modified such that $\Bf(\bfv_h,p_h) \mapsto \Bf_0(\bfv_h,p_h)$. This can be beneficial to do, making the system matrices more symmetric and sparse.
\end{remark}

\section{Efficient stabilization} \label{sec:equivalent_stab}
To obtain stability and control of the condition number of the resulting linear systems independently of how the boundary $\partial \Omega$ is positioned relative to the background mesh we have utilized stabilization terms. To obtain less fill in and to have a method that is less sensitive to the stabilization constants one can add these ghost penalty terms only where they are necessary. This can be done by partitioning the unfitted mesh into macro-elements that have a large intersection with $\Omega$ and apply stabilization only at interior edges of these macro-elements and never on the interface between two macro-elements, see \cite{LaZa21}. For higher order elements than linear it is easier to implement the version of the ghost penalty terms proposed in \cite{Prauss} and hence avoid jump terms of derivatives. In our implementation of the proposed methods we have used these two modifications of the stabilization terms. 

\subsection{Macroelement stabilization}
To partition the mesh into macroelements that have a large intersection with $\Omega$ each element in the active mesh $\mcT_{h}$ is first classified either as having a large intersection with the domain $\Omega$ or a small intersection. We say that an element $T \in \mcT_{h}$ has a large $\Omega$-intersection if
\begin{equation}\label{eq:largeel}
	\delta  \leq \frac{|T \cap \Omega|}{|T|},
\end{equation}
where $0<\delta\leq 1$ is a constant which is chosen independent of the element and the mesh parameter. We collect all elements with a large intersection in 
\begin{equation}
	\mcTL=\left\{T \in \mcT_{h} : |T \cap \Omega| \geq \delta |T| \right\}.
\end{equation}
Using this classification we create a macroelement partition $\mcM_{h}$ of $\Omega_{\mcT_{h}}$:
\begin{itemize}
	\item To each $T_L\in\mcTL$ we associate a macroelement mesh $\mcT_{h}(T_L)$ containing $T_L$ and possibly adjacent elements that are in $\mcT_{h}\setminus \mcTL$, i.e., elements classified as having a small intersection with $\Omega$ and  are connected to $T_L$ via a bounded number of internal edges. 
	\item Each element $T \in \mcT_{h}$ belongs to precisely one macroelement mesh $\mcT_{h}(T_L)$.
	\item Each macroelement $M \in \mcM_{h}$ is the union of elements in $\mcT_{h}(T_L)$, i.e., 
	\begin{equation}
		M = \bigcup_{T \in \mcT_{h}(T_L)} T.
	\end{equation}
\end{itemize}
We denote by $\mcFM$ the set consisting of interior edges of $M\in\mcM_{h}$. Note that $\mcFM$ is empty when $T_L$ is the only element in $\mcTT$.
See Figure~\ref{fig:macro_mesh} for an illustration of a macroelement partitioning. Edges in $\mcFM$ are illustrated by dotted lines and black thick lines illustrate the boundary of macroelements consisting of more than one element.  We follow Algorithm 1 in \cite{LaZa21} when we construct the macroelement partition in our implementation. 
\begin{figure}[ht!]
	\centering
	\includegraphics[scale=0.2]{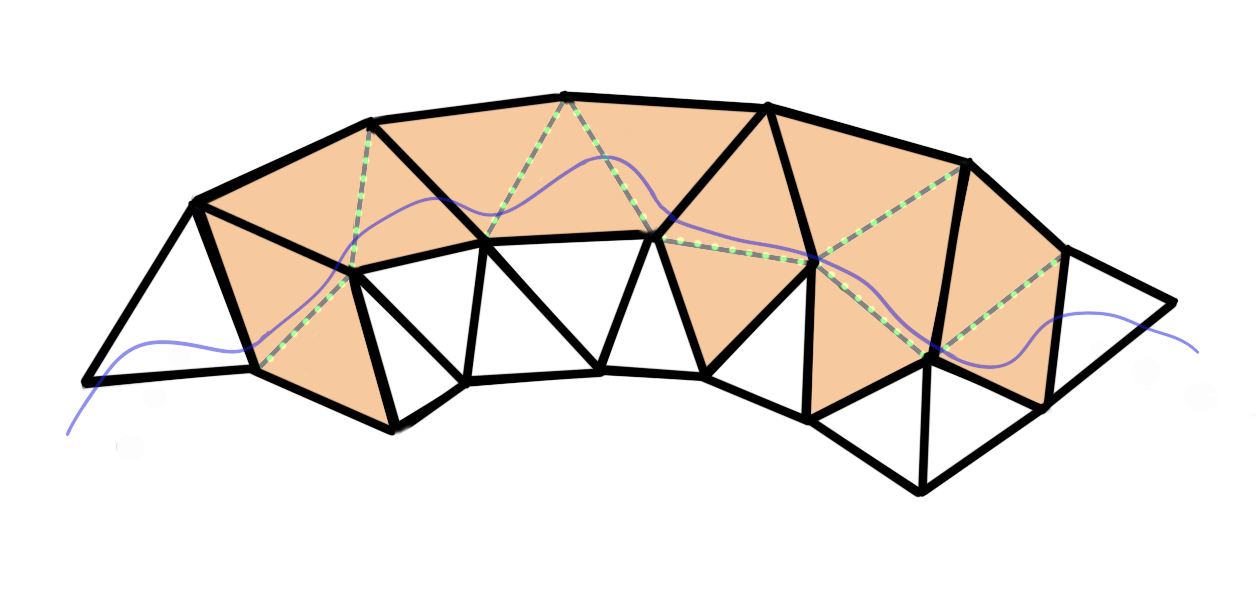}
	\caption{Illustration of a macroelement partition at the boundary $\partial\Omega.$ macroelements consisting of more than one triangle are colored orange, separated by thick edges. Given a macroelement $M\in \mcM_h$, edges $F\in \mcF_h(M)$ are illustrated by dotted lines.}  
	\label{fig:macro_mesh} 
\end{figure}

\subsection{Macro patch stabilization for high order elements}
For high order elements the framework suggested in~\cite{Prauss} for applying stabilization can be beneficial since it avoids the evaluation of kth-order derivatives of the basis functions but yields the same control as the standard ghost penalty terms including derivatives up to order $k$. We therefore apply our proposed mixed stabilization terms into that framework. 
For each face $F\in \mcF_h(M)$ we denote the patch consisting of the two elements sharing face F by $\mathcal{P}_F=T\cup T'$. 
The stabilization terms used in our implementation read:
\begin{align}
	s_a(\bfu_h,\bfv_h) &= \sum_{M \in \mcM_h} \sum_{F \in \mcFM}\sum_{j=0}^{k} \tau_a h^{1-\gamma_a}  ( [\nDj \bfu_h], [\nDj \bfv_h] )_{\mathcal{P}_F}, \label{eq:mstab-a}\\
	s_b(\bfu_h,q_h) &= \sum_{M \in \mcM_h} \sum_{F \in \mcFM}\sum_{j=0}^{k} \tau_b  ( [D^j (\dive \bfu_h)], [D^j q_h] )_{\mathcal{P}_F}, \label{eq:mstab-b}\\
	s_c(\tau_h,\bfu_h) &= \sum_{M \in \mcM_h} \sum_{F \in \mcFM}\sum_{j=0}^{k} \tau_c ( [\nDj (\crl \tau_h)], [\nDj \bfu_h] )_{\mathcal{P}_F},\label{eq:mstab-c}\\
	s_m(\omega_h,\tau_h) &= \sum_{M \in \mcM_h} \sum_{F \in \mcFM}\sum_{j=1}^{k} \tau_m ( [\nDj \omega_h], [\nDj \phi_h] )_{\mathcal{P}_F}, \label{eq:mstab-m} \\
	s_p(p_h,q_h) &= \sum_{M \in \mcM_h} \sum_{F \in \mcFM}\sum_{j=0}^{k} \tau_b ( [D^j p_h], [D^j q_h] )_{\mathcal{P}_F}. \label{eq:mstab-p}
\end{align}
Here, 
\begin{equation}
 [\bfu_h]=\bfu_T^e(\bfx)-\bfu_{T'}^e(\bfx), \ \text{for $\bfx\in \mathcal{P}_F$},
\end{equation}
where $\bfu_T$ is the restriction of $\bfu_h$ to element $T$ and $\bfu_T^e$ is the canonical extension of the polynomial  $\bfu_T$ so that it can be evaluated for all $\bfx \in \mathcal{P}_F$.

\section{Numerical Experiments}\label{sec:numex}
In this section we study the performance of the proposed methods. For short-hand we will in this section write $P_{k+1}$ for the continuous piecewise polynomial space $\Wk$, $\RT_{k}$ for $\Vk$ with $\bfV_k(T)=\RT_{k}(T)$,  $\BDM_{k}$ for $\Vk$ with $\bfV_k(T)=\BDM_{k}(T)$, $Q_{k}$ for the piecewise polynomial space $\Qkk$, and $Q^{\Gamma}_k$ for the space $\Qkk^{\Gamma}$. 
We highlight that incompressibility \eqref{eq:conservation} is preserved with the proposed stabilization terms. We illustrate that if Dirichlet boundary conditions are imposed weakly via Nitsche's method or with a penalty parameter, the standard approach of prescribing $\int_{\Omega} q_h=0$ in the test space $\Qk$ via a Lagrange multiplier perturbs the incompressibility condition, see also Section \ref{sec:lagmult}. 
Furthermore, considering two examples introduced in \cite{VolkerDivConstr2017}, we show that the proposed methods can be made pressure robust at the cost of increasing the condition number of the linear systems.

We use the macro patch stabilization \eqref{eq:mstab-a}-\eqref{eq:mstab-p}, and choose the macro stabilization parameter $\delta=1$, meaning that all cut elements are marked as small but stabilization is applied only between interior edges inside macroelements. 
We choose $\tau_{\bullet}=1$ in all the stabilization terms. We have tried also smaller constants than one. Typically, the error decreases (down to some limit) when decreasing $\tau$, while the condition number increases. Our choice is not optimal and no attempts have been done to tune these parameters. One can obtain smaller errors or better condition numbers in some of the examples by tuning these stabilization parameters. However, we do believe that the main conclusions will not change. 

We study the following methods.
\begin{itemize}
	\item Method NC: We use the pairs $\BDM_1\times Q_0$ and $\RT_1\times Q_1$, each respective pair unstabilized ($\tau_{\bullet}=0$) and stabilized. We chooce $\lambda_t=1$.  The method is not sensitive to the choice of the interior penalty parameter $\lambda_t$ since the bilinear form $a(\bfu_h,\bfv_h)$ is chosen to be anti-symmetric, see \eqref{eq:a}. We specify in each example how large we have chosen $\lambda_{\bfu}>0$. In the stabilization $s_a$ we take $\gamma_a=-1$.
	\begin{itemize} 
		\item \textit{Proposed} method: We use the stabilization terms $s_b(\cdot,\cdot)$, see~\eqref{eq:mstab-b}, and follow the formulation~\eqref{eq:mNC_discreteStokes}. We implement the space $\Vk^0$ for the velocity test space following formulation \eqref{eq:mNC_lagr2}.
		\item \textit{Standard} approach: We use the stabilization term $s_p(\cdot,\cdot)$, see \eqref{eq:mstab-p}, and $\Qk^0$ for the pressure test space, resulting in formulation \eqref{eq:standard_mNC}.
	\end{itemize}
	\item Method C1: We use the triples $P_1\times\RT_0\times Q_0$ and $P_2\times\RT_1\times Q_1$ for the approximation of (vorticity, velocity, pressure), unstabilized and stabilized. For each example $\lambda_{u_n}>0$ has to be chosen large enough. In the stabilization $s_a$ we take $\gamma_a=1$.
	\begin{itemize}   
        \item \textit{Proposed} method: We use the stabilization terms $s_b(\cdot,\cdot)$ and follow formulation \eqref{eq:mC_momentumbalance}. We implement the space $\Vk^0$ for the velocity test space following formulation~\eqref{eq:mC_lagr2}.
		\item \textit{Standard} approach: We use the stabilization term $s_p(\cdot,\cdot)$, and the pressure test space $\Qk^0$. This yields formulation \eqref{eq:standard_mC}.  
	\end{itemize}
	\item Method C2: We mostly use the quadruple $P_1\times\RT_0\times Q_0\times Q_0^{\Gamma}$ for the approximation of (vorticity, velocity, pressure, boundary pressure), unstabilized ($\tau_\bullet = 0$) and stabilized. Results using the quadruple $P_2\times\RT_1\times Q_1\times Q_1^{\Gamma}$ are included in Table \ref{tab:fourfield2}. Here we always use $\Qk^0$ for our trial and test spaces.
	\begin{itemize}
		\item \textit{Proposed} method: We use the stabilization terms $s_b(\cdot,\cdot)$. See formulation \eqref{eq:MClagmult}.
		\item \textit{Standard} approach: We use the stabilization term $s_p(\cdot,\cdot)$. See formulation \eqref{eq:standard_MClagmult}.
	\end{itemize}
\end{itemize}

We compute the condition number of the system matrix using MATLAB's \textit{condest} function, which computes a lower bound for the $1-$norm condition number. All linear systems are solved using a direct solver, either UMFPACK or MUMPS, and we state when we use which. The plots are organized so that red color means no stabilization is applied, while blue means stabilization is applied.

Note that in general we do not know the exact boundary $\partial\Omega$ but rather an approximation $(\partial\Omega)_h$. In the implementation of all the schemes, we use a piecewise linear approximation of $\partial\Omega$. Consequently, we have an approximation $\Omega_{h}$ of $\Omega$. 

The code we used for the computations is available at \cite{repo}.

\subsection{Example 1 - Homogeneous Dirichlet BC}
Consider the example from \cite{LiuNeiOls23}. The domain $\Omega$ is  the disk of radius $0.5$ centered at $(0.5,0.5)$, and the exact solution to the Stokes problem is 
\begin{equation}
	\bfu(x,y) = \begin{pmatrix}
		2((x-0.5)^2+(y-0.5)^2-0.25)(2y-1)\\
		-2((x-0.5)^2+(y-0.5)^2-0.25)(2x-1)
	\end{pmatrix},\ p(x,y) = 10(x^2-y^2)^2.
\end{equation}
We embed $\Omega\subset\Omega_0:=[0,1]^2$ and use a uniform background mesh. Hence the entire boundary is unfitted, meaning $\Gamma=\partial\Omega$.

The convergence rates and errors are comparable between the proposed methods and the standard approaches, so we focus on the proposed methods in Section \ref{sec:convcondex1}. We study the divergence error in Section \ref{sec:num_lagr}.

\subsubsection{Convergence and condition number scaling}\label{sec:convcondex1}
We show the $L^2$-error in the approximate pressure and the velocity for different mesh sizes in Figure \ref{fig:ex1_up}. The condition number of the linear systems are shown in Figure \ref{fig:ex1_cond}. 
We choose $\lambda_{\bfu}=4000$ for Method NC and $\lambda_{u_n}=4000$ for Method C1. These choices are not optimal and there is a freedom to choose each parameter differently, but the penalty parameters $\lambda_{\bfu}$ and $\lambda_{u_n}$ have to be chosen sufficiently large.


\begin{figure}[ht]
	\centering 
	\begin{subfigure}{0.4\textwidth}
		\includegraphics[width=\textwidth]{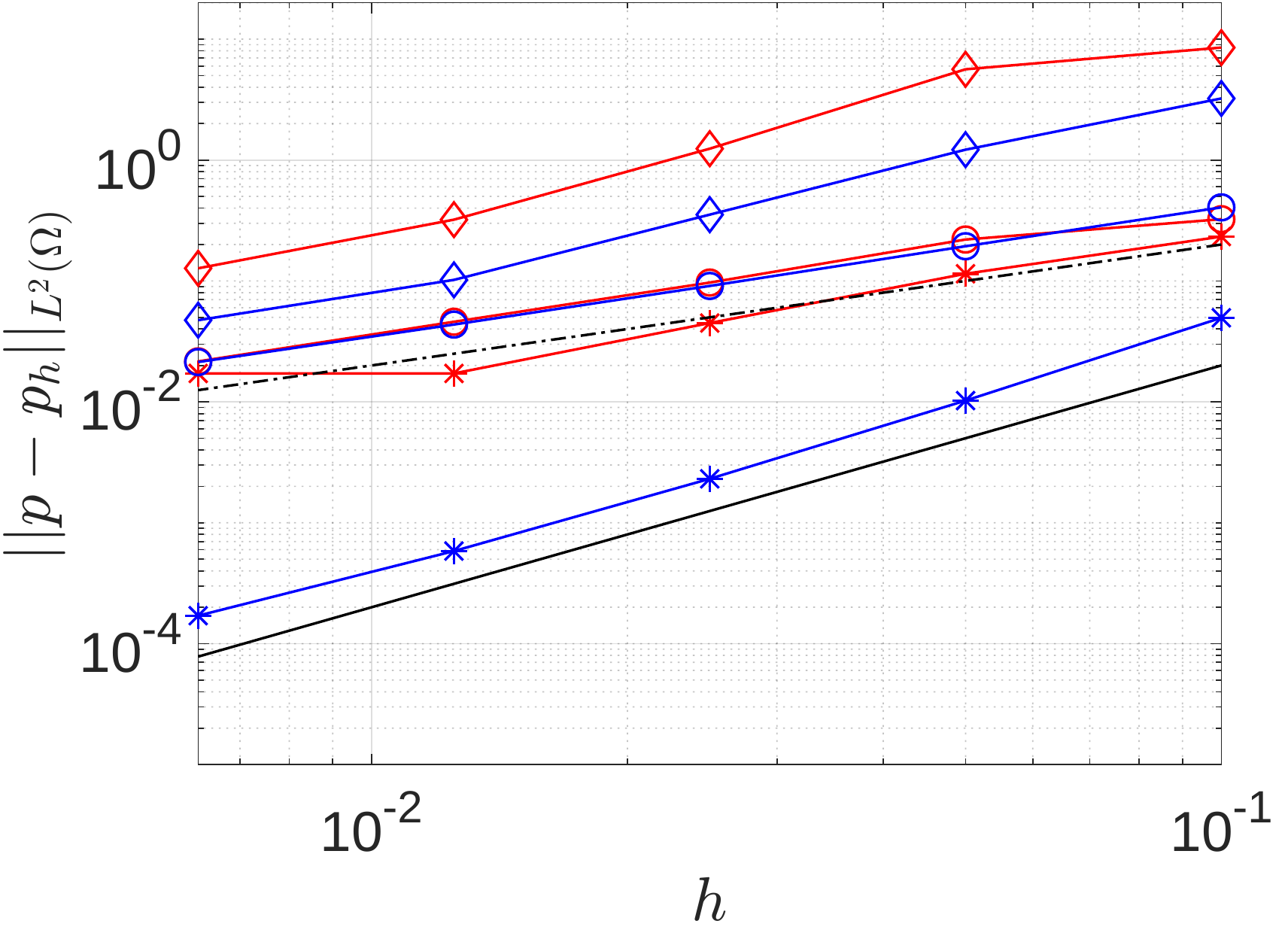}
	\end{subfigure}
	\quad
	\begin{subfigure}{0.4\textwidth}
		\includegraphics[width=\textwidth]{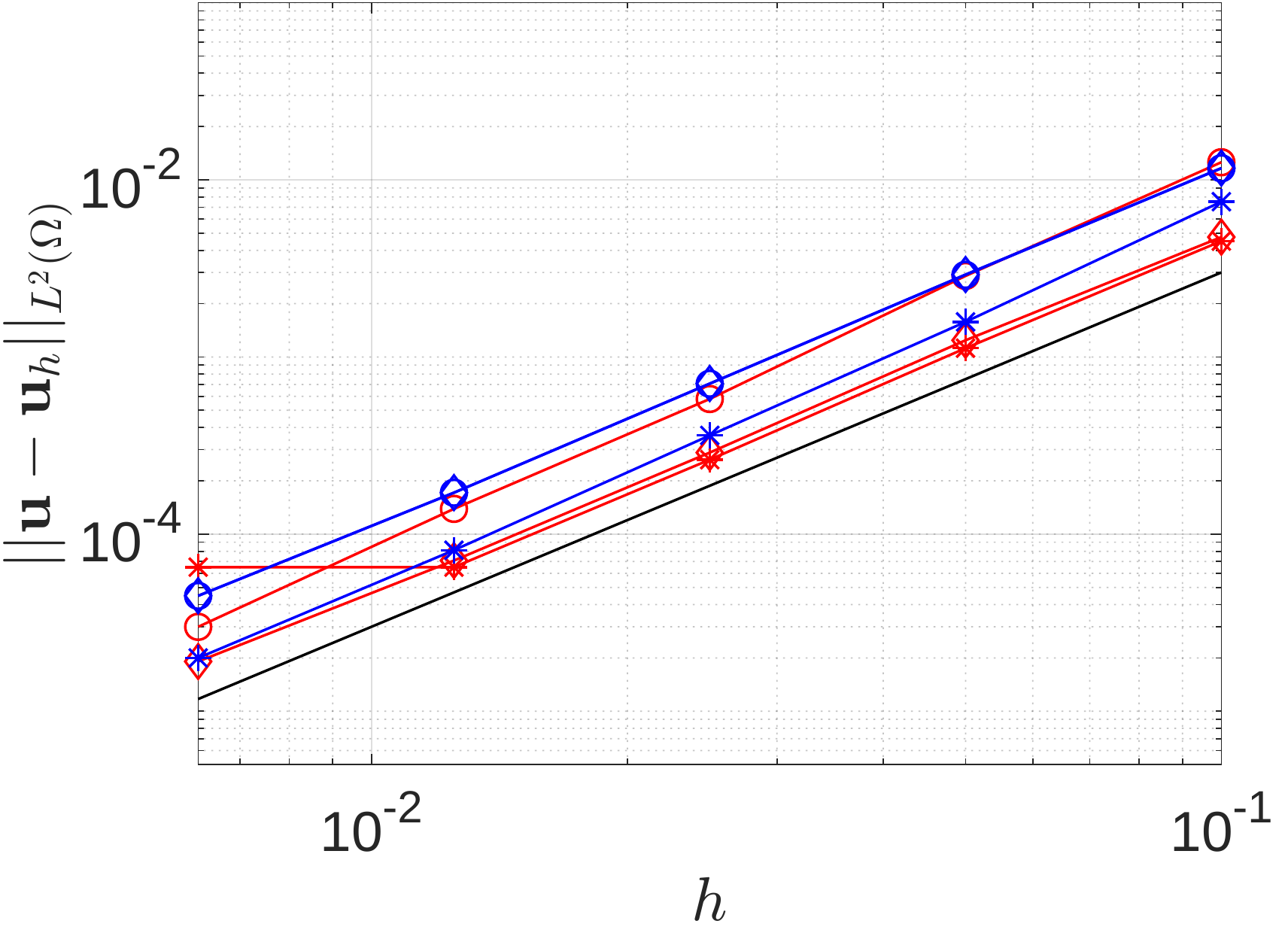}
	\end{subfigure} \\
	\includegraphics[scale=0.8]{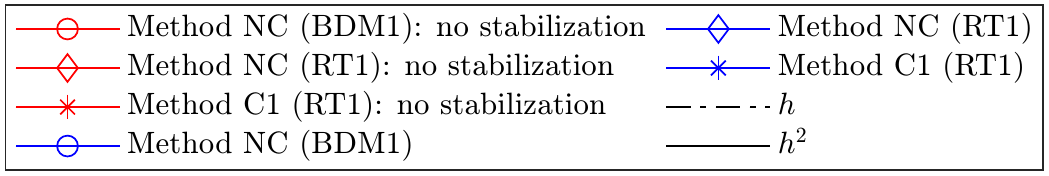}
	\caption{ 
		Example 1: Method NC ($\BDM_1\times Q_0,\ \RT_1\times Q_1$) vs. Method C1 ($P_2\times\RT_1\times Q_1$).
		Left: The $L^2$-error of the pressure versus mesh size h. Right:  The $L^2$-error of the velocity field versus mesh size $h$. (The result on the finest mesh for the unstabilized version of Method C1, represented by red stars, is faulty. The condition number is greater than $10^{20}$ and the linear solver fails.)
		\label{fig:ex1_up} }   
\end{figure}
We see that both stabilized methods have second order convergence for the velocity, which is optimal. In the left panel of Figure \ref{fig:ex1_up} we see that Method NC with the pair $\RT_1\times Q_1$ suffers from the expected suboptimal first order convergence of the pressure (although the convergence order seems to be slightly better than first  order), while with the pair $\BDM_1\times Q_0$ the convergence order is optimal also for the pressure (first order).  For Method C1 with $P_2\times\RT_1\times Q_1$ we obtain optimal second order convergence for the pressure when stabilization is applied. For the unstabilized version of Method C1, represented by red stars in Figure \ref{fig:ex1_up}, the condition number is greater than $10^{20}$ on the finest mesh and the linear solver fails.
\begin{figure}[ht]
	\centering
	\begin{subfigure}{0.4\textwidth}
		\includegraphics[width=\textwidth]{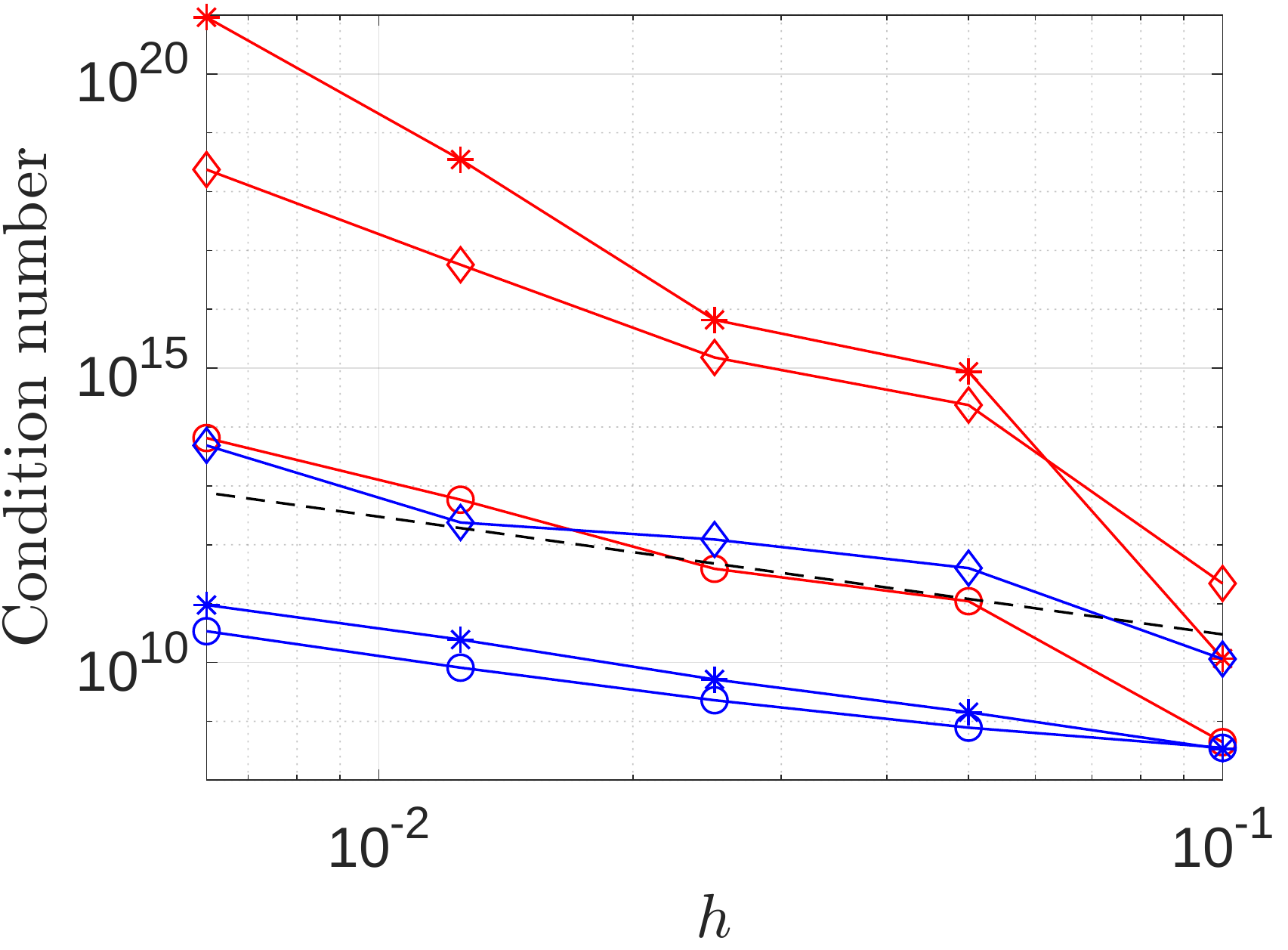}
	\end{subfigure} 
	\quad
	\begin{subfigure}{0.4\textwidth}
		\includegraphics[width=\textwidth]{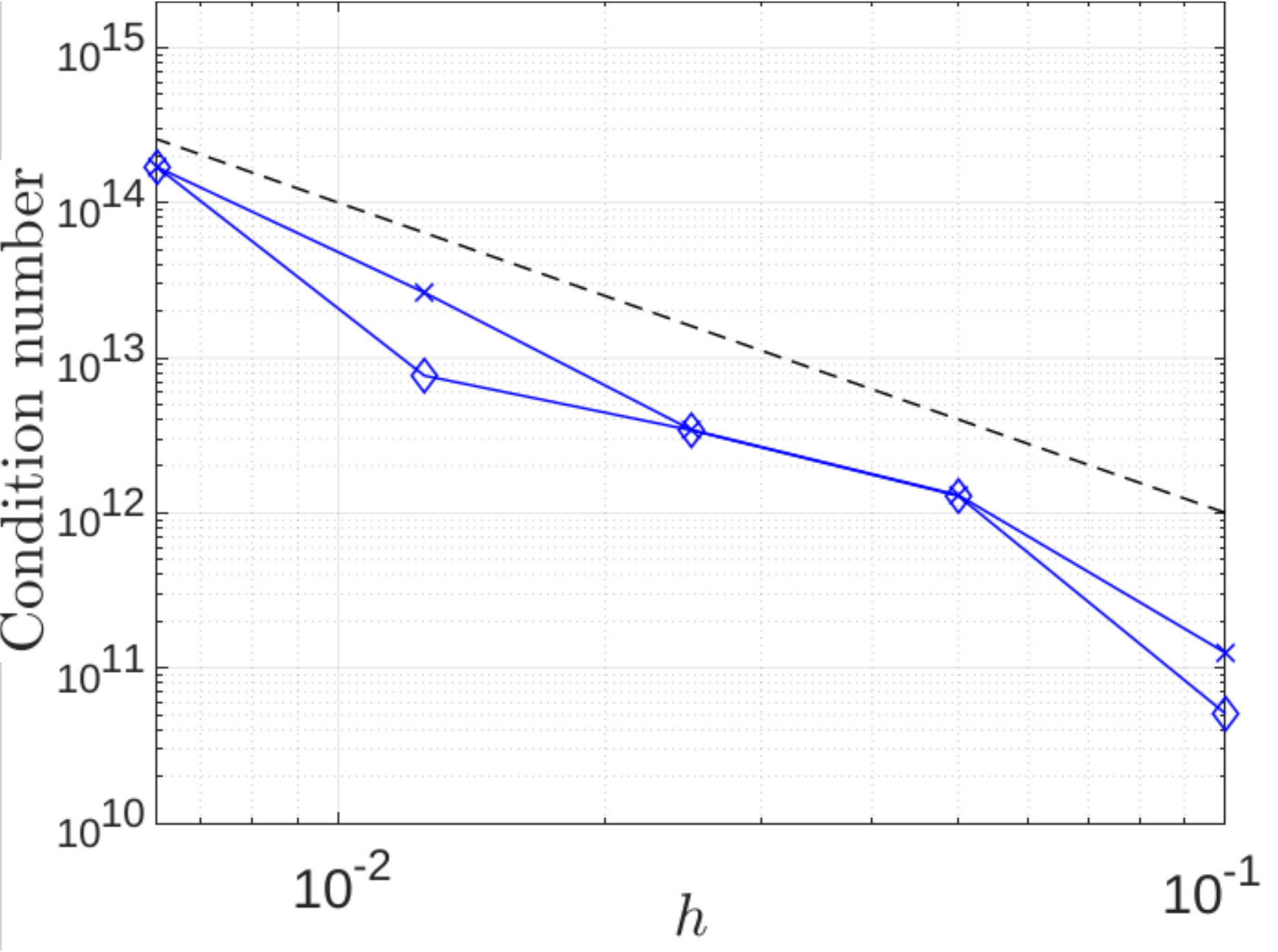}
	\end{subfigure} \\
	\includegraphics[scale=0.8]{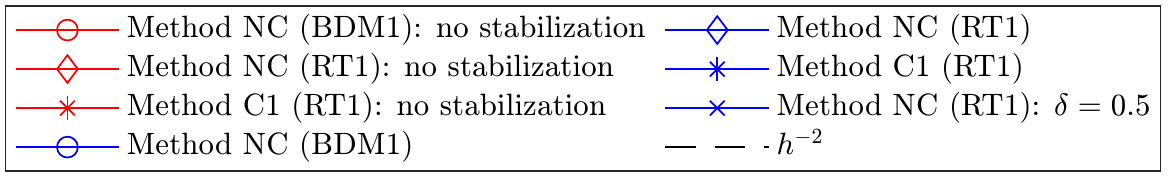}
	\caption{ 
		Example 1: Method NC ($\BDM_1\times Q_0,\ \RT_1\times Q_1$) vs. Method C1 ($P_2\times\RT_1\times Q_1$).
		Left: The $1$-norm condition number versus mesh size h. Right: The $1$-norm condition number versus mesh size h, comparing $\delta=0.5$ with $\delta=1$ for Method NC.
		\label{fig:ex1_cond} }                     
\end{figure}

In Figure \ref{fig:ex1_cond} we see that the condition number of the stabilized methods scale optimally as $O(h^{-2})$. For the $\RT_1$-elements we see that stabilization is really necessary to control the condition number.
In the right panel of Figure \ref{fig:ex1_cond} the condition number of Method NC is shown for the pair $\RT_1\times Q_1$ with both $\delta=0.5$ and $\delta=1$. We observe an $O(h^{-2})$ scaling. 
\subsubsection{The divergence error}\label{sec:num_lagr} 
The proposed methods preserve the incompressibility condition to machine precision. The $L^2$- and $L^{\infty}$-error of the divergence for different mesh sizes $h$ is shown in Table \ref{tab:ex1_div_L2} and Table \ref{tab:ex1_div_Linf}, respectively. 

In this section we study how different approaches of prescribing the constant in the space $\Qk$ 
influence the numerical results when boundary conditions are imposed weakly. We also study the influence of the stabilization terms on the divergence error for Method NC. (The results are comparable for Methods C1 and C2.)  
The convergence order of the methods and the scaling of the condition number with respect to mesh size are not influenced by these aspects, so we focus on the divergence error. 

In the proposed methods we have modified the standard approach in two ways, changing the test spaces and the stabilization, see Section \ref{sec:lagmult} and Section~\ref{sec:standardm}. 
In order to isolate the effects we study these aspects separately:
\begin{itemize}
	\item Modification 1: We use $\Qk^0$ for the pressure test space (as in the standard approach) but we use the proposed stabilization $s_b$.
	\item Modification 2: We stabilize using $s_p$ (as in the standard approach) but we use the proposed spaces, i.e. the pressure test space $\Qk$ and the velocity test space $\Vk^0$.
\end{itemize}

We consider Modification 1 and compare Method NC ($\BDM_1\times Q_0$) with $\lambda_{\bfu}=800$ and Method C1 ($P_1\times\RT_0\times Q_0$) with $\lambda_{u_n}=800$. In Table \ref{tab:ex1_div_Linf_lagr} we show the $L^{\infty}$-error of the divergence for different mesh sizes $h$. We clearly see that the proposed methods result in discrete velocity fields where the error in the divergence is of the order of machine epsilon. However, the divergence-free property of the $\bfH^{\dive}$-conforming finite element spaces is not preserved when the proposed spaces are not used. We see in the left panel of Figure \ref{fig:ex1_diverror} that the divergence is incorrect by exactly the value of the Lagrange multiplier with which the mean value of the functions in the pressure test space are prescribed, see Section~\ref{sec:lagmult}. The shown result is for Method NC but the same holds for Method C1 with Modification 1. 
\begin{figure}[ht]  	
	\centering	
	\begin{subfigure}{0.46\textwidth}
	\includegraphics[width=\textwidth]{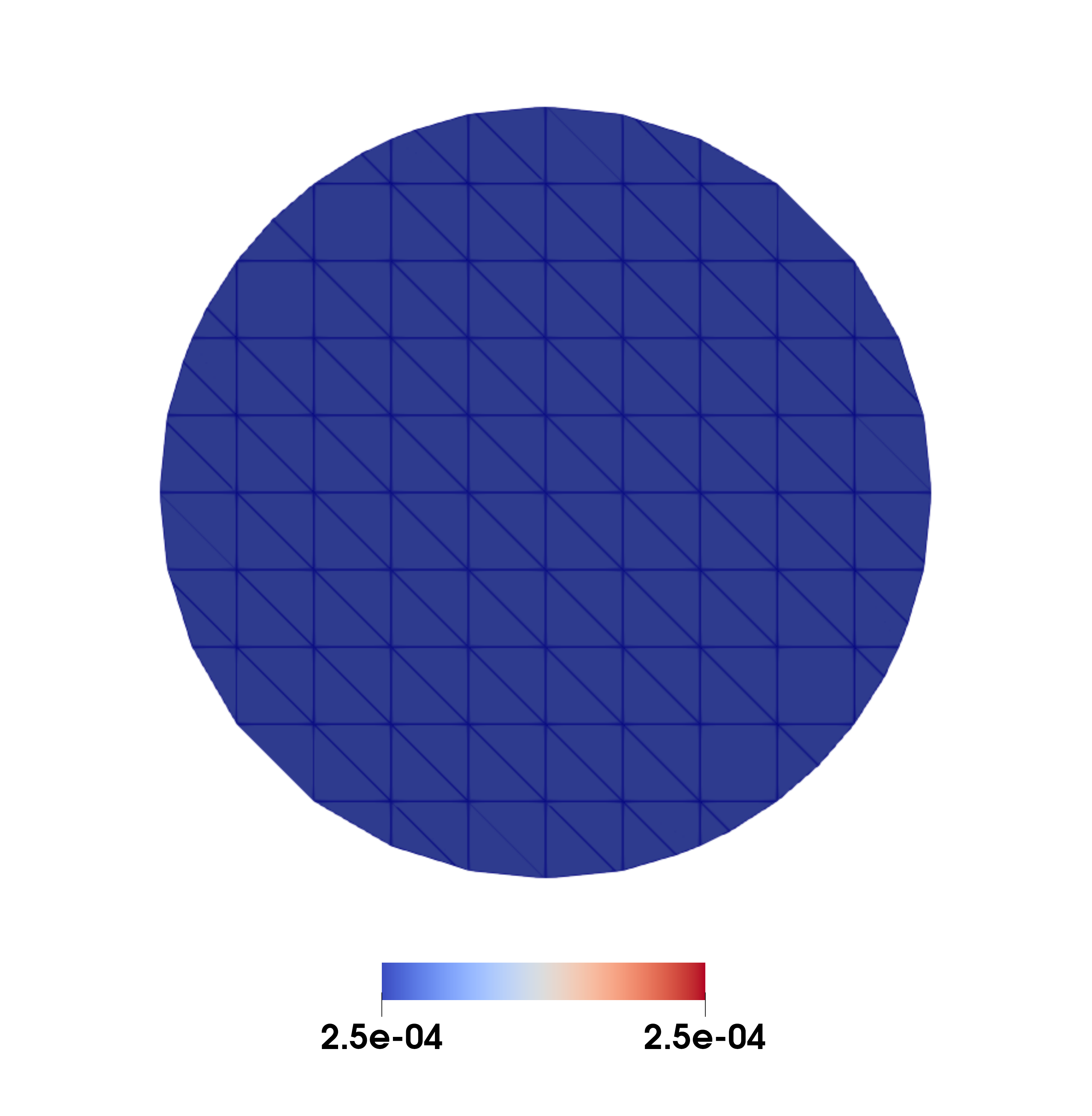}
	\end{subfigure} 
	\quad\quad\quad
	\begin{subfigure}{0.46\textwidth}
		\includegraphics[width=\textwidth]{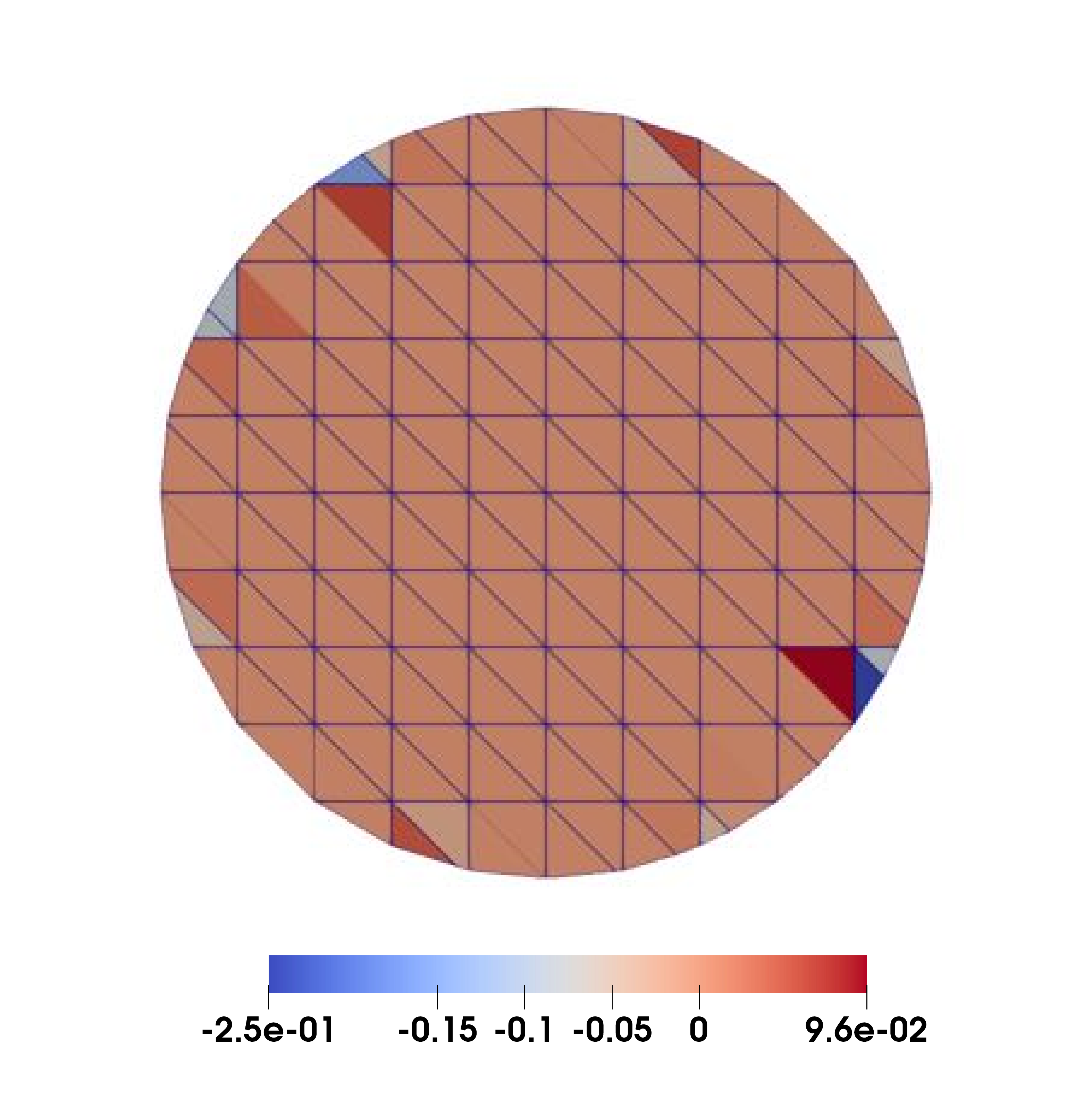}
	\end{subfigure} 
	\caption{ 
		Example 1: Heatmaps of the divergence error using Method NC ($\BDM_1\times Q_0$) for $h=0.05$. Left: The error in the divergence with Modification 1. The error is constant and equal to $2.5e{-04}$, the value of the Lagrange multiplier. Right: The error in the divergence with Modification 2. The divergence-free property of the $\bfH^{\dive}$-conforming elements we use are destroyed by the standard ghost penalty stabilization $s_p$. 
		\label{fig:ex1_diverror} }            
\end{figure}

\begin{table}[ht]
	\resizebox{\textwidth}{!}{
		\begin{tabular}{|l|l|l|l|l|l|l|}
			\hline
			$h$    & 
			 \begin{tabular}[c]{@{}l@{}}Method NC (BDM1):\\ Modification 1\end{tabular} & Method NC (BDM1) & 
			  \begin{tabular}[c]{@{}l@{}}Method C1 (RT0):\\ Modification 1\end{tabular} & Method C1 (RT0) & Method C2 (RT0) \\ \hline
			0.1     & 0.00112924  & 1.55431E-15 & 0.00100837  & 2.66454E-15    &    1.15463e-14       \\
			0.05    & 0.000643155 & 5.77316E-15 & 0.000632555 & 7.99361E-15    &    1.12133e-14       \\
			0.025   & 0.000341402 & 1.5099E-14  & 0.000338262 & 1.55431E-14    &    1.06581e-14       \\
			0.0125  & 0.000176349 & 2.84217E-14 & 0.000177648 & 2.84217E-14    &    3.19744e-14       \\ \hline
	\end{tabular}}
	\caption{Example 1: Max-norm error $\|\dive \bfu- \dive \bfu_h\|_{L^\infty(\Omega)}$ for different mesh sizes $h$, $\lambda_{\bfu}=\lambda_{u_n}=800$. Comparing the proposed methods with the standard way of prescribing $\int_{\Omega} q_h=0$ in the pressure test space $\Qk$ using a Lagrange multiplier. We also compare with the lowest order Method C2.
		\label{tab:ex1_div_Linf_lagr}}
\end{table}
Consider now Method C2. Here we impose the normal component of the Dirichlet boundary condition by introducing a new variable residing in the Lagrange multiplier space $\Qkk^{\Gamma}$, see Section~\ref{sec:method_vort}. Then we can impose the mean value of functions $q_h$ in the test space $\Qk$ without destroying the divergence-free property of the $\bfH^{\dive}$-conforming finite element spaces, see Table~\ref{tab:ex1_div_Linf_lagr}. The divergence error is of order machine epsilon in $L^\infty$-norm.  For Method C2 with $P_1\times \RT_0\times Q_0\times Q_0^{\Gamma}$ and $P_2\times \RT_1\times Q_1\times Q_1^{\Gamma}$ we report the error in the pressure, the velocity, the divergence, as well as the condition number for different mesh sizes in Tables \ref{tab:fourfield1}-\ref{tab:fourfield2}.  Our numerical studies show that Method C1 and Method C2 perform similarly for the $\RT_1$ elements; compare Figure \ref{fig:ex1_up} with Table \ref{tab:fourfield2}. However, for the lowest order discretization (using $\RT_0$ for the velocity) the pressure does not converge optimally for Method C1, see Figure \ref{fig:ex12_upc} and Table \ref{tab:fourfield1}.

The choice of stabilization for the pressure is another important aspect which influences the divergence error in both methods.  
Consider for instance Method NC with Modification 2. Exchanging the stabilization term $s_b(\cdot,\cdot)$ acting on $\Vk\times\Qk$ with the pressure stabilization term $s_p(\cdot,\cdot)$ acting on $\Qk\times\Qk$, which has been used in most previous work, ruins the divergence-free property of the underlying elements. See the right panel of Figure \ref{fig:ex1_diverror} for heatmap of the error; the error in the divergence is large in the vicinity of edges where stabilization is applied. 
A detailed comparison in the case of the Darcy equations was given in \cite{FraHaNilZa22}. However, the $L^{\infty}$-errors we observe here for the Stokes problem are actually worse. 

\subsubsection{Method C for different boundary conditions} \label{sec:num_MethodC_BC}
We make a comparison between the present variant of Method C1, i.e., $(\tau_m,0,\tau_a)$, 
with the variant $(0,\tau_c,0),$ see see \eqref{eq:mC_variant10}-\eqref{eq:mC_variant01}. We compare these different choices of stabilization terms both for Dirichlet boundary conditions \eqref{eq:BCu}, and for the alternative boundary conditions \eqref{eq:altBC}.
For Dirichlet boundary conditions we also compare with Method C2 \eqref{eq:MClagmult}, and for all methods we limit ourselves to the lowest order discretization. (Note that Method C2 is designed specifically for Dirichlet boundary conditions.) 
Namely, we study
\begin{itemize}
	\item Dirichlet boundary conditions \eqref{eq:BCu}: $\bfu=\mathbf{0}$
	\begin{itemize}
		\item Method C2 ($P_1\times\RT_0\times Q_0\times Q_0^{\Gamma}$ ) 
		\item Method C1 ($P_1\times\RT_0\times Q_0$): variant $(\tau_m,0,\tau_a)$ with $\lambda_{u_n} = 400$
		\item Method C1 ($P_1\times\RT_0\times Q_0$): variant $(0,\tau_c,0)$ with $\lambda_{u_n} = 400$
	\end{itemize}
	\item Alternative boundary conditions \eqref{eq:altBC}: $\bfu\cdot\bft=0$ and $p=10(x^2-y^2)^2$
	\begin{itemize}
		\item Method C1: variant $(\tau_m,0,\tau_a)$
		\item Method C1: variant $(0,\tau_c,0)$
	\end{itemize}
\end{itemize}

We can start by observing that it is really necessary to stabilize Method C2; if this is not done we can see in Figure \ref{fig:ex12_upc} that the pressure error does not converge, and the condition number is very large at $\approx 10^{40}$.
In the middle panel of Figure \ref{fig:ex12_upc} we see not much difference between the two versions $(\tau_m,0,\tau_a)$ and $(0,\tau_c,0)$ in terms of their velocity convergence. The pressure convergence is best for Method C2, being the only one to achieve and convergence order of $1$. The other two methods seem to tend to a convergence order of $0.5$.
In the right panel we see the condition number as a function of mesh size $h$. The condition number is controlled with all variants but the magnitude is lowest for Method C2.
\begin{figure}[ht]
	\centering 
	\begin{subfigure}{0.33\textwidth}
		\includegraphics[width=\textwidth]{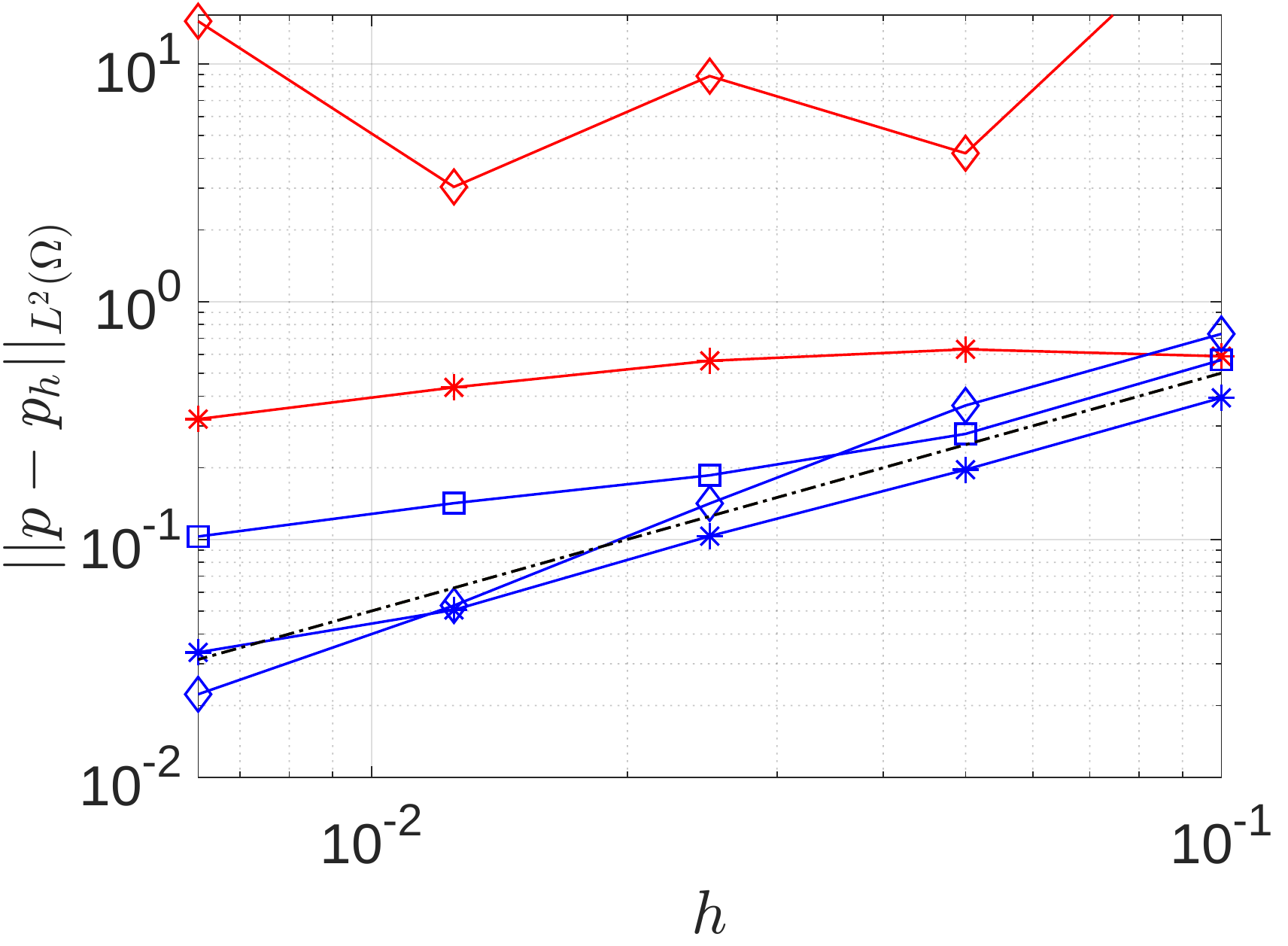}
	\end{subfigure}
	\hfill
	\begin{subfigure}{0.33\textwidth}
		\includegraphics[width=\textwidth]{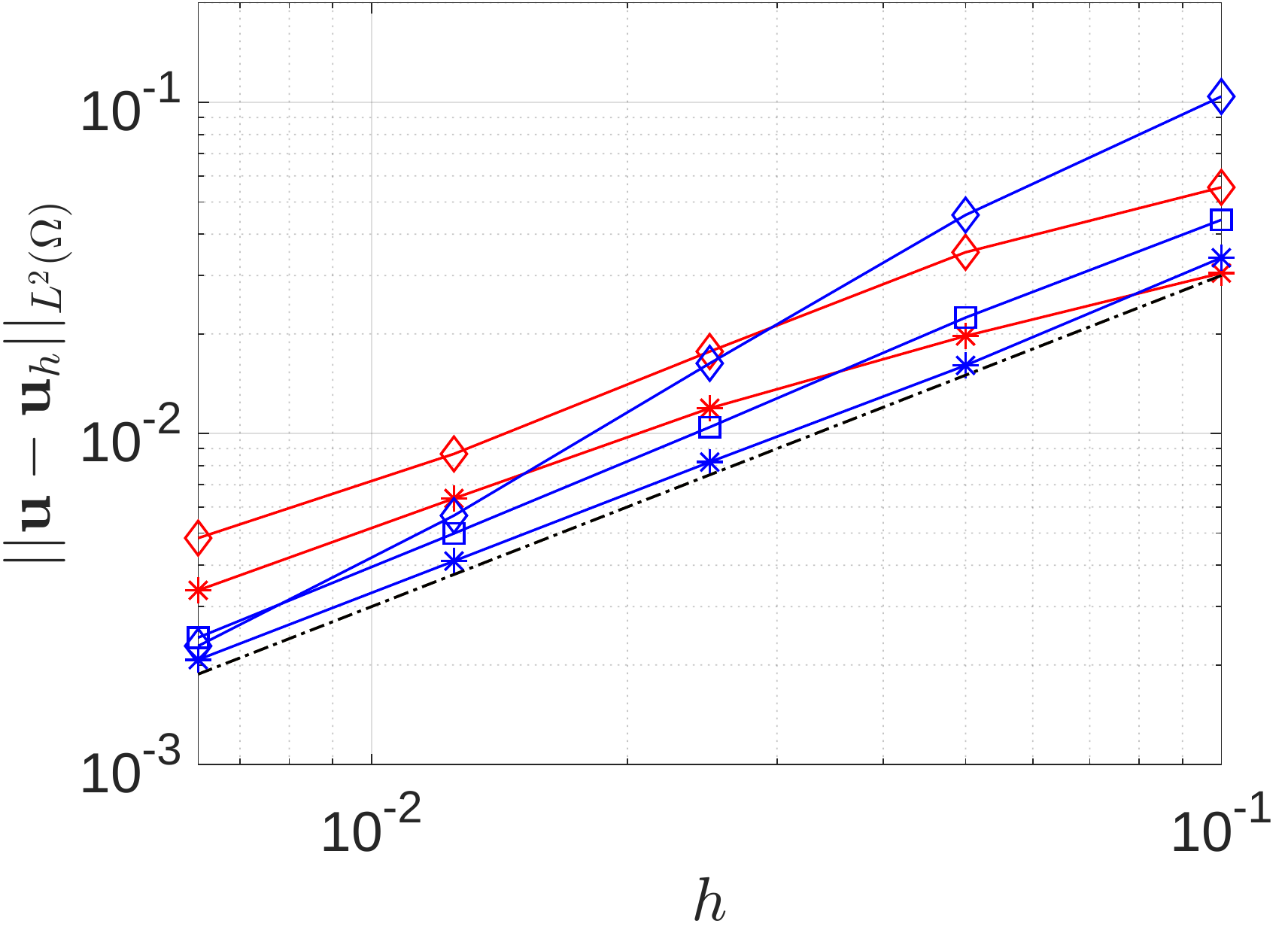}
	\end{subfigure}
	\hfill
	\begin{subfigure}{0.33\textwidth}
		\includegraphics[width=\textwidth]{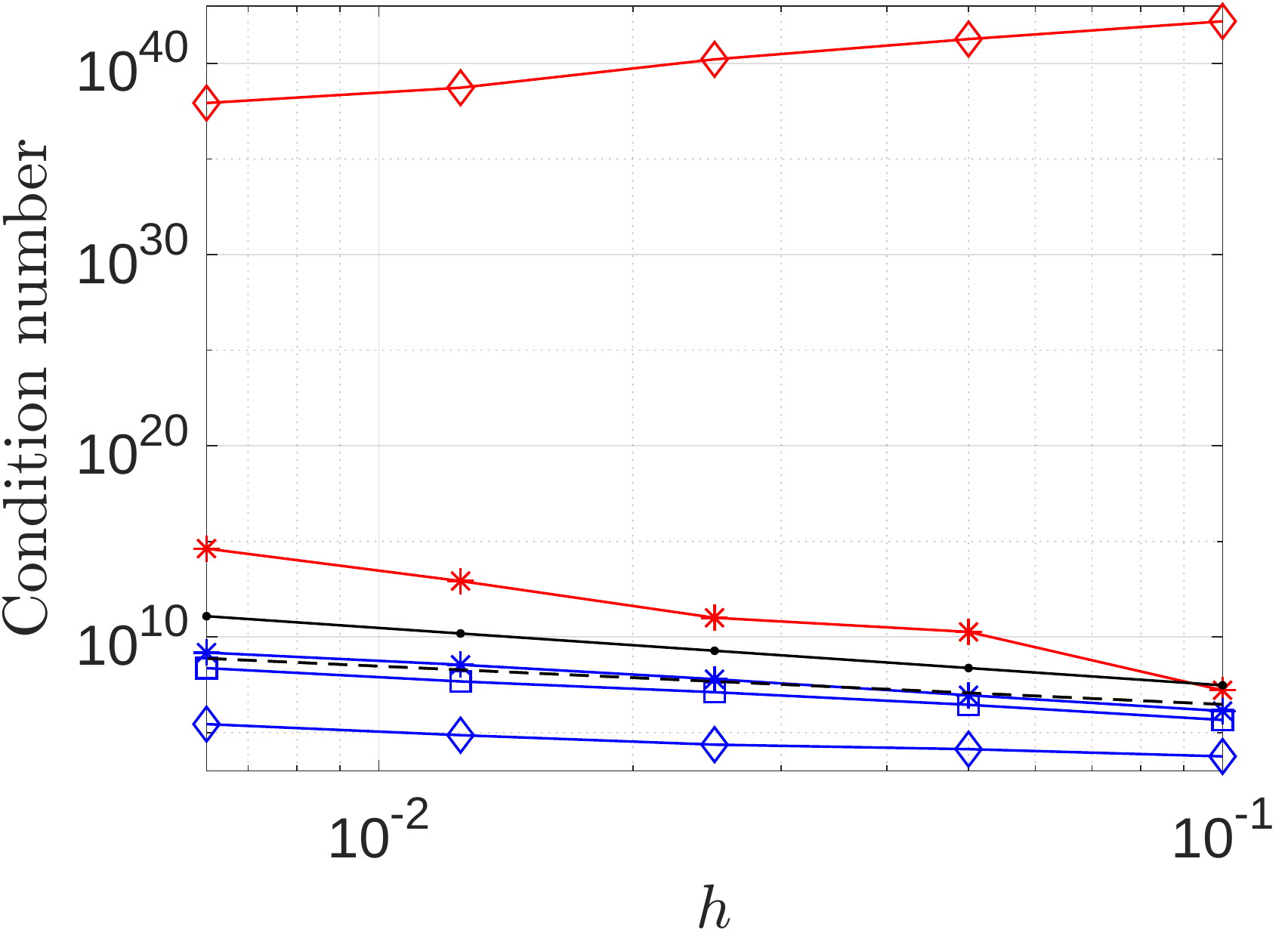}
	\end{subfigure} \\
	\includegraphics[scale=0.8]{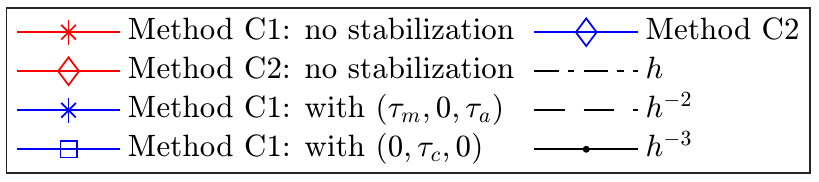}
	\caption{ 
		Example 1 Dirichlet BC: Method C1 $(\tau_m,0,\tau_a)$ vs. $(0,\tau_c,0)$ vs. no stabilization ($\tau_\bullet=0$) vs. Method C2. Pressure and velocity errors compared for the lowest order elements $P_1\times \RT_0\times Q_0$.
		Left: The $L^2$-error of the pressure versus mesh size h. Middle:  The $L^2$-error of the velocity field versus mesh size h. 
		Right: The condition number versus mesh size $h$.
		\label{fig:ex12_upc} }   
\end{figure}

The divergence is not impacted in any remarkable way; the errors have the same order of magnitude for each mesh size $h$. In summary, for Dirichet boundary conditions it seems that Method C2 performs best.

Next we study the performance of Method C1 for the Stokes equations with boundary conditions \eqref{eq:altBC}. 
With these boundary conditions there is no uniqueness issue for the pressure, so we do not have to introduce any Lagrange multiplier. 
Now the most natural choice of stabilization for Method C1 is $(0,\tau_c,0)$, see \eqref{eq:mC_variant10}-\eqref{eq:mC_variant01}. 
We compare this variant with the $(\tau_m,0,\tau_a)$-variant of the previous subsection. 
For this boundary condition the discretization seems to not require as much stabilization. We could also have chosen $\delta=0.25$. We can see in Figure \ref{fig:exBC2_upc} that both variants of stabilization perform optimally in all respects for the lowest order elements, order $k=0$, i.e., using the triple $P_1\times\RT_0\times Q_0$. However, for $k=1$, i.e., using $P_2\times\RT_1\times Q_1$, we see in the right panel of Figure \ref{fig:exBC2_upc} that variant $(0,\tau_c,0)$ behaves better with respect to the scaling of the condition number with mesh size. It is clear that the stabilization $(\tau_m,0,\tau_a)$ is not correct for this type of boundary condition since we see that the condition number of the resulting system is larger with this stabilization compared to not having any stabilization. 
Hence which stabilization terms one should include in Method C1 depends on the type of boundary conditions used. For the Dirichlet boundary conditions of \eqref{eq:BCu} the variant $(\tau_m,0,\tau_a)$ behaves better, while for the alternative boundary conditions of \eqref{eq:altBC} one should choose the $(0,\tau_c,0)$-variant. 
\begin{figure}[ht]
	\centering
	\begin{subfigure}{0.33\textwidth}
		\includegraphics[width=\textwidth]{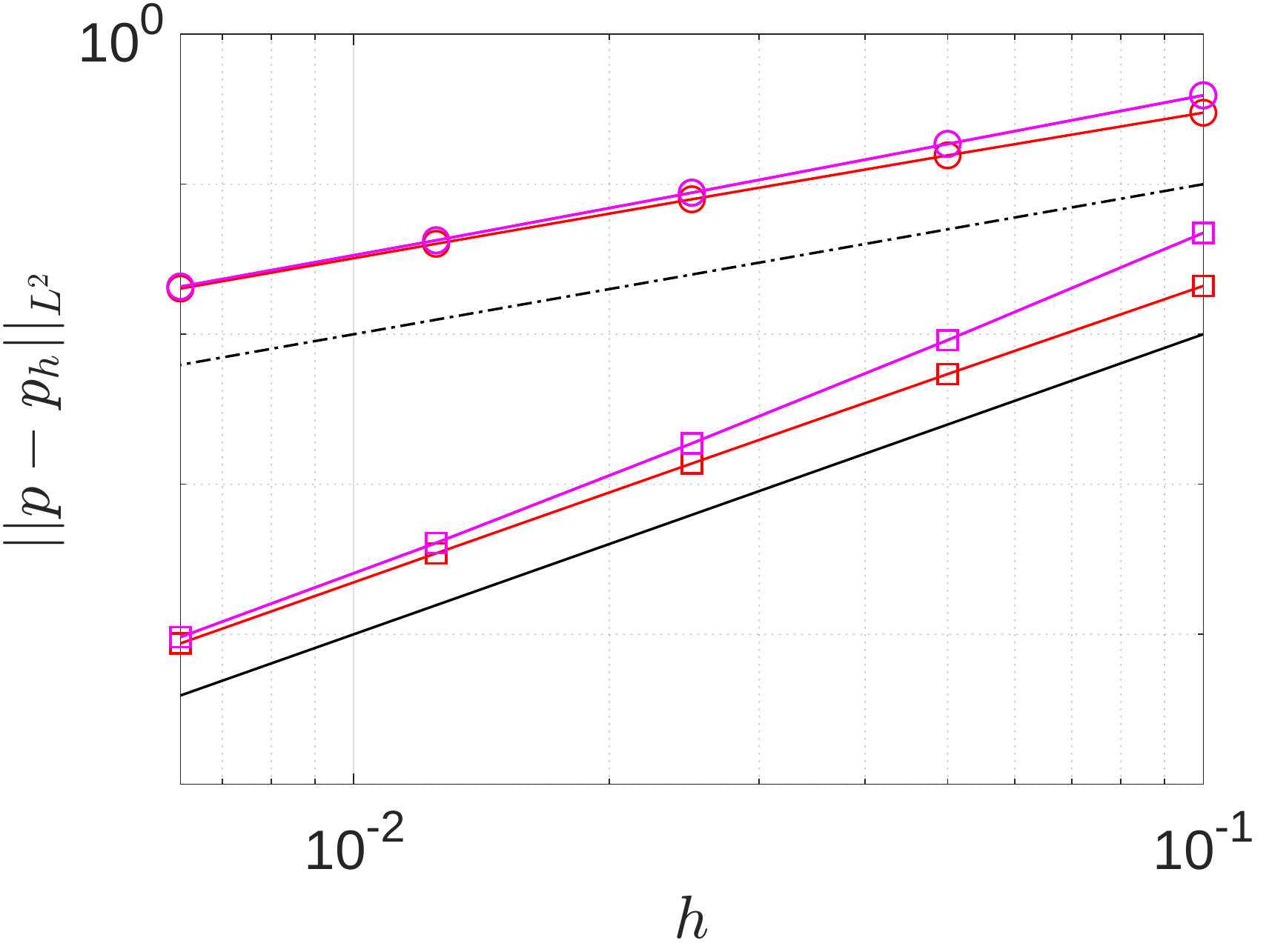}
	\end{subfigure} 
	\hfill
	\begin{subfigure}{0.33\textwidth}
		\includegraphics[width=\textwidth]{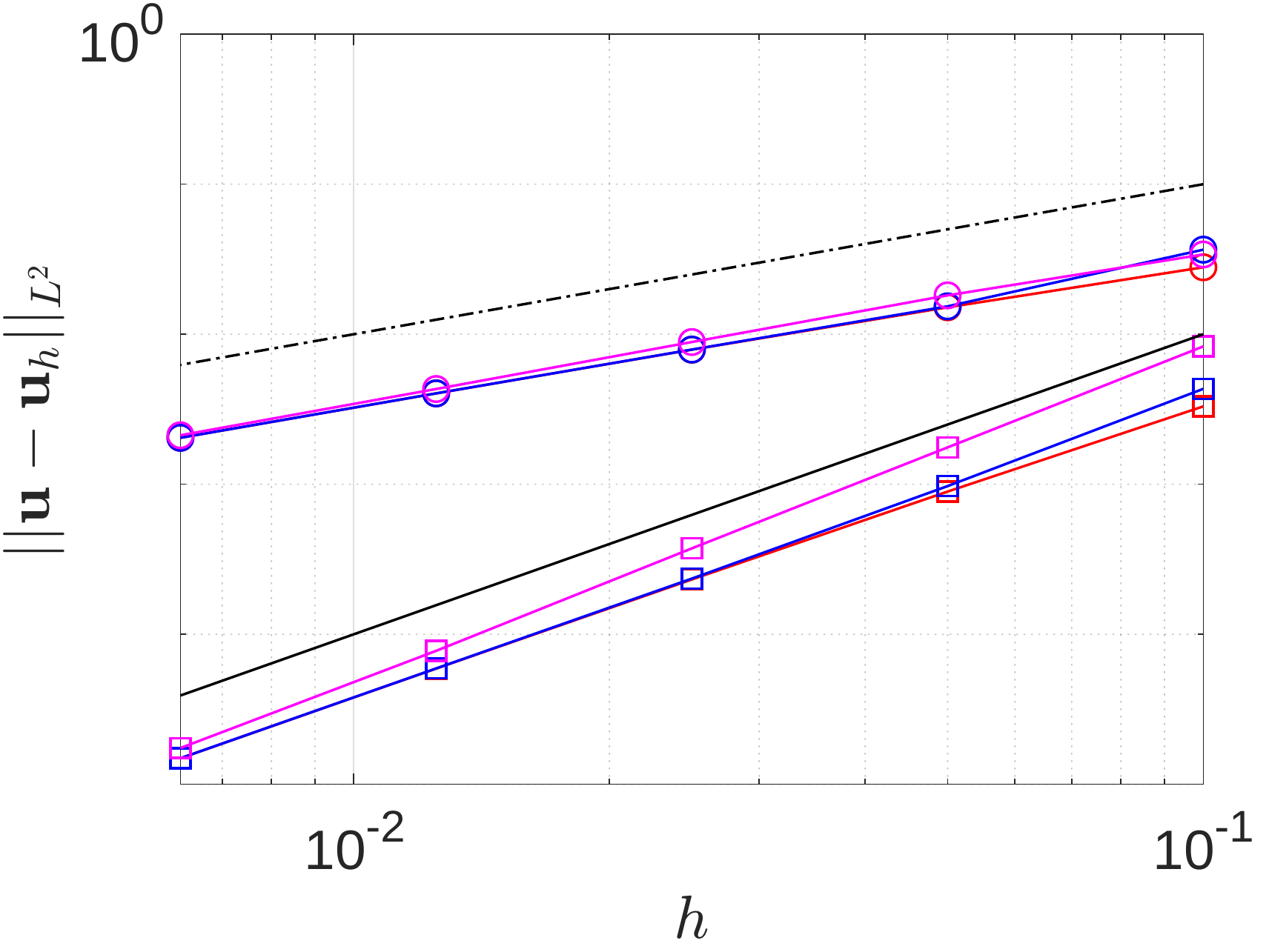}
	\end{subfigure} 
	\hfill
	\begin{subfigure}{0.33\textwidth}
		\includegraphics[width=\textwidth]{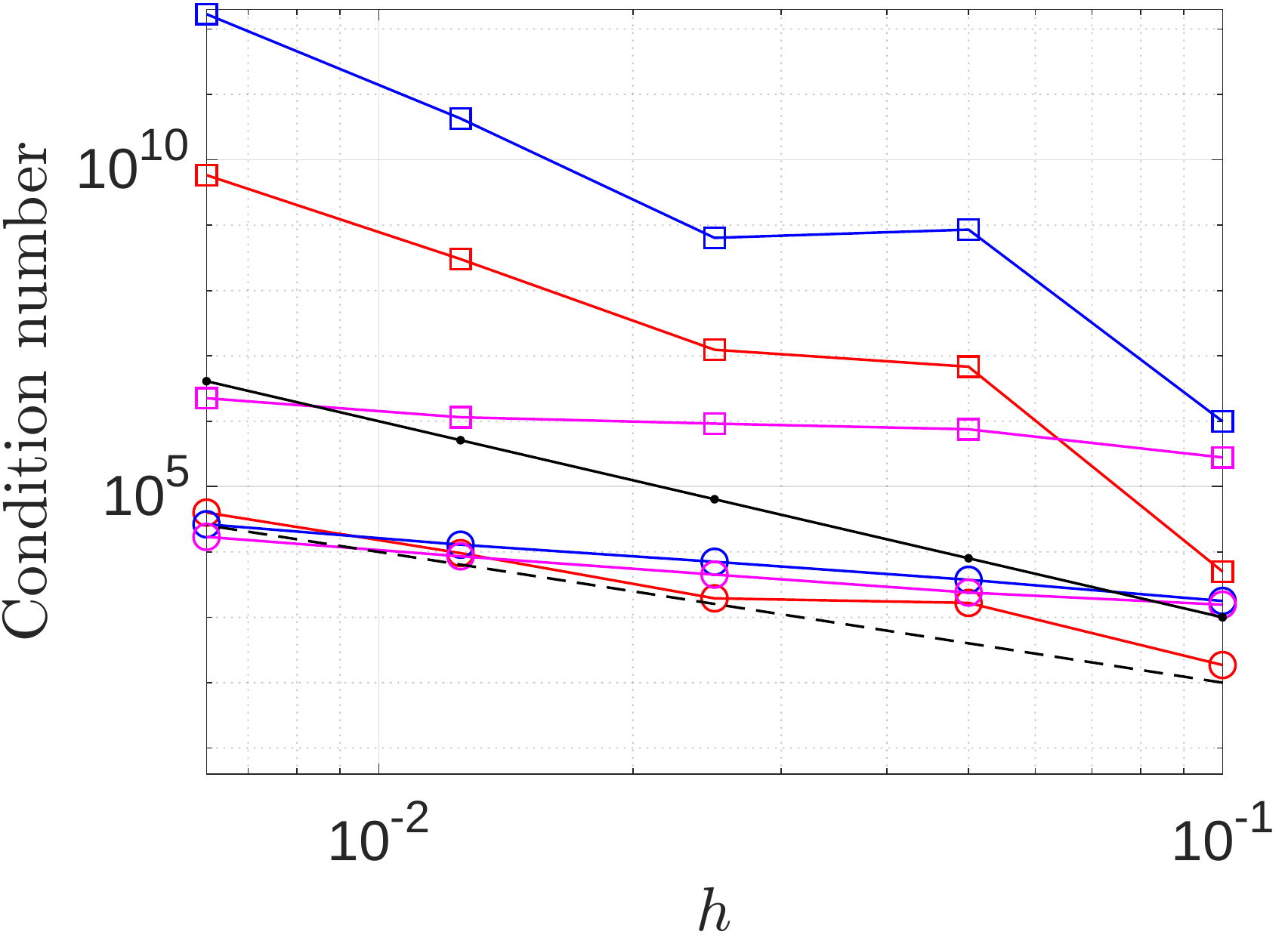}
	\end{subfigure}
	\includegraphics[scale=0.8]{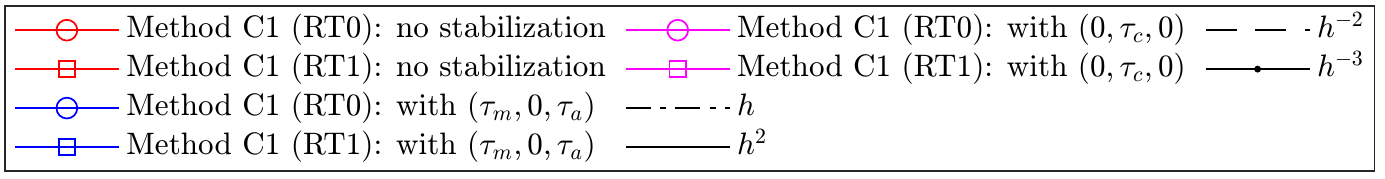}
	\caption{ 
		Example 1 alternative BC: Method C1 $(\tau_m,0,\tau_a)$ vs. $(0,\tau_c,0)$ vs. no stabilization ($\tau_{\bullet}=0$). Pressure and velocity errors and condition numbers for both element triples $P_1\times\RT_0\times Q_0$ and $P_2\times\RT_1\times Q_1$. 
		Left: The $L^2$-error of pressure versus mesh size $h$. Middle: The $L^2$-error of velocity versus mesh size $h$. Right: The $1$-norm condition number versus mesh size $h$.
		\label{fig:exBC2_upc} }                     
\end{figure}

\subsection{Example 2 - Coriolis force, pressure robustness}
We consider an example from~\cite{VolkerDivConstr2017}, where for $\bfu=(u_1,u_2,0)$ the momentum balance is modified with a Coriolis force $\pmb{c}\times\bfu$,
 so that we consider the problem 
\begin{align}
	-\dive(\mu\nabla\bfu - p\mathbf{I}) +\pmb{c}\times\bfu = -\dive(\mu\nabla\bfu - p\mathbf{I}) + 2\Lambda \begin{bmatrix}
		-u_2 \\ u_1
	\end{bmatrix} = 0, \quad \Lambda>0.
\end{align} 
In the weak formulations of Method NC and Method C, the term $(2\Lambda \begin{bmatrix}
	-u_2 \\ u_1
\end{bmatrix},\bfv)_{\Omega}$ is added in \eqref{eq:mNC_discreteStokes} and \eqref{eq:mC_momentumbalance}.

For standard boundary-fitted FEM, pressure robustness in a discretization of Stokes and incompressible Navier-Stokes is the following property (which we only describe very briefly, see \cite{VolkerDivConstr2017} for details): A modification of the momentum balance with a gradient term should only impact the pressure solution, i.e. 
\begin{align}\label{eq:presrob}
-\dive(\mu\nabla\bfu - p\mathbf{I}) + \nabla\psi 
\implies (\bfu_h,p_h)\mapsto (\bfu_h,p_h+\psi). 
\end{align}
In the present example, since $\crl (\pmb{c}\times\bfu) = 0$, there is a potential function $\psi$ such that $\nabla\psi = \pmb{c}\times\bfu$, and therefore the term $\pmb{c}\times\bfu$ should not impact the velocity in a pressure robust scheme.
A necessary condition for \eqref{eq:presrob} to hold is that $\dive\Vk\subset\Qk,$ and a question is whether or not this also holds in the unfitted setting. 

\begin{figure}[ht]
	\centering 
	\begin{subfigure}[b]{0.45 \textwidth}  	
		\centering	
		\includegraphics[scale=0.30]{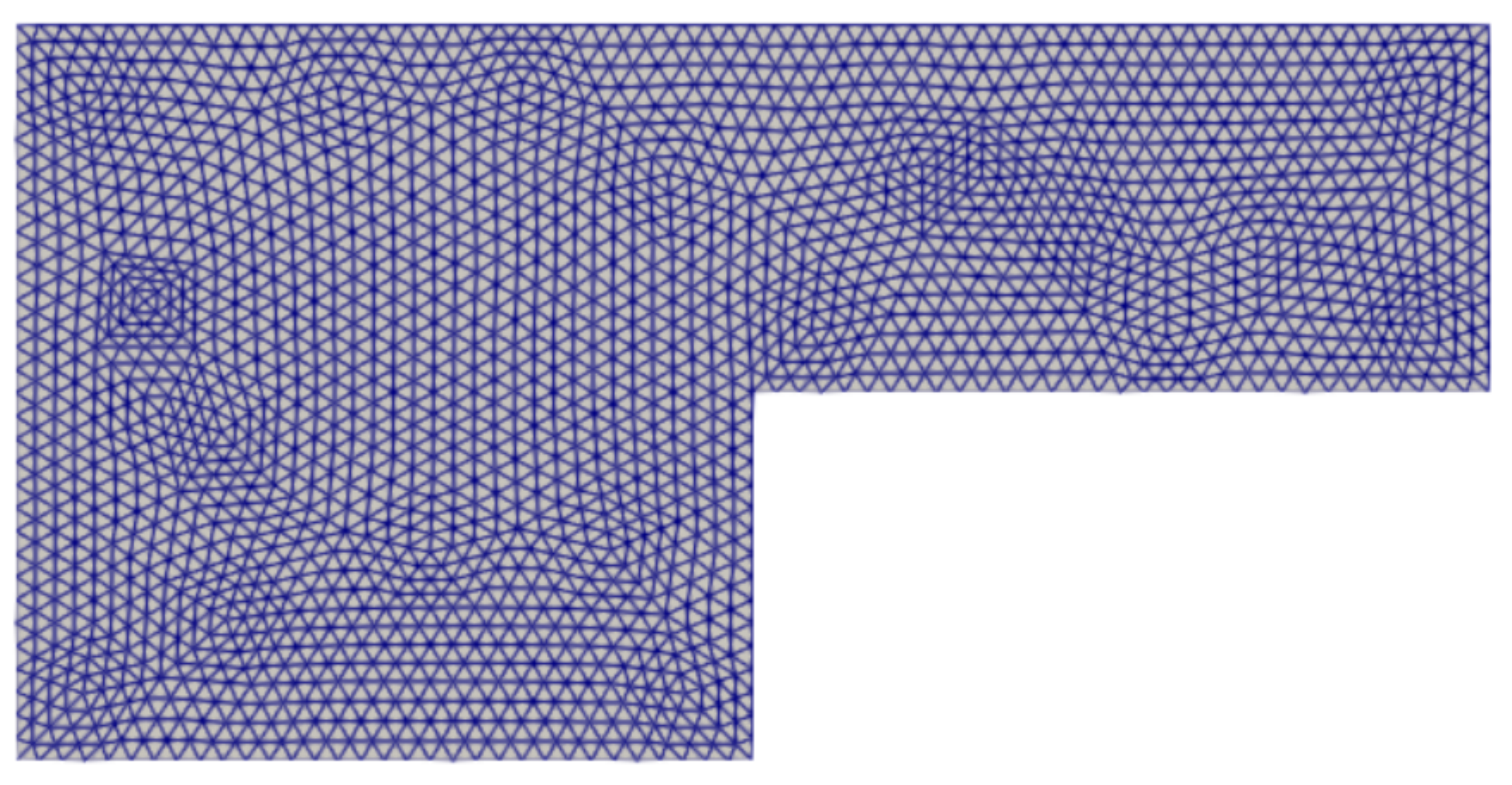}
	\end{subfigure}
	\begin{subfigure}[b]{0.45 \textwidth}  	
		\centering	
		\includegraphics[scale=0.16]{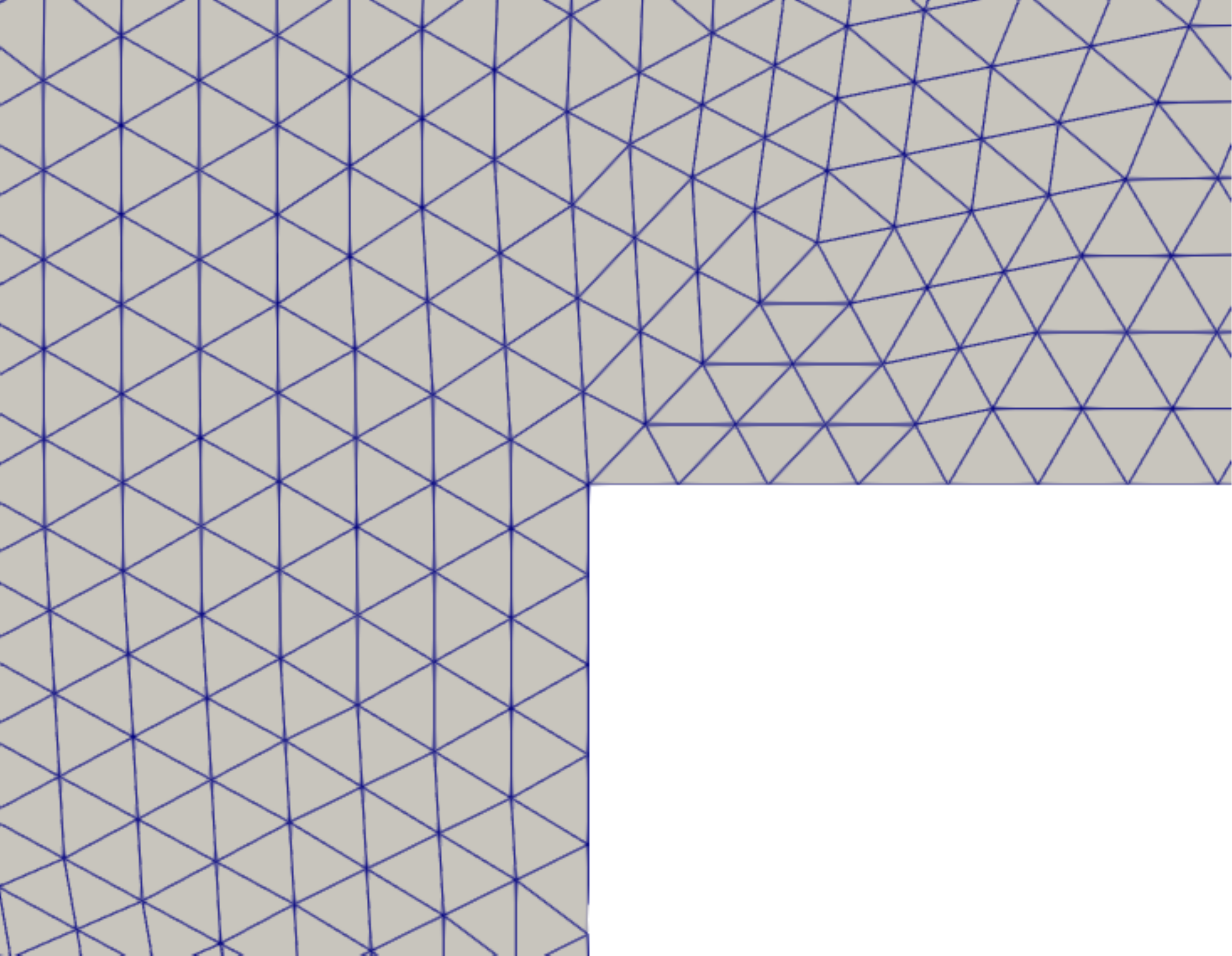}
	\end{subfigure}
	\caption{ 
			Example 2: A fitted computational mesh with $h=\max_{T\in\mesh}h_T=0.0742408$. A zoom-in on the corner at $(x,y)=(2,1)$ is shown on the right panel. 
			\label{fig:presrob_fitted_mesh} }      
\end{figure}

\begin{figure}[ht]
	\centering 
	\begin{subfigure}[b]{0.45 \textwidth}  	
		\centering	
		\includegraphics[scale=0.19]{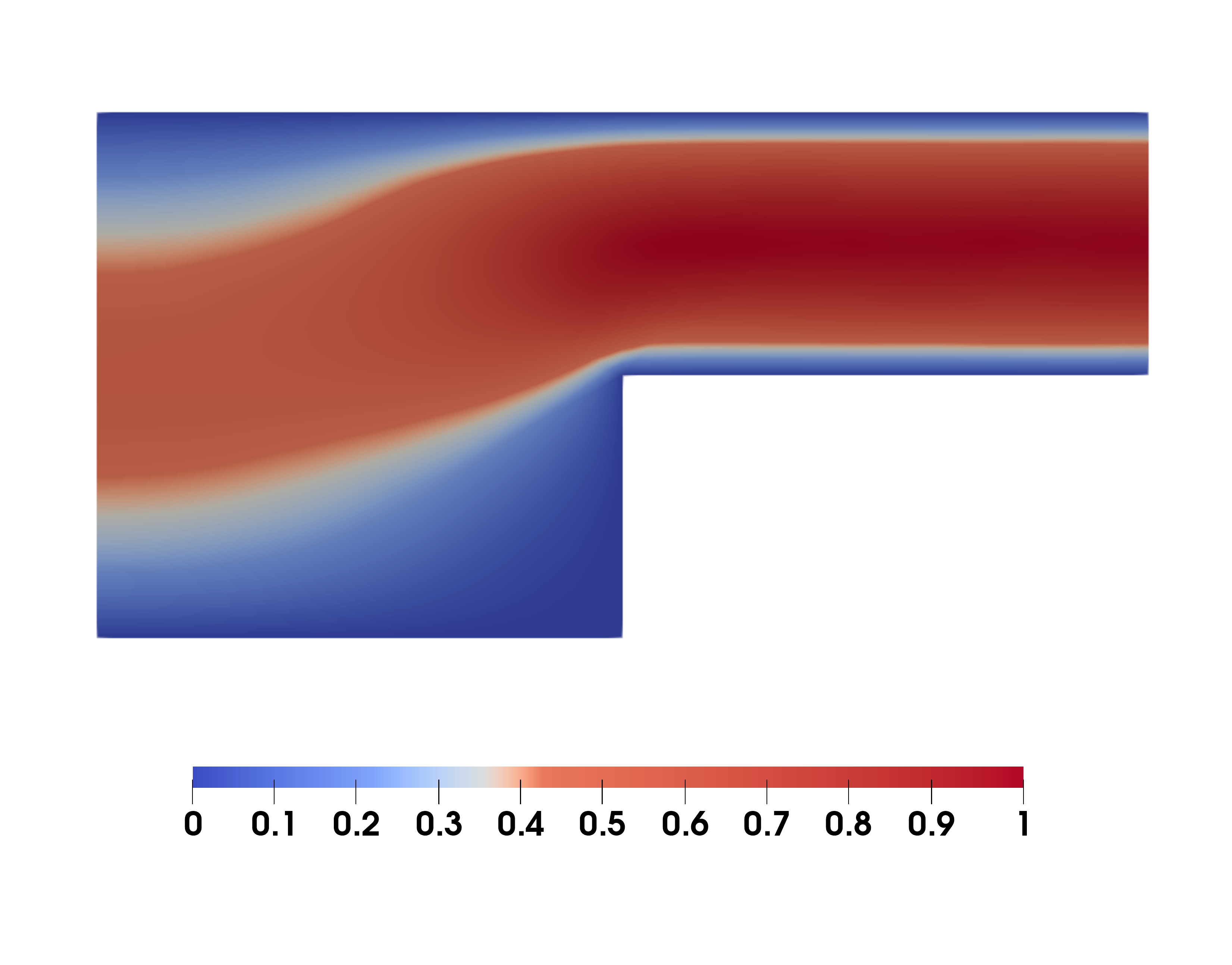}
	\end{subfigure} 
	\begin{subfigure}[b]{0.45 \textwidth}  	
		\centering	
		\includegraphics[scale=0.19]{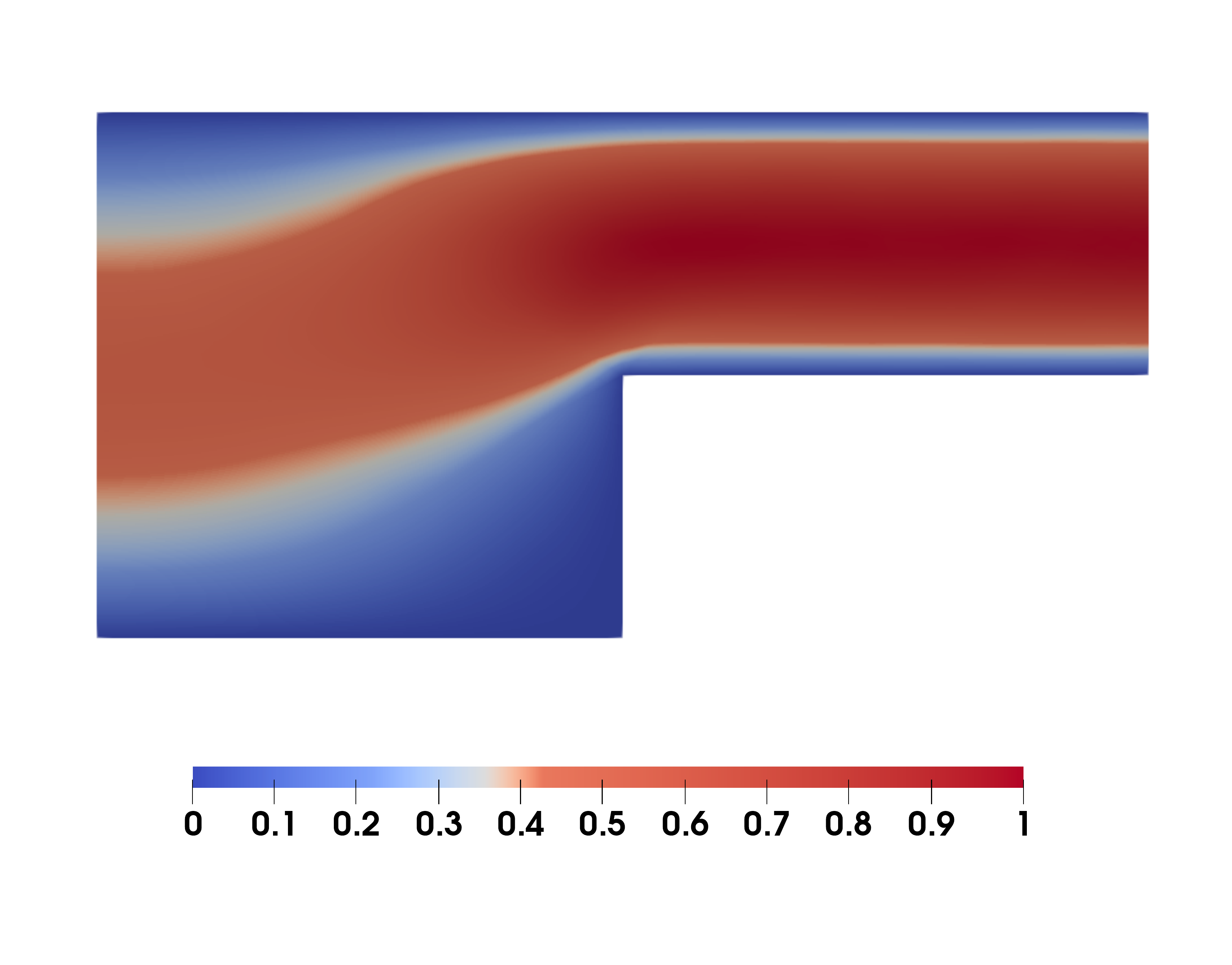}
	\end{subfigure}
	\caption{ 
			Example 2: Computations on the boundary-fitted mesh with $h=\max_{T\in\mesh}h_T=0.0742408$. Magnitude of the approximate velocity field using the method of \cite{WangRobust2009} with the element pair $\BDM_1\times Q_0$. Left: $\Lambda=0.$ Right: $\Lambda=2000.$
			\label{fig:presrob_BDM1_fitted} }            
\end{figure}

We take $\mu=0.01$ and compare the approximate velocity field for $\Lambda=0$ with $\Lambda=2000$.  The geometry is an L-shaped domain $[0,2]\times [0,2] \cup [2,4]\times [1,2]$.  On the lower and upper boundary we have no-slip boundary conditions (i.e. $\bfu=0$). At the inlet, at $x=0$, and at the outlet, $x=4$, we have the following boundary conditions for $\bfu$:
\begin{align}
	\bfu_{\text{in}} &= (y(2-y)/2,0)^T, \\
	\bfu_{\text{out}} &= (4(2-y)(y-1),0)^T.
\end{align}

On a fitted mesh, Method NC without stabilization terms and imposing the normal component of the boundary condition strongly, gives the same discretization as the method of \cite{WangHdiv2007,WangRobust2009}. In Figure \ref{fig:presrob_fitted_mesh} we show the fitted mesh we use and in Figure \ref{fig:presrob_BDM1_fitted} we show the approximate velocity for $h=\max_{T\in\mesh}h_T=0.0742408$.  Results for $\Lambda=0$ and $\Lambda=2000$ are shown in the right and the left panel of the figure, respectively. We see that the velocity field is not affected by the added gradient term. 

For Method NC we have divided the penalty $\lambda_{\bfu}$ in the tangential component $\lambda_{u_t}=\lambda_t=1$ and the normal component $\lambda_{u_n}=400$. 
For Method C1 we 
also choose $\lambda_{u_n}=400$. 

Let the unfitted boundary be $\Gamma=\{(x,y)\in\RR^2: x=2,\ 0\leq y \leq 1 \}\cup \{(x,y)\in\RR^2: y=1,\ 2\leq x \leq 4 \}$, and consider a background mesh $\Omega_0=[0,4]\times [0,2]$. The unfitted mesh is shown in Figure \ref{fig:presrob_unfitted_mesh}. We see also on this unfitted mesh that the Coriolis force does not seem to impact the velocity field, see the top panels of Figure \ref{fig:presrob_BDM1_unfitted}. 

If we on the other hand use the test spaces $(\Vk, \Qk^0)$ and stabilize the pressure using $s_p$, the Coriolis force has an impact on the velocity solution, see the bottom panels of Figure \ref{fig:presrob_BDM1_unfitted}. Note that apart from the qualitative change of the flow there is also a slight change in the magnitude of the velocity field. 
With Method C1 \eqref{eq:mC_curleq} the situation is analogous, see Figure \ref{fig:presrob_RT1_unfitted}.
\begin{figure}[ht]
	\centering 
	\begin{subfigure}[b]{0.45 \textwidth}  	
		\centering	
		\includegraphics[scale=0.30]{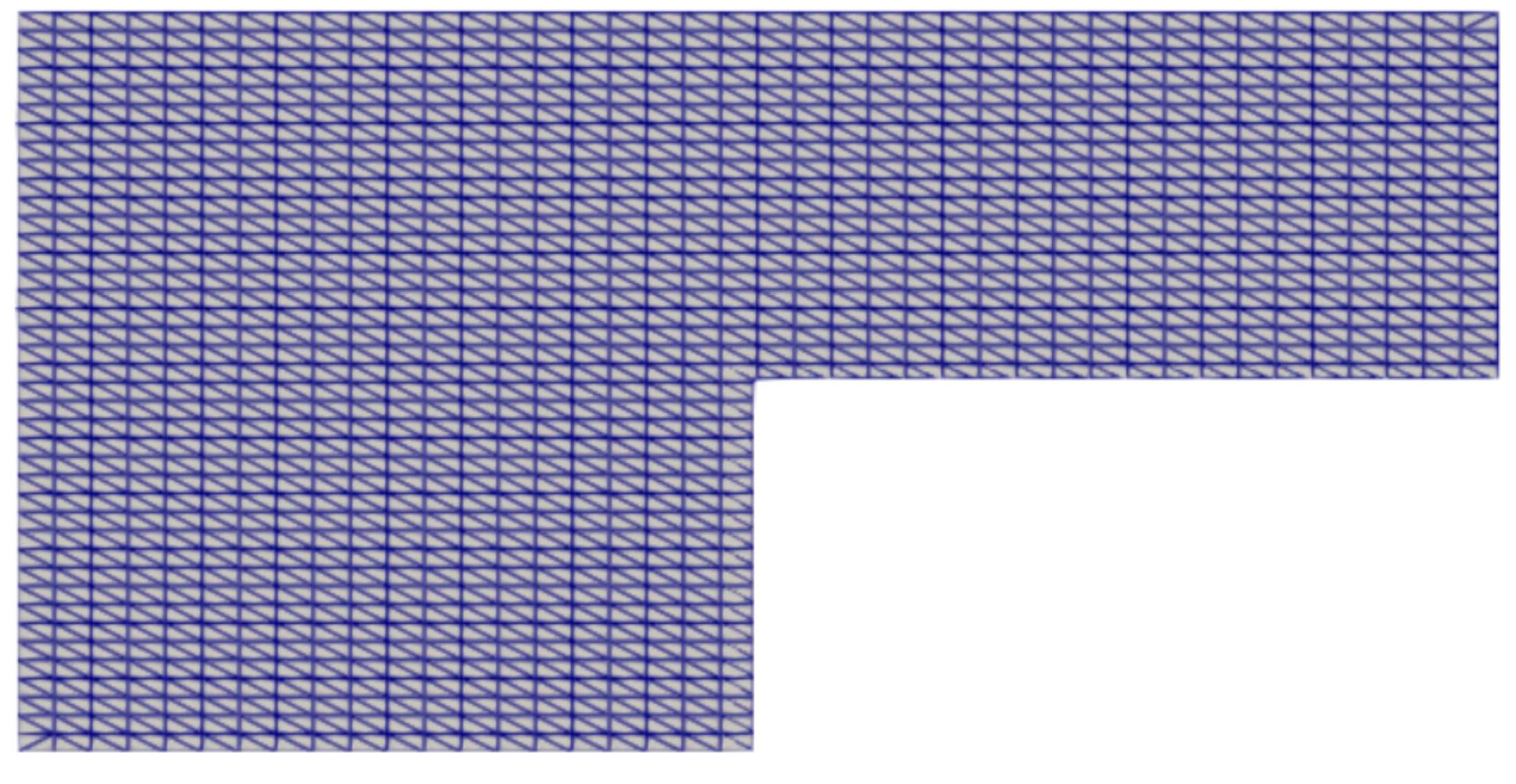}
	\end{subfigure}
	\begin{subfigure}[b]{0.45 \textwidth}  	
		\centering	
		\includegraphics[scale=0.16]{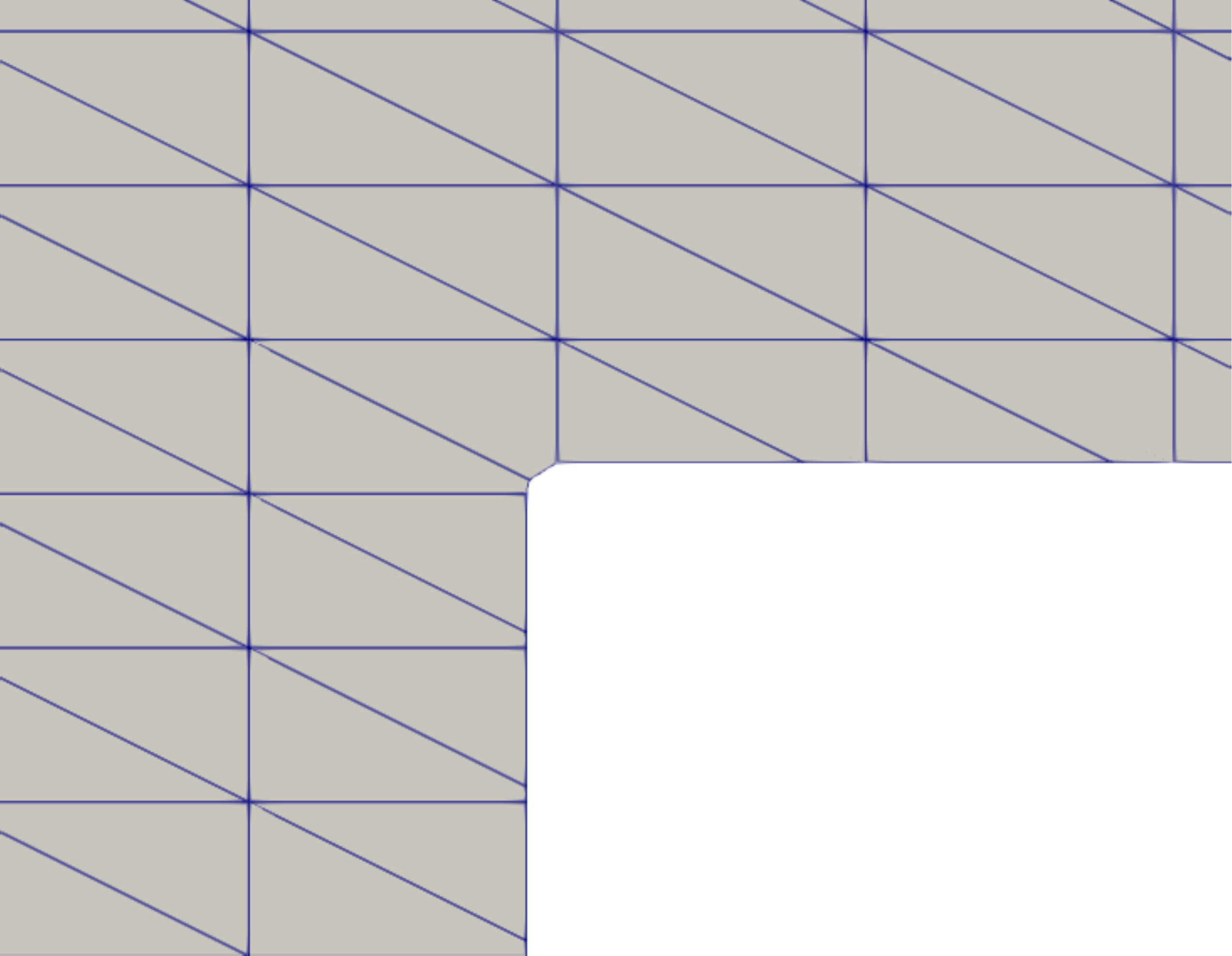}
	\end{subfigure}
	\caption{ 
		Example 2: The uniform unfitted mesh. In the right panel a zoom of the unfitted part $\Gamma=\{(x,y)\in\partial\Omega: x=2 \text{ or } y=1\}$ is shown.
		\label{fig:presrob_unfitted_mesh} }
\end{figure}
\begin{figure}[ht]  	
	\centering	
	\begin{subfigure}[b]{0.45 \textwidth}  	
		\centering	
		\includegraphics[width=\textwidth]{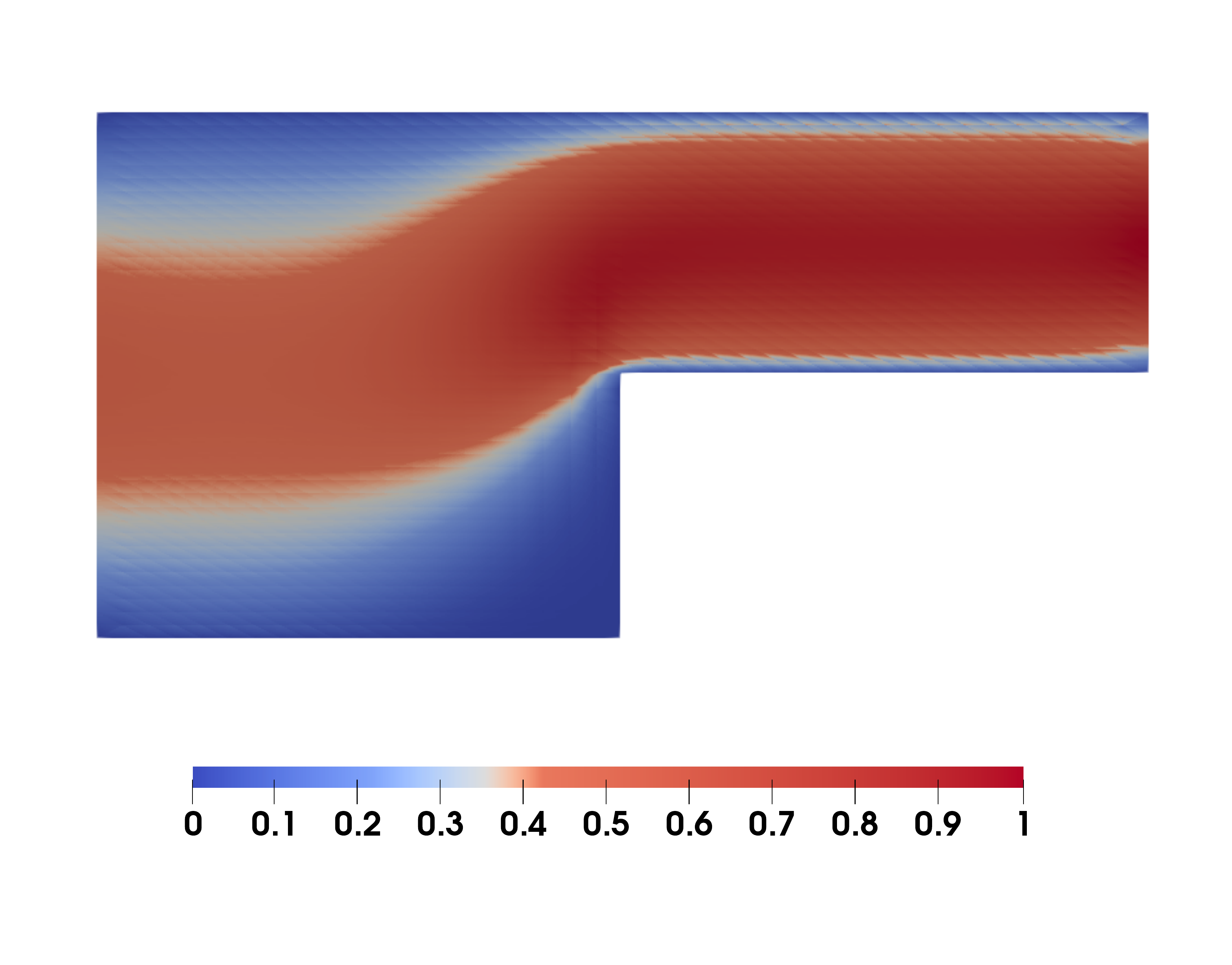}
		\end{subfigure}\quad\quad\quad
	\begin{subfigure}[b]{0.45 \textwidth}  	
		\centering	
		\includegraphics[width=\textwidth]{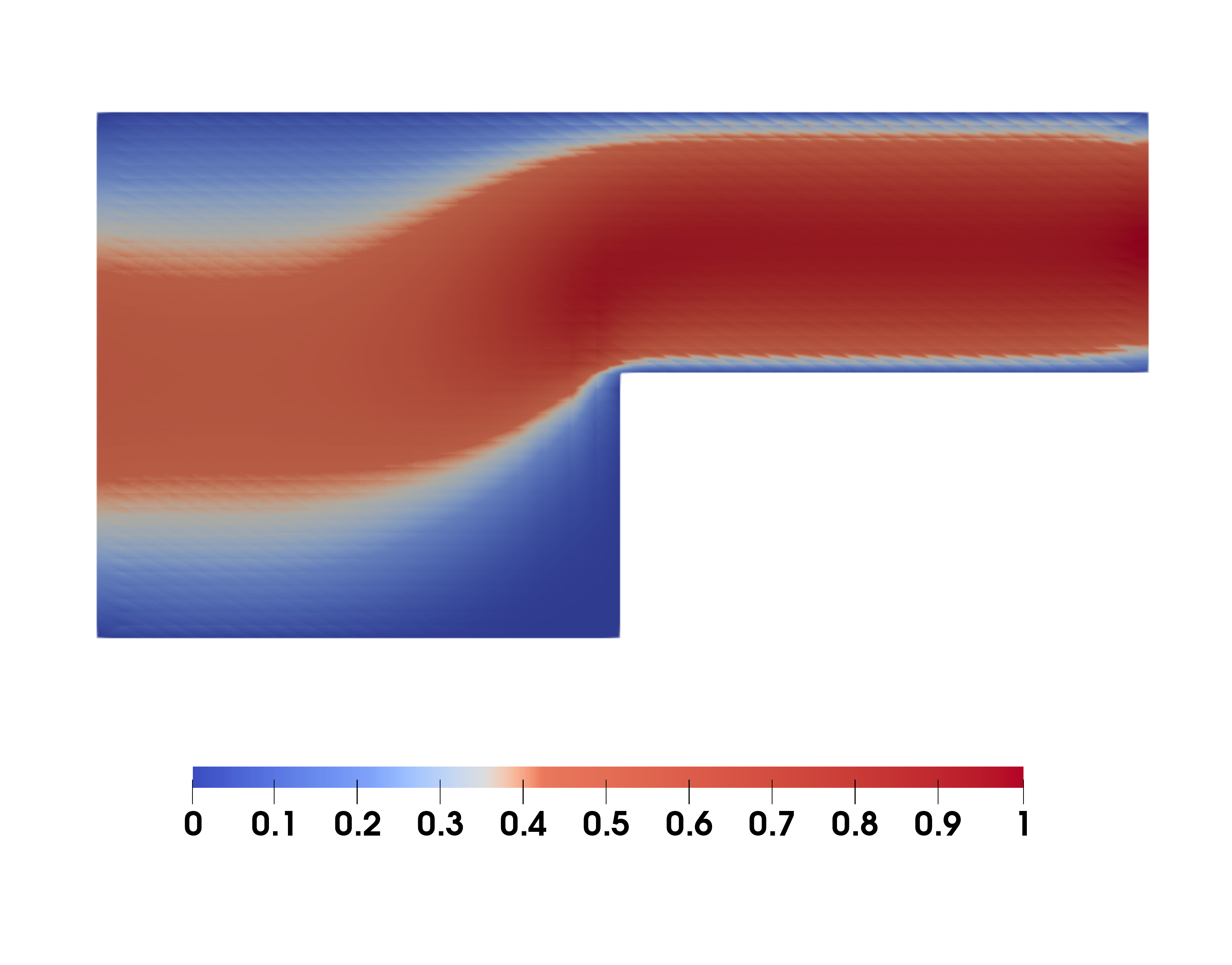}
	\end{subfigure}\\
	\begin{subfigure}[b]{0.45 \textwidth}  	
		\centering	
		\includegraphics[width=\textwidth]{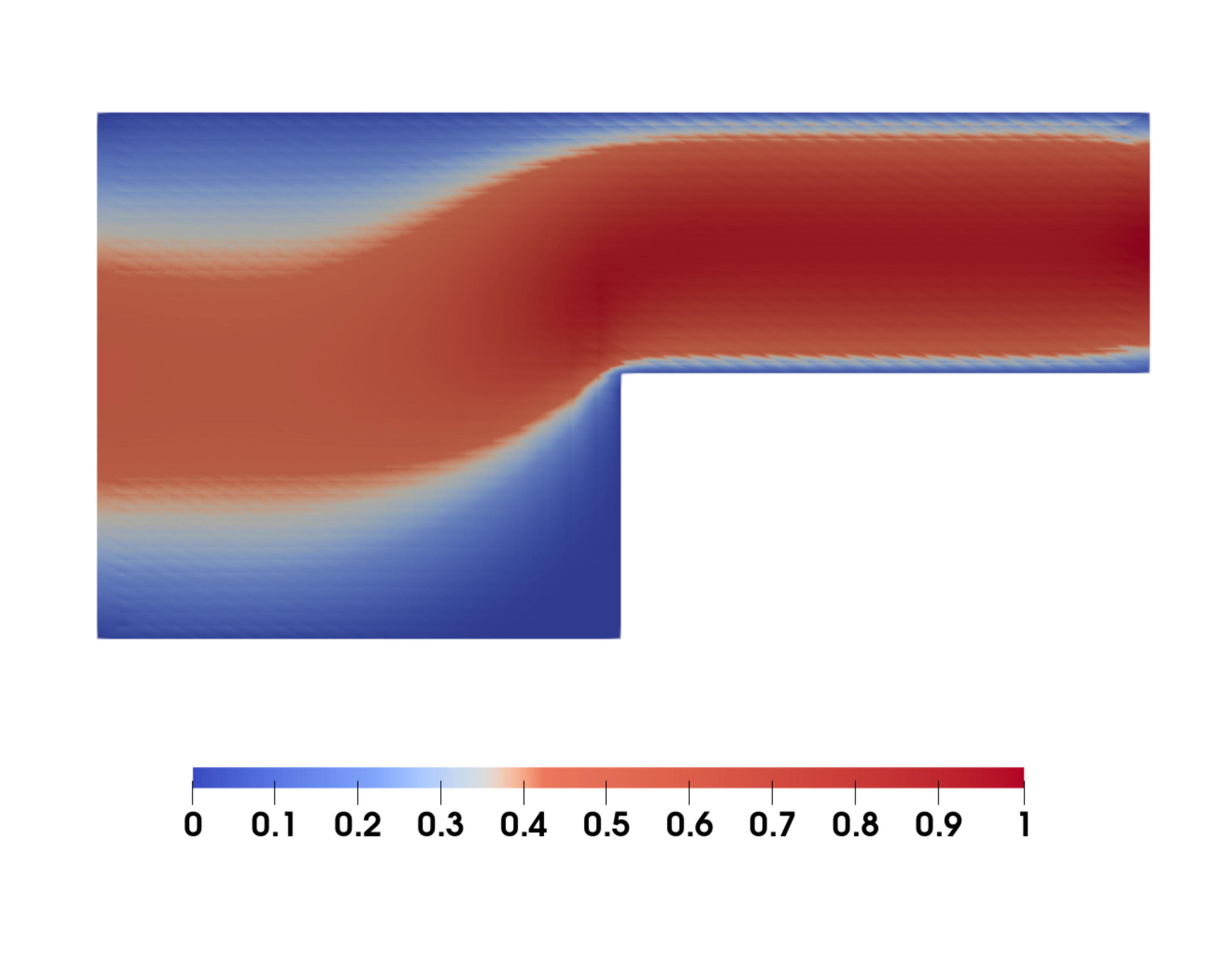}
	\end{subfigure}\quad\quad\quad
	\begin{subfigure}[b]{0.45 \textwidth}  	
		\centering	
		\includegraphics[width=\textwidth]{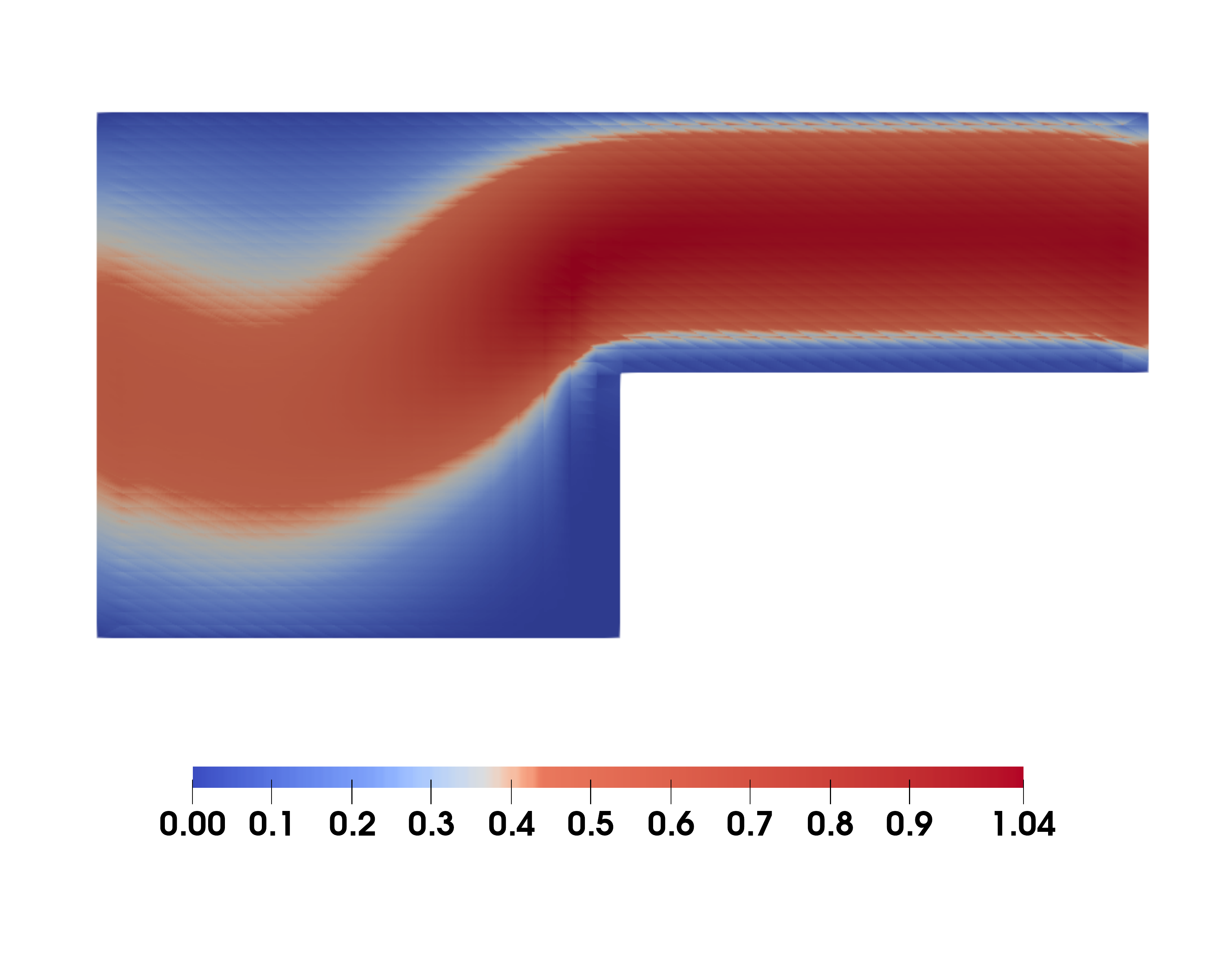}
	\end{subfigure}
	\caption{ 
		Example 2: The magnitude of the approximate velocity field using Method NC ($\BDM_1\times Q_0$), $\Lambda=0$ in the left panels and $\Lambda=2000$ in the right panels. 
		Top: Results using the proposed method.
		Bottom: Results using the standard approach. 
		\label{fig:presrob_BDM1_unfitted} } 
\end{figure}


Of greatest detrimental effect to pressure robustness in this example seems to be the changing of the stabilization $s_b$ to $s_p,$ that is Modification 2 from Section \ref{sec:num_lagr}. Modification 1 changes the magnitude of the divergence error by a smaller amount, and the macroscopic behaviour is not altered as much.

\subsection{Example 3 - No-flow, pressure robustness}\label{sec:num_ex3}
We consider another benchmark from \cite[Example 1.1]{VolkerDivConstr2017} which proves particularly tough for weak imposition of Dirichlet boundary conditions. 
We take $\mu=1$ and the domain is the unit square $\Omega=[0,1]^2$. In order to study our unfitted discretizations we take $\Omega_0=[0,1]\times [0,1+\epsilon]$, $\epsilon=10^{-12}$, and define $\Omega=[0,1]^2$ with the unfitted boundary $\Gamma=\{(x,y): y=1\}$. The boundary conditions are no-slip, i.e. $\bfu=\mathbf{0}$ on $\partial\Omega$ and datum 
\begin{align}
	\mathbf{f}=(0,\text{Ra}(1 - y + 3y^2)), \quad \text{in } \Omega,
\end{align}
where $\text{Ra}>0$. The exact solution is $\bfu=\mathbf{0}$ and $p=\text{Ra}(y^3 - y^2/2 + y - 7/12)$, whereby a change in the parameter $\text{Ra}$ changes only the pressure. A pressure robust numerical method should retain this property. 

We consider $\text{Ra}\in\{10^2,10^4,10^6\}$ and study the two ways of imposing the normal component of the Dirichlet boundary conditions weakly on the unfitted boundary; with Nitsche's method and with a Lagrange multiplier variable. We impose the boundary conditions strongly on fitted boundary edges for every method considered here.
For Method NC 
we pick the pair $\BDM_1\times Q_0$. For Method C1 we pick the triple $P_1\times\RT_0\times Q_0$ and so $P_1\times\RT_0\times Q_0\times Q^{\Gamma}_0$ is chosen for Method C2. A heatmap of the pressure solution given by Method C1 with Ra$=10^6$ is shown in Figure \ref{fig:ex3_pressure}. For each method we compare two alternatives: our proposed method versus the standard approach. 
In Method C2 we only exchange the $s_b$-forms with $s_p$.
See Figure \ref{fig:presrob2} for the $\bfH_0^1$-error of the computed solutions.

\begin{figure}[ht]  	
	\centering	
	\includegraphics[scale=0.2]{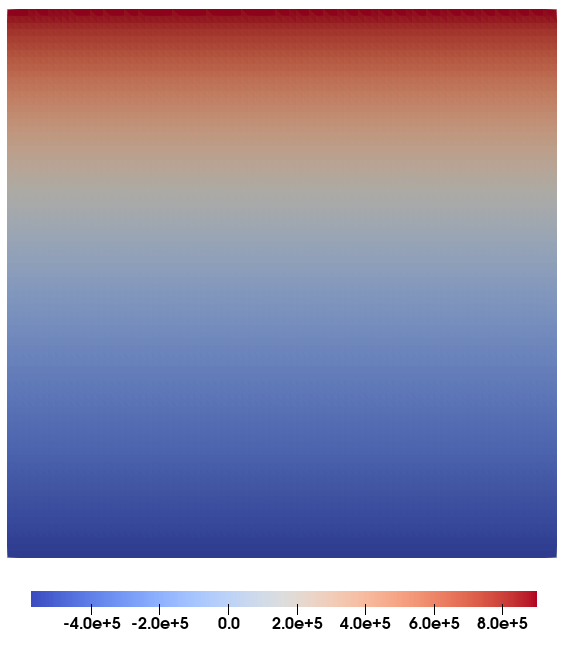}
	\caption{ 
		Example 3: Heatmap of the pressure using Method C1 ($P_1\times\RT_0\times Q_0$) for $h=0.0125$ and Ra$=10^6$.
		\label{fig:ex3_pressure} }
\end{figure}

\begin{figure}[ht]  	
	\centering	
	\begin{subfigure}[b]{0.3 \textwidth}  	
		\centering	
		\includegraphics[width=\textwidth]{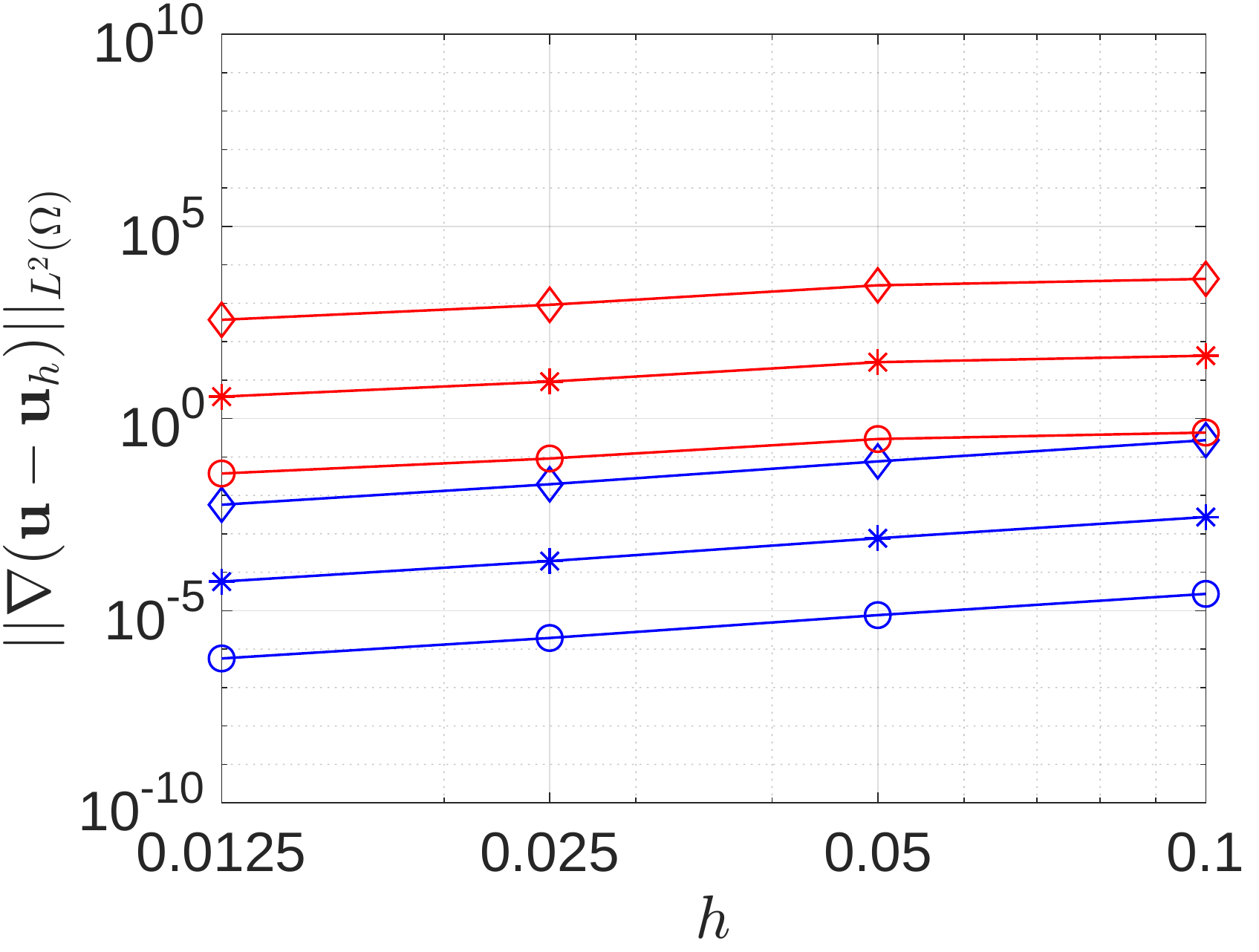}
	\end{subfigure}
	\begin{subfigure}[b]{0.3 \textwidth}  	
		\centering	
		\includegraphics[width=\textwidth]{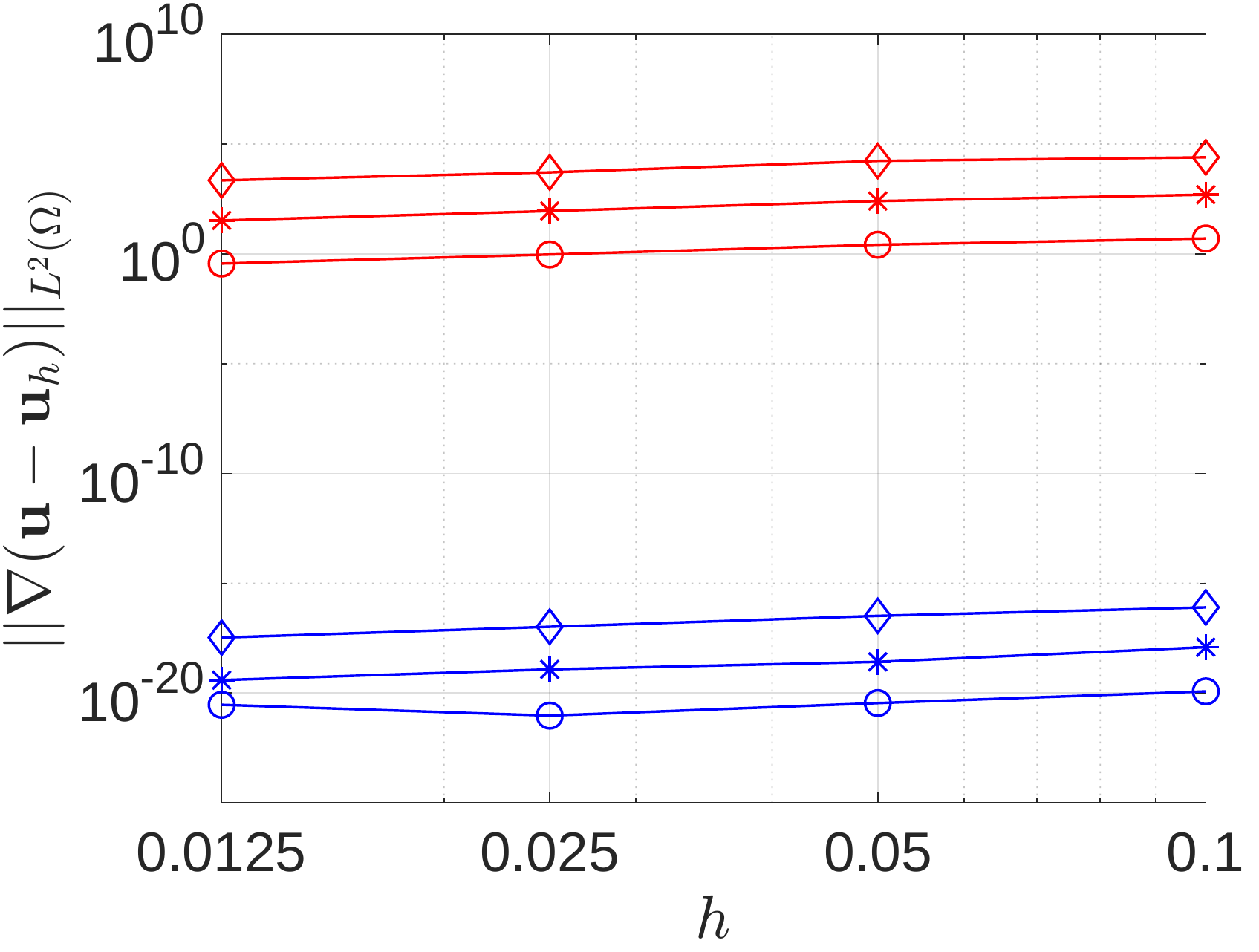}
	\end{subfigure}
	\begin{subfigure}[b]{0.3 \textwidth}  	
		\centering	
		\includegraphics[width=\textwidth]{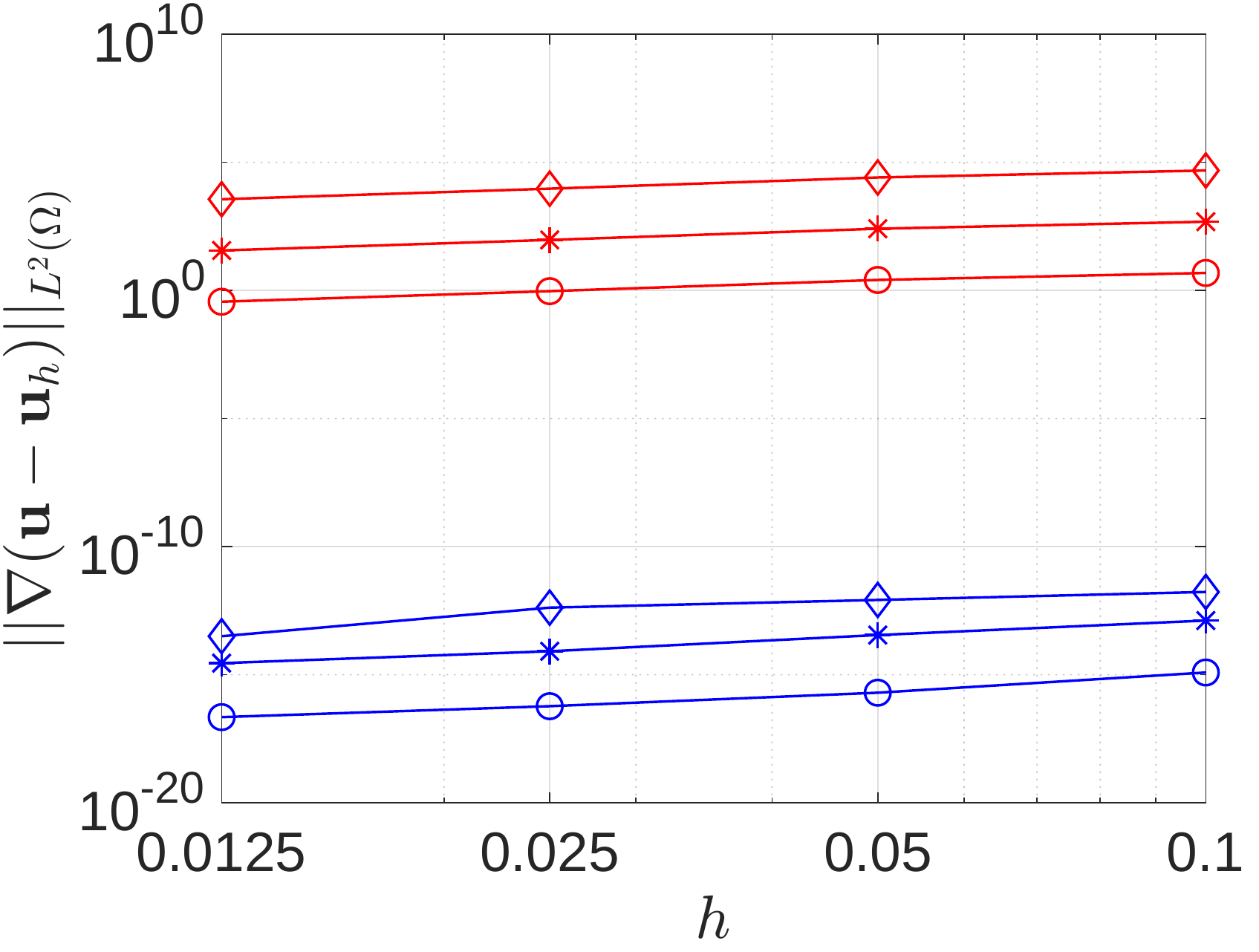}
	\end{subfigure}
	\includegraphics[scale=0.20]{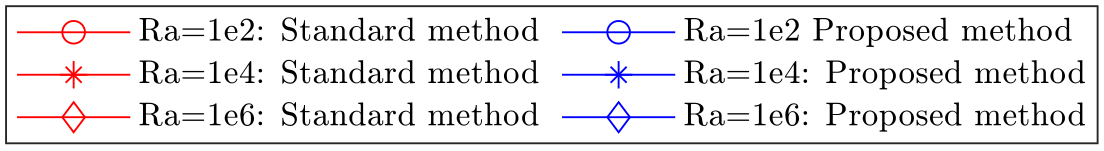}
	\caption{ 
		Example 3: The error of the gradient of the velocity. The proposed methods are compared to the standard approach. 
		Left: Method NC ($\BDM_1\times Q_0$). Middle: Method C1 ($P_1\times\RT_0\times Q_0$). Right: Method C2 ($P_1\times\RT_0\times Q_0\times Q^{\Gamma}_0$).
		\label{fig:presrob2} }            
\end{figure}

We see that it is preferable to use the proposed methods in favor of the standard approach. For all methods the error is reduced by several orders of magnitude. 
Both versions of Method C have a small error of the order of $10^{-18}$ and $10^{-14}$, respectively. 
The standard approach is not pressure robust, and the error in the gradient of the velocity increases by the same order of magnitude as Ra. The error for different Ra is also different for the proposed methods, although the errors between different values of Ra seem to get closer as $h\to 0$ for Method C1 and C2, see middle and right panels of Figure \ref{fig:presrob2}.

For Method NC, the error in the divergence is of the order of machine epsilon while the velocity error in the $\bfH^1_0$-norm is large. In order to reduce the error coming from the boundary conditions we choose to scale the boundary penalty parameters $\lambda_{\bfu}$ and $\lambda_{u_n}$ with the parameter Ra, i.e., $\lambda_{\bfu}=\lambda_{u_n}=\text{Ra}\cdot 10^5$. This is a huge number and not necessarily optimal, but it is sufficient to illustrate that when the error from the boundary is small, Method NC is pressure robust, see Figure \ref{fig:presrob2_scale} for the corresponding $\bfH^1$-errors. For Method C1, there is little change however. The two methods behave very differently in this example.

\begin{figure}[ht]  	
	\centering	
	\begin{subfigure}[b]{0.3 \textwidth}  	
		\centering	
		\includegraphics[width=\textwidth]{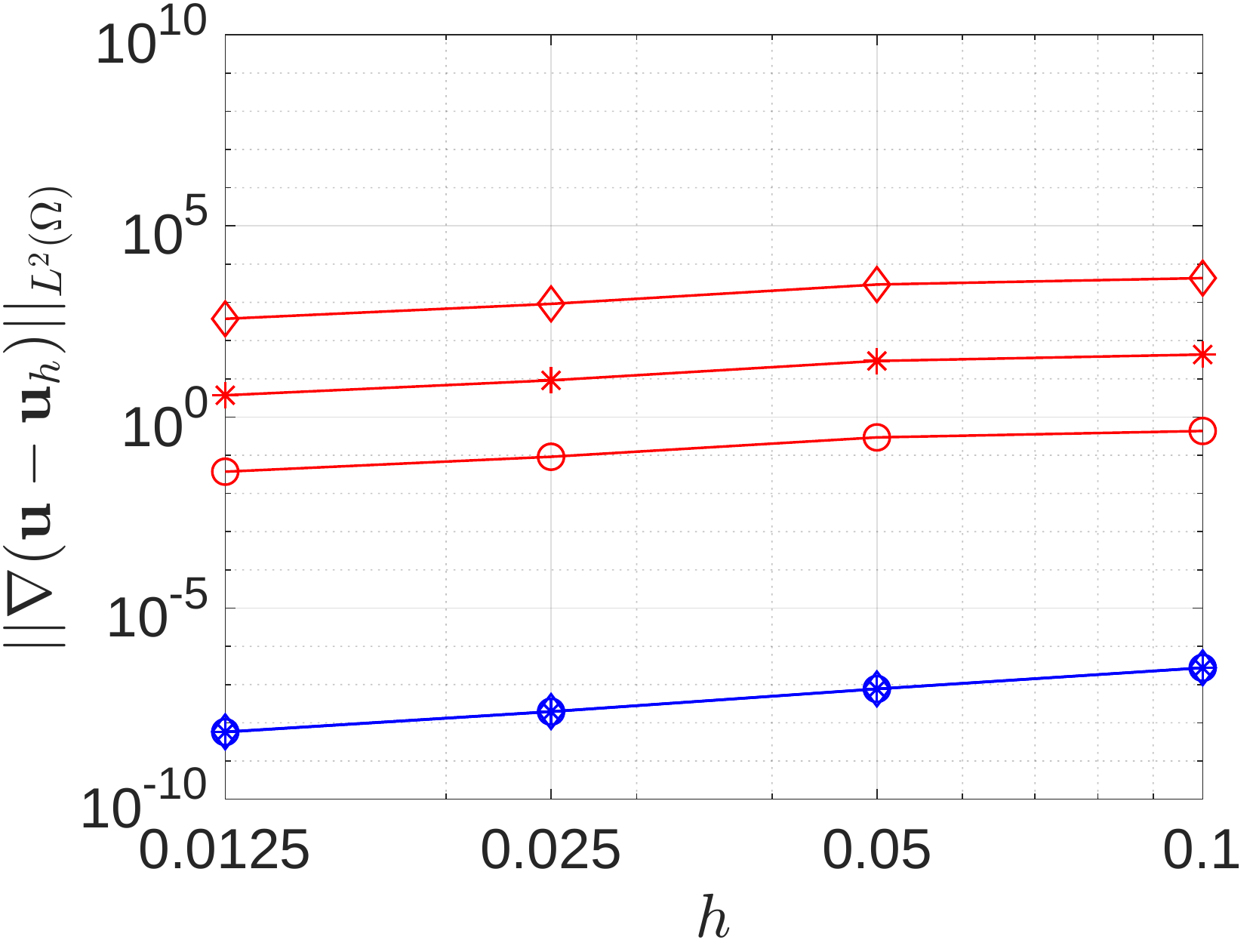}
	\end{subfigure}
	\begin{subfigure}[b]{0.3 \textwidth}  	
		\centering	
		\includegraphics[width=\textwidth]{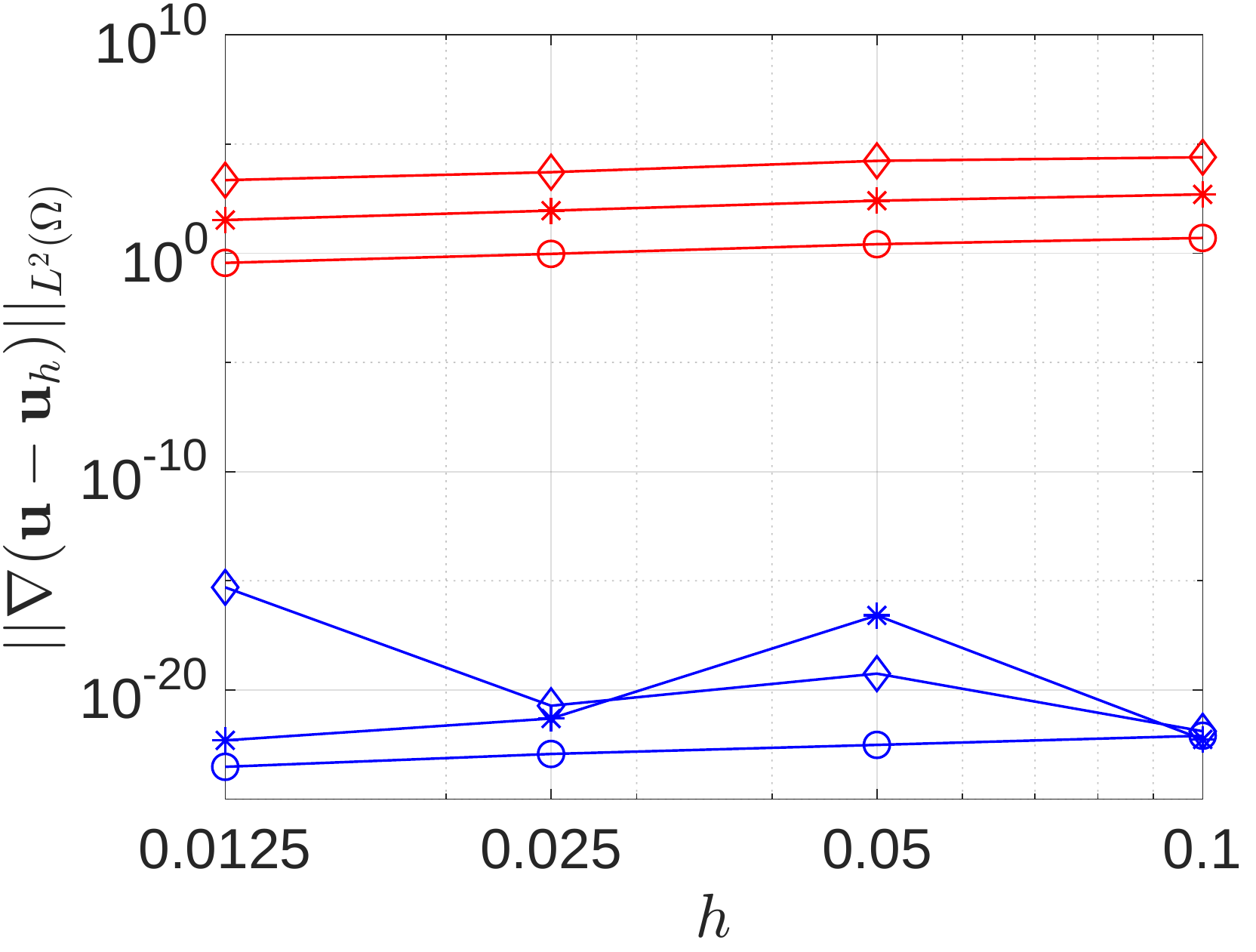}
	\end{subfigure}
	\includegraphics[scale=0.25]{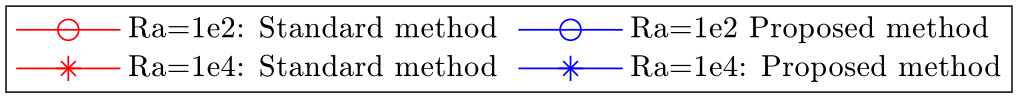}
	\caption{ 
		Example 3: The error of the gradient of the velocity, with $\lambda_{\bfu}=\lambda_{u_n}=10^5 \cdot \text{Ra}$. Proposed methods are compared to the standard approach. 
		Left: Method NC ($\BDM_1\times Q_0$). Right: Method C1 ($P_1\times\RT_0\times Q_0$). 
		\label{fig:presrob2_scale} }
\end{figure}

When we in Method NC scale $\lambda_{\bfu}$ with Ra, the condition numbers will also scale with Ra, see Table \ref{tab:ex3_cond}. For Method C2, the condition numbers are the same irrespective of Ra (since we don't scale the Lagrange multiplier variable with $Ra$).
\begin{table}[ht]
    \centering
    \begin{tabular}{|l|l|l|l|}
        \hline
        $h$     & \begin{tabular}[c]{@{}l@{}}Method NC: \\ Ra$=10^2$\end{tabular} & \begin{tabular}[c]{@{}l@{}}Method NC: \\ Ra$=10^4$\end{tabular} & \begin{tabular}[c]{@{}l@{}}Method NC: \\ Ra$=10^6$\end{tabular} \\ \hline
        0.1     & 5.64E+15 & 5.64E+19 & 5.64E+23 \\ 
        0.05    & 1.11E+16 & 1.11E+20 & 1.11E+24 \\ 
        0.025   & 2.21E+16 & 2.21E+20 & 2.21E+24 \\ 
        0.0125  & 4.41E+16 & 4.41E+20 & 4.41E+24 \\ \hline
    \end{tabular}
    \vspace{0.1cm}
    
    \begin{tabular}{|l|l|l|l|}
        \hline
        $h$     & \begin{tabular}[c]{@{}l@{}}Method C1: \\ Ra$=10^2$\end{tabular} & \begin{tabular}[c]{@{}l@{}}Method C1: \\ Ra$=10^4$\end{tabular} & \begin{tabular}[c]{@{}l@{}}Method C1: \\ Ra$=10^6$\end{tabular} \\ \hline
        0.1     & 1.00E+11 & 1.00E+11 & 1.00E+11 \\ 
        0.05    & 2.02E+11 & 2.02E+11 & 2.02E+11 \\ 
        0.025   & 4.08E+11 & 4.08E+11 & 4.08E+11 \\ 
        0.0125  & 8.36E+11 & 8.36E+11 & 8.36E+11 \\ \hline
    \end{tabular}
    \vspace{0.1cm}
    
    \begin{tabular}{|l|l|l|l|}
        \hline
        $h$     & \begin{tabular}[c]{@{}l@{}}Method C2: \\ Ra$=10^2$\end{tabular} & \begin{tabular}[c]{@{}l@{}}Method C2: \\ Ra$=10^4$\end{tabular} & \begin{tabular}[c]{@{}l@{}}Method C2: \\ Ra$=10^6$\end{tabular} \\ \hline
        0.1     & 5.33E+03 & 5.33E+03 & 3.14E+03 \\ 
        0.05    & 2.15E+04 & 2.15E+04 & 2.15E+04 \\ 
        0.025   & 6.98E+04 & 4.78E+04 & 8.97E+04 \\ 
        0.0125  & 3.85E+05 & 3.09E+05 & 3.85E+05 \\ \hline
    \end{tabular}
    \caption{Example 3: Condition numbers for different methods and values of Ra. Method NC ($\BDM_1\times Q_0$) has $\lambda_{\bfu}=10^5 \text{Ra}$, Method C1 ($P_1\times\RT_0\times Q_0$) has $\lambda_{\bfu}=10^5$. Method C2 uses $P_1\times\RT_0\times Q_0\times Q^{\Gamma}_0$ and no penalty.}
    \label{tab:ex3_cond}
\end{table}

\subsubsection{Is incompressibility sufficient for pressure robustness?}\label{sec:numex_eyenorm}
Our proposed methods satisfy the incompressibility condition pointwise up to machine precision. A natural question is whether it is enough to have a pointwise divergence-free velocity field (up to machine precision) in order to achieve pressure robustness when imposing Dirichlet boundary conditions weakly?

In the tables below we label the condition number by $\kappa$. We consider Example 3, and the lowest order discretizations of our proposed methods. 
We compare Method NC and Method C1 in Table \ref{tab:weakweak} (and Table \ref{tab:weakweak_fitted}), while results for Method C2 are presented in Table \ref{tab:weakstrong}. We can see that satisfying the incompressibility condition is not sufficient to achieve pressure robustness if the normal component of the velocity field at the boundary is not prescribed accurately enough. 
We can increase the penalty parameter, however, increasing the penalty increases the condition number. Therefore there is a trade-off: how big can you make penalty, thereby reducing the error from the boundary, before the condition number becomes too big, thereby increasing the error in the solution. 
\begin{table}[ht]
	\resizebox{\textwidth}{!}{
	\begin{tabular}{l|llll|llll}
	& \begin{tabular}[c]{@{}l@{}}Method NC,\\ $h=0.025$\end{tabular} & & &  & \begin{tabular}[c]{@{}l@{}}Method C1,\\ $h=0.025$\end{tabular} & & & \\ \cline{2-9} 
	& \multicolumn{1}{l|}{$ \| \nabla\cdot\mathbf{u}_h\|_{L^{\infty}(\Omega)}$} & \multicolumn{1}{l|}{$ \| \mathbf{u}_h\|_{L^{2}(\Omega)} $} & \multicolumn{1}{l|}{$ \| \nabla\mathbf{u}_h\|_{L^{2}(\Omega)}$} & $\kappa$ & \multicolumn{1}{l|}{$ \| \nabla\cdot\mathbf{u}_h\|_{L^{\infty}(\Omega)}$} & \multicolumn{1}{l|}{$ \| \mathbf{u}_h\|_{L^{2}(\Omega)} $} & \multicolumn{1}{l|}{$ \| \nabla\mathbf{u}_h\|_{L^{2}(\Omega)}$} & \multicolumn{1}{l|}{$\kappa$} \\ \hline
	\multicolumn{1}{|l|}{\begin{tabular}[c]{@{}l@{}}$\lambda_{\mathbf{u}}=\lambda_{u_n}=10^5$\\ \textbf{Ra$=\mathbf{10^2}$}\end{tabular}}  & \multicolumn{1}{l|}{5.0822E-20} & \multicolumn{1}{l|}{4.0794E-08} & \multicolumn{1}{l|}{1.9512E-06} & 2.23E+12 & \multicolumn{1}{l|}{2.2023E-20} & \multicolumn{1}{l|}{3.6336E-08} & \multicolumn{1}{l|}{9.3563E-22} & \multicolumn{1}{l|}{4.08E+11}         \\ \hline
	\multicolumn{1}{|l|}{\begin{tabular}[c]{@{}l@{}}$\lambda_{\mathbf{u}}=\lambda_{u_n}=10^5$\\ \textbf{Ra$=\mathbf{10^6}$}\end{tabular}}  & \multicolumn{1}{l|}{3.3307E-16} & \multicolumn{1}{l|}{0.0004}  & \multicolumn{1}{l|}{0.0195}  & 2.23E+12 & \multicolumn{1}{l|}{3.7470E-16}   & \multicolumn{1}{l|}{0.00036} & \multicolumn{1}{l|}{1.0408E-17} & \multicolumn{1}{l|}{4.08E+11}         \\ \hline\hline
	\multicolumn{1}{|l|}{\begin{tabular}[c]{@{}l@{}}$\lambda_{\mathbf{u}}=\lambda_{u_n}=10^5$Ra\\ \textbf{Ra$=\mathbf{10^2}$}\end{tabular}} & \multicolumn{1}{l|}{2.3823E-22} & \multicolumn{1}{l|}{4.0799E-10} & \multicolumn{1}{l|}{1.9521E-08} & 2.21E+16 & \multicolumn{1}{l|}{5.0293E-22} & \multicolumn{1}{l|}{3.6334E-10} & \multicolumn{1}{l|}{1.1807E-23} & \multicolumn{1}{l|}{4.00E+15}         \\ \hline
	\multicolumn{1}{|l|}{\begin{tabular}[c]{@{}l@{}}$\lambda_{\mathbf{u}}=\lambda_{u_n}=10^5$Ra\\ \textbf{Ra$=\mathbf{10^6}$}\end{tabular}} & \multicolumn{1}{l|}{6.6370E-15} & \multicolumn{1}{l|}{4.0890E-10} & \multicolumn{1}{l|}{1.9716E-08} & 2.21E+24 & \multicolumn{1}{l|}{1.3752E-19} & \multicolumn{1}{l|}{5.2968E-10} & \multicolumn{1}{l|}{1.9054E-21} & \multicolumn{1}{l|}{4.00E+23}         \\ \hline
	\end{tabular}  }
	\caption{Example 3: Results for Method NC ($\BDM_1\times Q_0$) and Method C1 ($P_1\times\RT_0\times Q_0$) for $h=0.025$. The boundary condition on $\Gamma$ is imposed weakly via penalty. }\label{tab:weakweak}
\end{table}
(It is interesting to compare our results on the unfitted mesh (with respect to the line $y=1$) to results on a fitted mesh. We consider the same Method NC and Method C1, imposing boundary conditions weakly but on a fitted mesh (on the entire boundary.) The condition numbers are smaller, but the trend is the same as for the unfitted case of Table \ref{tab:weakweak}; scaling with Ra improves the error from the boundary, but increases the condition number. We refer to Table \ref{tab:weakweak_fitted} for these results.

In Table \ref{tab:weakstrong} we show the results with Method C2 where we impose the boundary conditions weakly on $\Gamma$ via the Lagrange multiplier variable. We see that the $L^2$-error of the velocity is greatly affected by what space we pick for our Lagrange multiplier variable, and especially we see that $Q_0^{\partial\Omega}$ is not sufficient. 
The best results are obtained when we also scale the test and trial Lagrange multiplier variables as $(\xi,\chi)\mapsto (\text{Ra}\xi,\text{Ra}\chi)$. This choice has a similar effect as the choice of scaling the penalty parameter $\lambda_{u_n}$ with Ra which we used in the previous table.
The condition numbers are in general smaller than in Table \ref{tab:weakweak}, where a penalty parameter is used, but we can see the same trend when scaling $(\xi,\chi)\mapsto (\text{Ra}\xi,\text{Ra}\chi)$. 
\begin{table}[ht]
	\centering
	\begin{tabular}{l|llll} & Method C2, $h=0.025$ & & & \\ \cline{2-5}
	& \multicolumn{1}{l|}{$\| \nabla\cdot\mathbf{u}_h\|_{L^{\infty}(\Omega)}$} & \multicolumn{1}{l|}{$ \| \mathbf{u}_h\|_{L^{2}(\Omega)} $} & \multicolumn{1}{l|}{$ \| \nabla\mathbf{u}_h\|_{L^{2}(\Omega)}$} & \multicolumn{1}{l|}{$\kappa$}  \\ \hline 
	\multicolumn{1}{|l|}{ $P_1\times \RT_0\times Q_0\times Q^{\Gamma}_0$: \textbf{Ra$=\mathbf{10^2}$}} & \multicolumn{1}{l|}{1.0658e-14} & \multicolumn{1}{l|}{0.00896}  & \multicolumn{1}{l|}{2.0365e-15} & \multicolumn{1}{l|}{6.98e+04} \\
	\multicolumn{1}{|l|}{ $P_1\times \RT_0\times Q_0\times Q^{\Gamma}_0$: \textbf{Ra$=\mathbf{10^6}$}} & \multicolumn{1}{l|}{8.7312e-11} & \multicolumn{1}{l|}{89.6156}     & \multicolumn{1}{l|}{2.2160e-11} & \multicolumn{1}{l|}{8.97e+04} \\ \hline 
	\multicolumn{1}{|l|}{\begin{tabular}[c]{@{}l@{}} $P_1\times \RT_0\times Q_0\times Q^{\Gamma}_0$ \\ and $(\xi,\chi)\mapsto $Ra$(\xi,\chi)$: \textbf{Ra$=\mathbf{10^2}$}\end{tabular}} & \multicolumn{1}{l|}{7.1054e-15} & \multicolumn{1}{l|}{0.00665}  & \multicolumn{1}{l|}{1.1309e-16} & \multicolumn{1}{l|}{1.55e+05} \\
	\multicolumn{1}{|l|}{\begin{tabular}[c]{@{}l@{}} $P_1\times \RT_0\times Q_0\times Q^{\Gamma}_0$ \\ and $(\xi,\chi)\mapsto $Ra$(\xi,\chi)$: \textbf{Ra$=\mathbf{10^6}$}\end{tabular}} & \multicolumn{1}{l|}{5.8208e-11} & \multicolumn{1}{l|}{66.5144}     & \multicolumn{1}{l|}{1.1628e-12} & \multicolumn{1}{l|}{2.84e+09} \\ \hline \hline
	\multicolumn{1}{|l|}{ $P_1\times \RT_0\times Q_0\times Q^{\Gamma}_1$: \textbf{Ra$=\mathbf{10^2}$}} & \multicolumn{1}{l|}{1.2534e-13} & \multicolumn{1}{l|}{1.3242e-13} & \multicolumn{1}{l|}{8.8632e-14} & \multicolumn{1}{l|}{6.48e+05} \\
	\multicolumn{1}{|l|}{ $P_1\times \RT_0\times Q_0\times Q^{\Gamma}_1$: \textbf{Ra$=\mathbf{10^6}$}} & \multicolumn{1}{l|}{1.0758e-09} & \multicolumn{1}{l|}{1.1077e-09} & \multicolumn{1}{l|}{7.6073e-10} & \multicolumn{1}{l|}{6.48e+05} \\ \hline 
	\multicolumn{1}{|l|}{\begin{tabular}[c]{@{}l@{}} $P_1\times \RT_0\times Q_0\times Q^{\Gamma}_1$ \\ and $(\xi,\chi)\mapsto $Ra$(\xi,\chi)$: \textbf{Ra$=\mathbf{10^2}$}\end{tabular}} & \multicolumn{1}{l|}{7.0786e-16} & \multicolumn{1}{l|}{8.9933e-16} & \multicolumn{1}{l|}{5.0053e-16} & \multicolumn{1}{l|}{4.17e+07}         \\
	\multicolumn{1}{|l|}{\begin{tabular}[c]{@{}l@{}} $P_1\times \RT_0\times Q_0\times Q^{\Gamma}_1$ \\ and $(\xi,\chi)\mapsto $Ra$(\xi,\chi)$: \textbf{Ra$=\mathbf{10^6}$}\end{tabular}} & \multicolumn{1}{l|}{1.0959e-15} & \multicolumn{1}{l|}{3.7830e-12} & \multicolumn{1}{l|}{7.7489e-16} & \multicolumn{1}{l|}{4.14e+11}         \\ \hline
	\end{tabular}
	\caption{Example 3: Results for Method C2 ($P_1\times\RT_0\times Q_0\times Q^{\Gamma}_0$) for $h=0.025$. The boundary condition $\bfu\cdot\bfn=\bfg\cdot\bfn$ on $\Gamma$ is imposed weakly via the Lagrange multiplier variable. }\label{tab:weakstrong}
\end{table}

\subsection{Method C1 in three space dimensions} \label{sec:mC_3D}
The description of Method C changes only slightly when going up one dimension to three space dimensions. The exact sequence picks up an extra intermediate space:
\begin{alignat}{2}
  &[2D]\ \quad\quad\quad  0 &&\hookrightarrow H^1(\Omega) \overset{\crl}{\to} \bfH^{\dive}(\Omega) \overset{\dive}{\to} L^2(\Omega) \to 0 \\
  &[3D]\ \quad\quad\quad 0 &&\hookrightarrow H^1(\Omega) \overset{\nabla}{\to} \bfH^{\crl}(\Omega) \overset{\crl}{\to} \bfH^{\dive}(\Omega) \overset{\dive}{\to} L^2(\Omega) \to 0 \label{eq:3Dcomplex}
\end{alignat}
Setting $\mathcal{H}(\Omega):=\bfH^{\crl}(\Omega)$ with $\crl\bfu:=\nabla\times\bfu$, the method reads identically up to changing edges to faces and $\bfv\cdot\bft \mapsto \bfn\times\bfv$. One exchanges $W_{k+1,h}$ with vector-valued Nédelec edge elements of the first kind, see \cite{defelementDeRham,PTFE}.

The boundary conditions \eqref{eq:altBC0} and \eqref{eq:altBC} change to
\begin{align}
  &\bfn\times \pmb{\omega} = \pmb{\omega}_0,\ \bfu\cdot\bfn = u_0, \text{ or }\\
  &\bfn\times\bfu = \bfu_0,\ p=p_0.
\end{align}
We illustrate Method C1 in three space dimensions via one example here. 

Let $\Omega = B_{2/3}(\mathbf{0})$ be the ball of radius $2/3$ centered at $\mathbf{0}$, and let $\bfu = (z/2,0,0)^T$ be the velocity solution to the Stokes problem with $\mu=1,\ \bff=\mathbf{0}$ and inhomogeneous Dirichlet boundary conditions. With $\mathbf{N}_0\subset \bfH^{\crl}(\Omdh)$ denoting the lowest order Nedelec element of the first kind, we use the triple $\mathbf{N}_0\times\RT_0\times Q_0$ with $\tau_m=\tau_a=\tau_b=1$ and full stabilization. Then Method C produces the velocity as seen in Figure \ref{fig:ex3D}. We see that the divergence free property of the Raviart-Thomas elements are kept by the proposed unfitted discretization and the magnitude of the error in $\dive \bfu$ is of the order of machine epsilon.  
\begin{figure}[ht]
  \centering
  \begin{subfigure}{0.4\textwidth}
    \includegraphics[width=\textwidth]{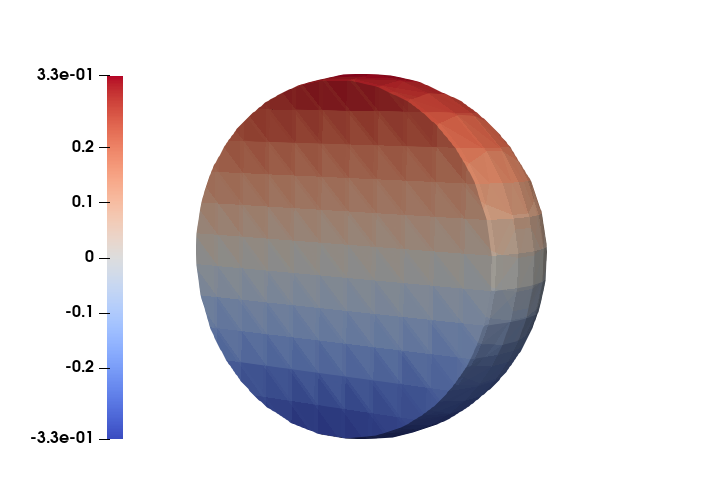}
  \end{subfigure} 
  \hfill
  \begin{subfigure}{0.38\textwidth}
    \includegraphics[width=\textwidth]{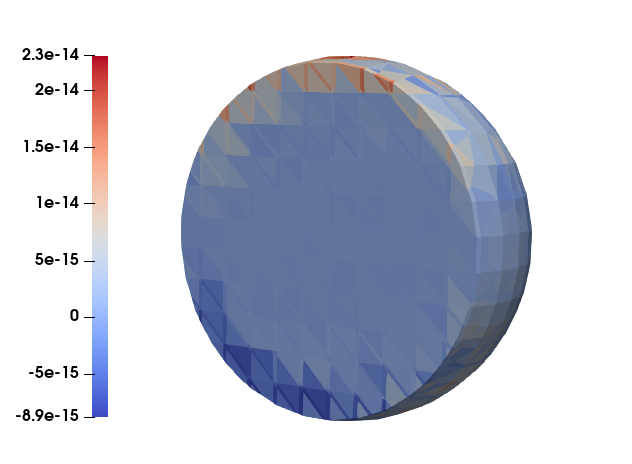}
  \end{subfigure}
  \caption{ 
    Example 3D: Method C1 ($\mathbf{N}_0\times\RT_0\times Q_0$) in 3D.
    Left: $x$-component of the velocity. Right: divergence of the velocity.
    \label{fig:ex3D} }                     
\end{figure}

\section{Conclusion}\label{sec:conclusion}
We developed and studied several cut finite element discretizations for the Stokes equations that result in divergence-free velocity approximations. We presented an unfitted discretization, Method NC, based on the non-conforming method of \cite{WangHdiv2007} where essential boundary conditions where enforced using Nitsche's method, see Section~\ref{sec:method_wang}. We also presented two unfitted discretizations for the vorticity-velocity-pressure formulation of the Stokes problem, that we refer to as Method C1 and Method C2, see Section \ref{sec:method_vort}. We obtained best results with Method C2 where $\bfu \cdot \bfn$ was imposed using a Lagrange multiplier method compared to Method C1 where a penalty parameter was used. We obtained optimal convergence order both for the velocity and the pressure, with pointwise divergence-free velocity approximations for Method NC with $\BDM$-elements and for Method C2 with $\RT$-elements. The condition number of the resulting linear systems scaled with respect to mesh size as for standard finite element methods. In our implementation, the assembly and solve of the linear systems associated to Method NC is slower than Method C. Method C has an extra variable for the vorticity (and yet another for Method C2), but Method NC has interior edge terms. With that being said, the assembly and solve speed of course depends on the implementation and factors such as CPU, GPU, type of solver, and so on.
 x
To guarantee divergence-free velocity approximations two conditions where of importance 1) we used ghost penalty stabilization terms that do not perturb the divergence condition; 2) The condition $\int_{\partial \Omega} \bfu \cdot \bfn=\int_{\partial \Omega} \bfg \cdot \bfn = 0$ has to be satisfied.
Our numerical experiments illustrated that a scheme which produces pointwise divergence-free velocity approximations is pressure robust in case $\bfu \cdot \bfn$ is imposed strongly, but not necessarily otherwise. Thus, the weak imposition of Dirichlet boundary conditions creates a challenge in the development of a pressure robust unfitted discretization. We found that when Nitsche's method is used the penalty parameter may have to be chosen very large in order to obtain pressure robustness. In the Lagrange multiplier method the stabilization term is of importance for stability, but perturbs the boundary condition. However, by choosing a larger space for the Lagrange multiplier; we choose $\tilde{k}=k+1$, improved the approximation at the boundary. We also showed that scaling the terms involving the Lagrange multipliers improved the errors from the boundary, though at the price of a larger condition number.

\paragraph{Conflict of interest statement} The authors declare no conflicts of interest. This research was supported by the Swedish Research Council Grant No. 2018-04192, 2022-04808, and the Wallenberg Academy Fellowship KAW 2019.0190.
 
\appendix
\section{: Standard methods}\label{sec:standardm}
We include a short appendix outlining the methods we compare the proposed methods with. In the initial works on CutFEM one stabilizes each unknown, the velocity and the pressure, with standard ghost penalty stabilization terms, and enforces Dirichlet boundary conditions weakly but without ensuring that $\int_{\partial \Omega} \bfu \cdot \bfn=0$. Hence when we either use standard ghost penalty stabilization for each unknown, or choose the test spaces as $\Vk \times \Qk^0$ while not ensuring that $\int_{\partial \Omega} \bfu \cdot \bfn=0$, we refer to the discretization as the standard approach. We remark that since the elements that were used in the early works on CutFEM did not preserve incompressibility condition pointwise these considerations were not of importance. 

Following the standard approach for Method NC gives: Find $(\bfu_h,p_h)\in \Vk \times \Qk^0$ such that 
\begin{alignat}{2}
    \Af(\bfu_h,\bfv_h) - b(\bfv_h,p_h) &= \FF (\bfv_h),\ &&\text{ for all } \bfv_h\in\Vk, \nonumber \\
     b_0(\bfu_h,q_h) + s_p(p_h,q_h) &= 0, &&\text{ for all } q_h\in\Qk^0. \label{eq:standard_mNC}
\end{alignat}
Here $\Af, b, \FF, s_p$ are defined as in \eqref{eq:Af}-\eqref{eq:f0}, and \eqref{eq:stab-p}. 
Following the standard approach for Method C1 gives: Find $(\omega_h,\bfu_h,p_h)\in \Wk\times \Vk\times \Qkk^0$ such that 
\begin{alignat}{2} 
	m(\omega_h,\phi_h) - c(\phi_h,\bfu_h) &= (\bfg\cdot\bft,\phi_h)_{\partial\Omega}, \ &&\forall \phi_h\in \Wk, \nonumber \\
	\check{a}(\bfu_h,\bfv_h) + c(\omega_h,\bfv_h) - b(\bfv_h,p_h) &= (\bff,\bfv_h)_{\Omega} + (\lambda_{\bfu_n}h^{-1}\bfg\cdot\bfn,\bfv_h\cdot\bfn)_{\partial\Omega},\ \quad &&\forall \bfv_h\in\Vk, \nonumber \\
	b_0(\bfu_h,q_h)+s_p(p_h,q_h) &= 0, \ &&\forall q_h\in \Qkk^0. \label{eq:standard_mC} 
\end{alignat}
Here $m, c, \check{a}$ refer to \eqref{eq:m}-\eqref{eq:MCpenalty}. 
Following the standard approach for Method C2 gives: Find $(\omega_h,\bfu_h,p_h,\xi_h)\in \Wk\times \Vkg\times \Qkk^0\times\Qkkg$ such that 
\begin{alignat}{2} 
	m(\omega_h,\phi_h) - c(\phi_h,\bfu_h) &= (\bfg\cdot\bft,\phi_h)_{\partial\Omega}, \ &&\forall \phi_h\in \Wk, \nonumber \\
	c(\omega_h,\bfv_h) - b_0(\bfv_h,p_h) + (\xi_h,\bfv_h\cdot\bfn)_{\Gamma} &= (\bff,\bfv_h)_{\Omega},\ \quad &&\forall \bfv_h\in\Vkz, \nonumber\\
	b_0(\bfu_h,q_h)+s_p(p_h,q_h) &= 0, \ &&\forall q_h\in \Qkk^0, \nonumber\\
	(\bfu_h\cdot\bfn,\chi_h)_{\Gamma}-s(\xi_h,\chi_h) &= (\bfg\cdot\bfn,\chi_h)_{\Gamma}, \ &&\forall \chi_h\in \Qkkg. \label{eq:standard_MClagmult}
\end{alignat}
Here $s$ refers to \eqref{eq:stablag}.

\section{: Auxiliary tables and figures}
In this appendix we include some auxiliary tables. They are referred to in the main text.

\begin{table}[ht]
    \centering
	\begin{tabular}{|l|l|l|l|l|}
        \hline
        $h$     & \begin{tabular}[c]{@{}l@{}}Method NC (BDM1): \\ no stabilization\end{tabular} & Method NC (BDM1) & \begin{tabular}[c]{@{}l@{}}Method NC (RT1):\\ no stabilization\end{tabular} & Method NC (RT1)  \\ \hline
        0.1     & 1.15723E-15  & 6.92112E-16  & 1.89594E-15   & 2.46866E-14  \\ 
        0.05    & 3.37823E-15  & 1.46118E-15  & 3.44825E-15   & 2.69386E-14  \\ 
        0.025   & 7.19689E-15  & 3.65184E-15  & 6.82505E-15   & 2.04061E-14  \\ 
        0.0125  & 1.56563E-14  & 6.21536E-15  & 1.36712E-14   & 2.09123E-14  \\ 
        0.00625 & 8.96546E-14  & 2.40478E-13  & 4.55812E-13   & 2.58988E-12  \\ \hline
    \end{tabular}
    \vspace{0.2cm}
    
	\hspace{-6cm}
	\begin{tabular}{|l|l|l|}
        \hline
        $h$     & \begin{tabular}[c]{@{}l@{}}Method C1 (RT1):\\ no stabilization\end{tabular} & Method C1 (RT1) \\ \hline
        0.1     & 1.82643E-15  & 2.05161E-14  \\ 
        0.05    & 3.49482E-15  & 2.86811E-14  \\ 
        0.025   & 6.84685E-15  & 2.10946E-14  \\ 
        0.0125  & 1.36738E-14  & 2.06419E-14  \\ 
        0.00625 & -            & 8.81791E-11  \\ \hline
    \end{tabular}
    \caption{Example 1: $L^2$-error $\|\dive \bfu- \dive \bfu_h\|_{L^2(\Omega)}$ for different mesh sizes $h$, $\lambda_{\bfu}=\lambda_{u_n}=4000$. For the finest mesh, i.e. $h=0.00625$, we had to use MUMPS instead of UMFPACK, resulting in significantly worse divergence error.
    \label{tab:ex1_div_L2}}
\end{table}

\begin{table}[ht]
    \centering
    \begin{tabular}{|l|l|l|l|l|}
        \hline
        $h$     & \begin{tabular}[c]{@{}l@{}}Method NC (BDM1): \\ no stabilization\end{tabular} & Method NC (BDM1) & \begin{tabular}[c]{@{}l@{}}Method NC (RT1):\\ no stabilization\end{tabular} & Method NC (RT1)\\ \hline
        0.1     & 6.40599E-14  & 2.44249E-15  & 1.24345E-14 &  2.43139E-13 \\ 
        0.05    & 2.19824E-13  & 6.21725E-15  & 4.44089E-14  & 5.16032E-13 \\ 
        0.025   & 1.2168E-13   & 1.42109E-14  & 6.39488E-14  & 4.91607E-13 \\ 
        0.0125  & 3.84803E-13  & 2.84217E-14  & 1.49214E-13 &  5.11591E-13 \\ 
        0.00625 & 9.80893E-12  & 1.53653E-11  & 1.22526E-10  & 1.24654E-09 \\ \hline
    \end{tabular}
    \vspace{0.2cm}
    
    \hspace{-6cm}
    \begin{tabular}{|l|l|l|}
        \hline
        $h$     & \begin{tabular}[c]{@{}l@{}}Method C1 (RT1):\\ no stabilization\end{tabular} & Method C1 (RT1) \\ \hline
        0.1     & 1.24345E-14  & 1.7053E-13  \\ 
        0.05    & 3.86358E-14  & 4.33875E-13  \\ 
        0.025   & 6.03961E-14  & 3.75255E-13  \\ 
        0.0125  & 1.42109E-13  & 6.83453E-13  \\ 
        0.00625 & -            & 2.30028E-08  \\ \hline
    \end{tabular}
    \caption{Example 1: Max-norm error $\|\dive \bfu- \dive \bfu_h\|_{L^\infty(\Omega)}$ for different mesh sizes $h$, $\lambda_{\bfu}=\lambda_{u_n}=4000$. For the finest mesh, i.e. $h=0.00625$, we had to use MUMPS instead of UMFPACK, resulting in significantly worse divergence error.
    \label{tab:ex1_div_Linf}}
\end{table}

\begin{table}[ht]
	\centering
	\resizebox{\textwidth}{!}{
	\begin{tabular}{|l|l|l|l|l|l|l|}
	\hline
	h      & $\|p-p_h\|_{L^2(\Omega)}$ & $p_h$-convergence & $\|\bfu-\bfu_h\|_{L^2(\Omega)}$ & $\bfu_h$-convergence & $\|\dive(\bfu-\bfu_h)\|_{L^\infty(\Omega)}$ & Condition number \\ \hline
	0.1    & 0.730942  & -        & 0.104193   & -        & 1.15463E-14 & 5.74E+03     \\
	0.05   & 0.365635  & 0.999352 & 0.0456878  & 1.18938  & 1.12133E-14 & 1.38E+04     \\
	0.025  & 0.141565  & 1.36894  & 0.0162839  & 1.48837  & 1.06581E-14 & 2.36E+04     \\
	0.0125 & 0.0527105 & 1.4253   & 0.00565044 & 1.52701  & 3.19744E-14 & 7.55E+04     \\ \hline
	\end{tabular}}
	\caption{Example 1: Method C2 using $P_1\times \RT_0\times Q_0\times Q_0^{\Gamma}$. Dirichlet boundary conditions are imposed weakly using a Lagrange multiplier in the space $Q_0^{\Gamma}$. }
	\label{tab:fourfield1}
\end{table}
\begin{table}[ht]
	\centering
	\resizebox{\textwidth}{!}{
	\begin{tabular}{|l|l|l|l|l|l|l|}
	\hline
	h      & $\|p-p_h\|_{L^2(\Omega)}$ & $p_h$-convergence & $\|\bfu-\bfu_h\|_{L^2(\Omega)}$ & $\bfu_h$-convergence & $\|\dive(\bfu-\bfu_h)\|_{L^\infty(\Omega)}$ & Condition number \\ \hline
	0.1    & 0.112782    & -       & 0.0173769   & -       & 2.2804E-13  & 1.78E+06     \\
	0.05   & 0.0159967   & 2.8177  & 0.00241831  & 2.8451  & 4.83613E-13 & 5.20E+06     \\
	0.025  & 0.00268519  & 2.57467 & 0.000450208 & 2.42534 & 4.93605E-13 & 1.55E+07     \\
	0.0125 & 0.000519604 & 2.36954 & 0.000100538 & 2.16285 & 7.22977E-13 & 5.93E+07     \\ \hline
	\end{tabular}}
	\caption{Example 1: Method C2 using $P_2\times \RT_1\times Q_1\times Q_1^{\Gamma}$. Dirichlet boundary conditions are imposed weakly using a Lagrange multiplier in the space $Q_1^{\Gamma}$. }
	\label{tab:fourfield2}
\end{table}

\begin{table}[ht]
	\resizebox{\textwidth}{!}{
	\begin{tabular}{l|llll|llll}
	& \begin{tabular}[c]{@{}l@{}}Method NC,\\ $h=0.025$\end{tabular} & (fitted) & &  & \begin{tabular}[c]{@{}l@{}}Method C1,\\ $h=0.025$\end{tabular} & (fitted) & & \\ \cline{2-9} 
	& \multicolumn{1}{l|}{$ \| \nabla\cdot\mathbf{u}_h\|_{L^{\infty}(\Omega)}$} & \multicolumn{1}{l|}{$ \| \mathbf{u}_h\|_{L^{2}(\Omega)} $} & \multicolumn{1}{l|}{$ \| \nabla\mathbf{u}_h\|_{L^{2}(\Omega)}$} & $\kappa$ & \multicolumn{1}{l|}{$ \| \nabla\cdot\mathbf{u}_h\|_{L^{\infty}(\Omega)}$} & \multicolumn{1}{l|}{$ \| \mathbf{u}_h\|_{L^{2}(\Omega)} $} & \multicolumn{1}{l|}{$ \| \nabla\mathbf{u}_h\|_{L^{2}(\Omega)}$} & \multicolumn{1}{l|}{$\kappa$} \\ \hline
	\multicolumn{1}{|l|}{\begin{tabular}[c]{@{}l@{}}$\lambda_{\mathbf{u}}=\lambda_{u_n}=10^5$\\ \textbf{Ra$=\mathbf{10^2}$}\end{tabular}}  & \multicolumn{1}{l|}{2.7105E-20} & \multicolumn{1}{l|}{4.6872E-07} & \multicolumn{1}{l|}{8.1299E-06} & 5.75E+11 & \multicolumn{1}{l|}{3.4594E-05} & \multicolumn{1}{l|}{5.3180E-07} & \multicolumn{1}{l|}{3.5396E-06} & \multicolumn{1}{l|}{1.03E+10}         \\ \hline
	\multicolumn{1}{|l|}{\begin{tabular}[c]{@{}l@{}}$\lambda_{\mathbf{u}}=\lambda_{u_n}=10^5$\\ \textbf{Ra$=\mathbf{10^6}$}\end{tabular}}  & \multicolumn{1}{l|}{2.7756E-16} & \multicolumn{1}{l|}{0.0047}  & \multicolumn{1}{l|}{0.0813}  & 5.75E+11 & \multicolumn{1}{l|}{0.3459}   & \multicolumn{1}{l|}{0.0053} & \multicolumn{1}{l|}{0.0354} & \multicolumn{1}{l|}{1.03E+10}         \\ \hline\hline
	\multicolumn{1}{|l|}{\begin{tabular}[c]{@{}l@{}}$\lambda_{\mathbf{u}}=\lambda_{u_n}=10^5$Ra\\ \textbf{Ra$=\mathbf{10^2}$}\end{tabular}} & \multicolumn{1}{l|}{3.1764E-22} & \multicolumn{1}{l|}{4.6873E-09} & \multicolumn{1}{l|}{8.1323E-08} & 5.53E+15 & \multicolumn{1}{l|}{3.4594E-07} & \multicolumn{1}{l|}{5.3180E-09} & \multicolumn{1}{l|}{3.5396E-08} & \multicolumn{1}{l|}{1.03E+12}         \\ \hline
	\multicolumn{1}{|l|}{\begin{tabular}[c]{@{}l@{}}$\lambda_{\mathbf{u}}=\lambda_{u_n}=10^5$Ra\\ \textbf{Ra$=\mathbf{10^6}$}\end{tabular}} & \multicolumn{1}{l|}{1.6652E-15} & \multicolumn{1}{l|}{4.6873E-09} & \multicolumn{1}{l|}{8.1323E-08} & 5.53E+23 & \multicolumn{1}{l|}{3.4594E-07} & \multicolumn{1}{l|}{5.3180E-09} & \multicolumn{1}{l|}{3.5396E-08} & \multicolumn{1}{l|}{1.03E+16}         \\ \hline\hline
	\multicolumn{1}{|l|}{\begin{tabular}[c]{@{}l@{}}Strong BC, $\lambda_{\mathbf{u}}=1$\\ \textbf{Ra$=\mathbf{10^2}$}\end{tabular}}  & \multicolumn{1}{l|}{6.7954E-18} & \multicolumn{1}{l|}{2.0053E-16} & \multicolumn{1}{l|}{4.1830E-15} & 2.04E+06 & \multicolumn{1}{l|}{2.9258E-25} & \multicolumn{1}{l|}{2.8336E-16} & \multicolumn{1}{l|}{7.1644E-27} & \multicolumn{1}{l|}{7.66E+04}         \\ \hline
	\multicolumn{1}{|l|}{\begin{tabular}[c]{@{}l@{}}Strong BC, $\lambda_{\mathbf{u}}=1$\\ \textbf{Ra$=\mathbf{10^6}$}\end{tabular}}  & \multicolumn{1}{l|}{2.5372E-13} & \multicolumn{1}{l|}{3.9319E-12}  & \multicolumn{1}{l|}{4.5678E-11}  & 2.04E+06 & \multicolumn{1}{l|}{4.2443E-20}   & \multicolumn{1}{l|}{1.3475E-12} & \multicolumn{1}{l|}{1.0277E-21} & \multicolumn{1}{l|}{7.66E+04}         \\ \hline
	\end{tabular}  }
	\caption{Example 3: Results for Method NC ($\BDM_1\times Q_0$) and Method C1 ($P_1\times\RT_0\times Q_0$) for $h=0.025$ on a fitted mesh. Shown in the top four rows is when the boundary condition is imposed weakly everywhere via penalty. In the bottom two rows, the normal component of the velocity is imposed strongly. Note that Method C1 does not have a penalty term for imposing the tangential component of the velocity, since it is imposed as a Neumann condition.}\label{tab:weakweak_fitted}
\end{table}

\begin{figure}[ht]  	
	\centering	
	\begin{subfigure}[b]{0.45 \textwidth}  	
		\centering	
		\includegraphics[width=\textwidth]{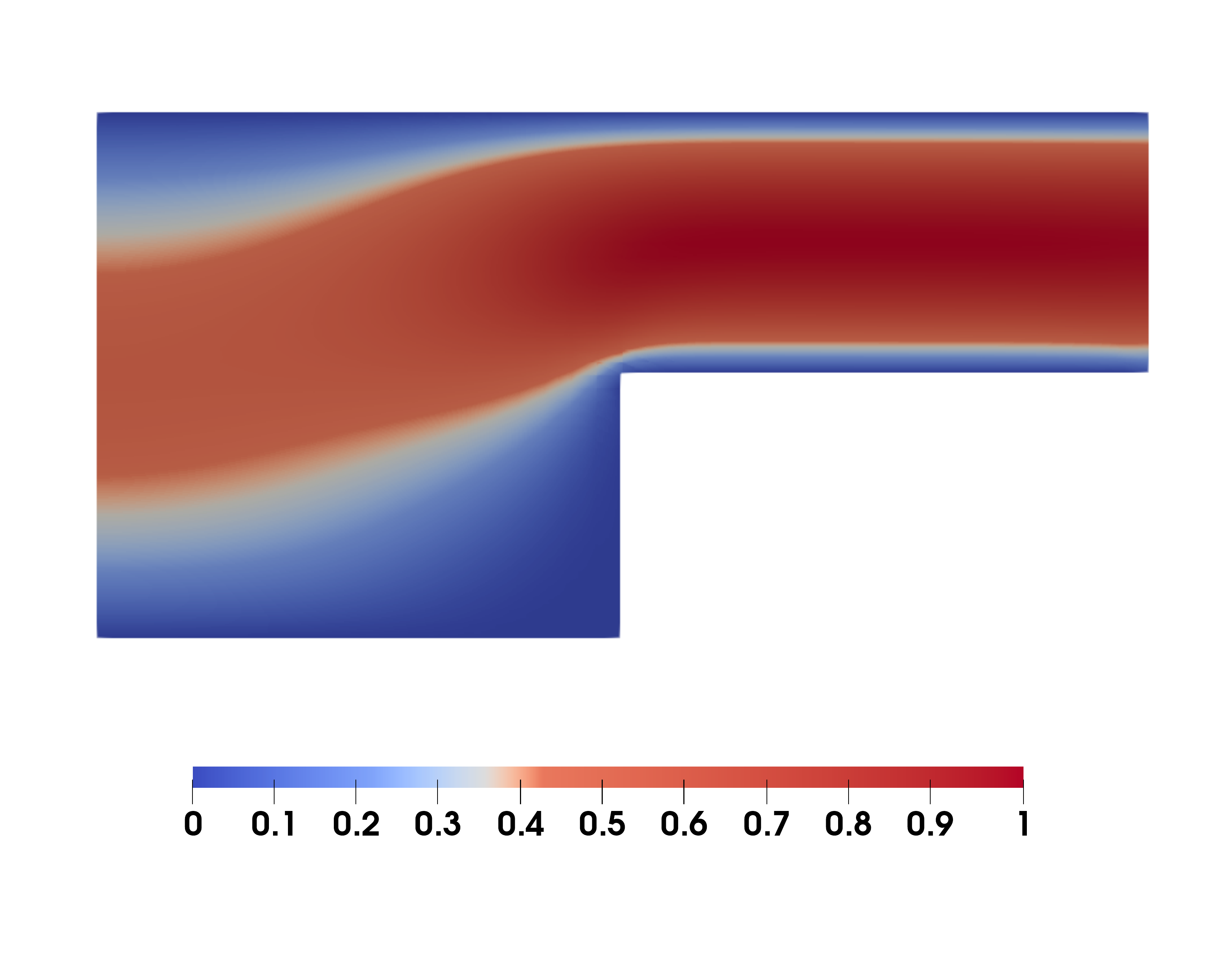}
			\end{subfigure}
	\begin{subfigure}[b]{0.45 \textwidth}  	
		\centering	
		\includegraphics[width=\textwidth]{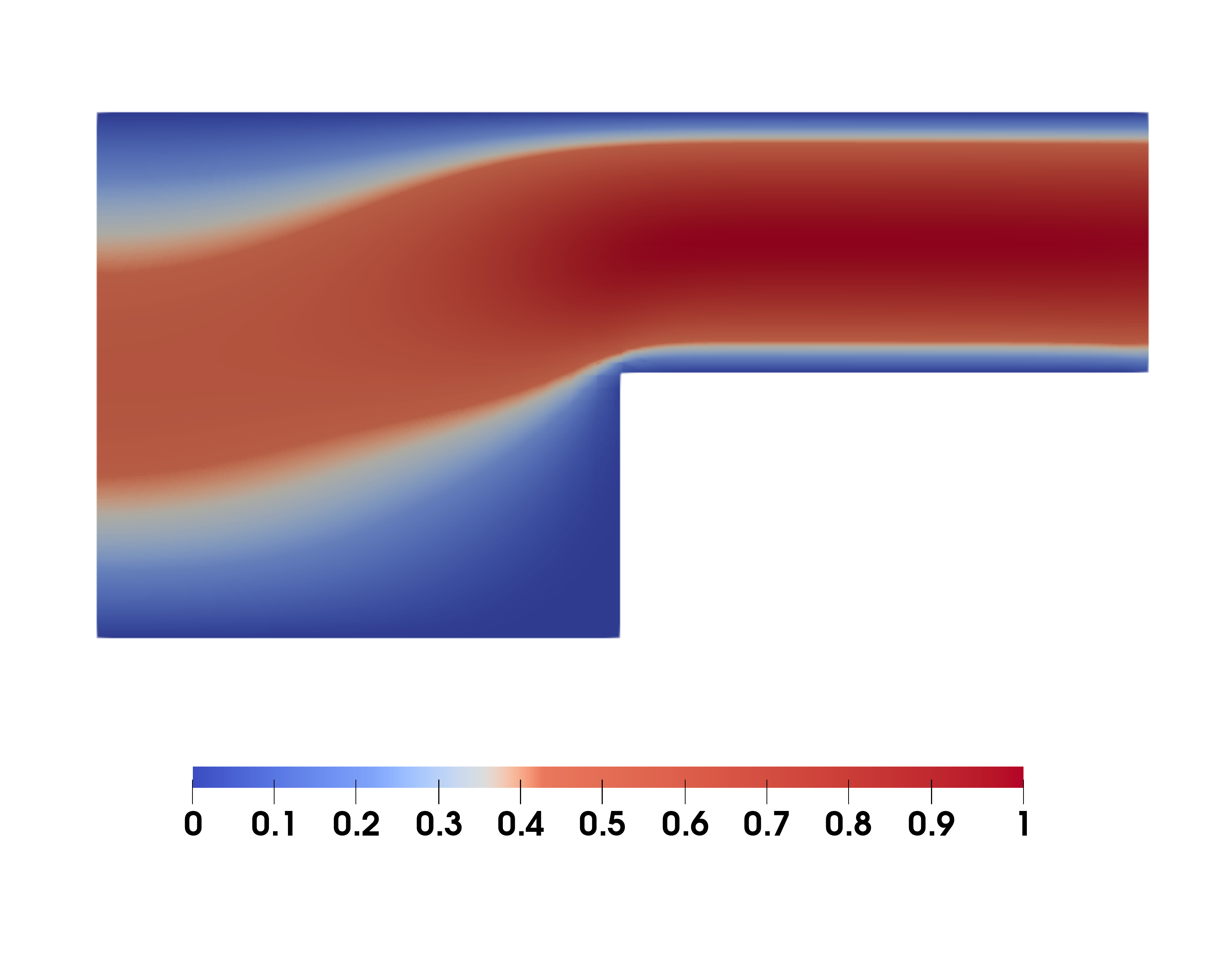}
			\end{subfigure}\\
	\begin{subfigure}[b]{0.45 \textwidth}  	
		\centering	
		\includegraphics[width=\textwidth]{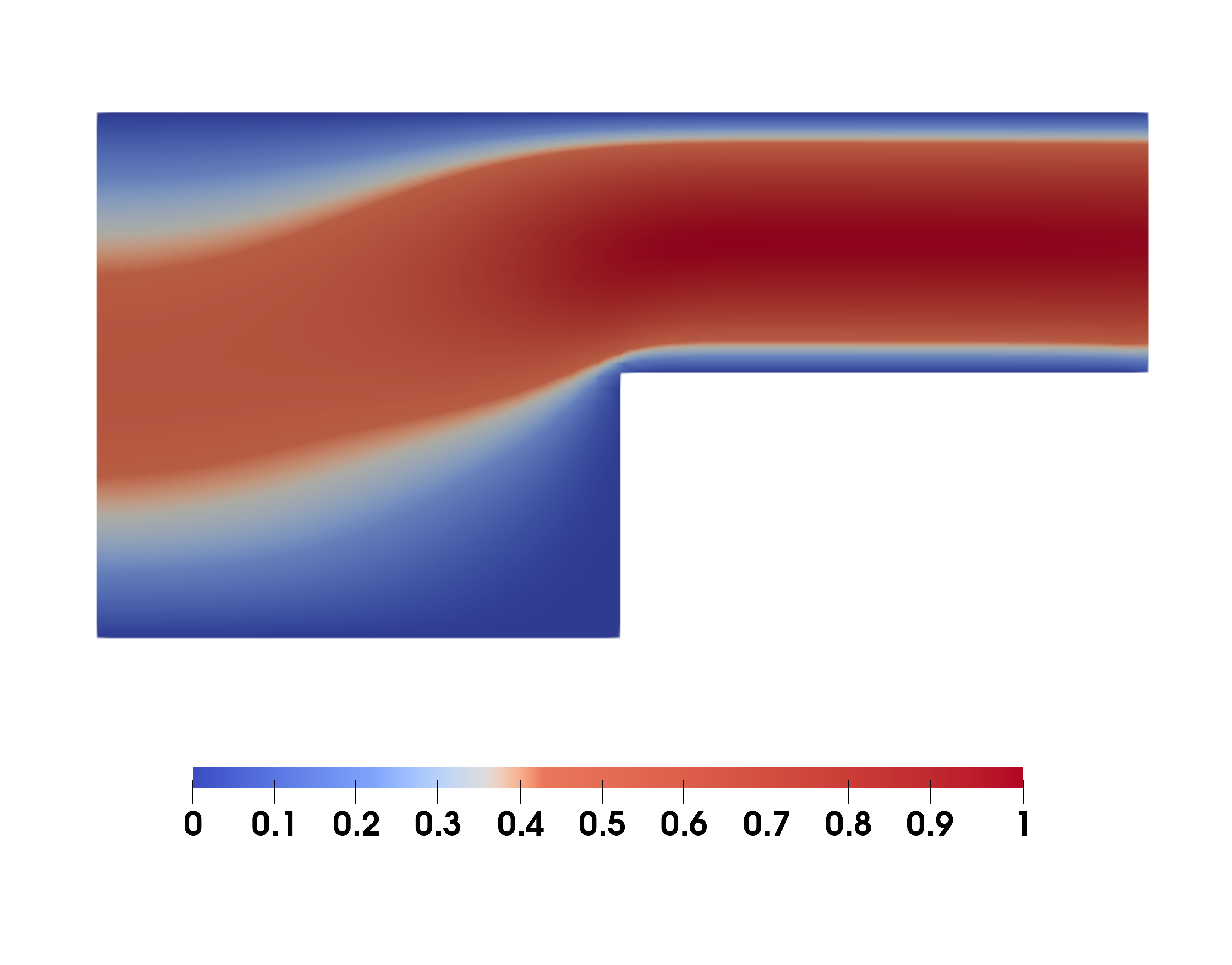}
	\end{subfigure}
	\begin{subfigure}[b]{0.45 \textwidth}  	
		\centering	
		\includegraphics[width=\textwidth]{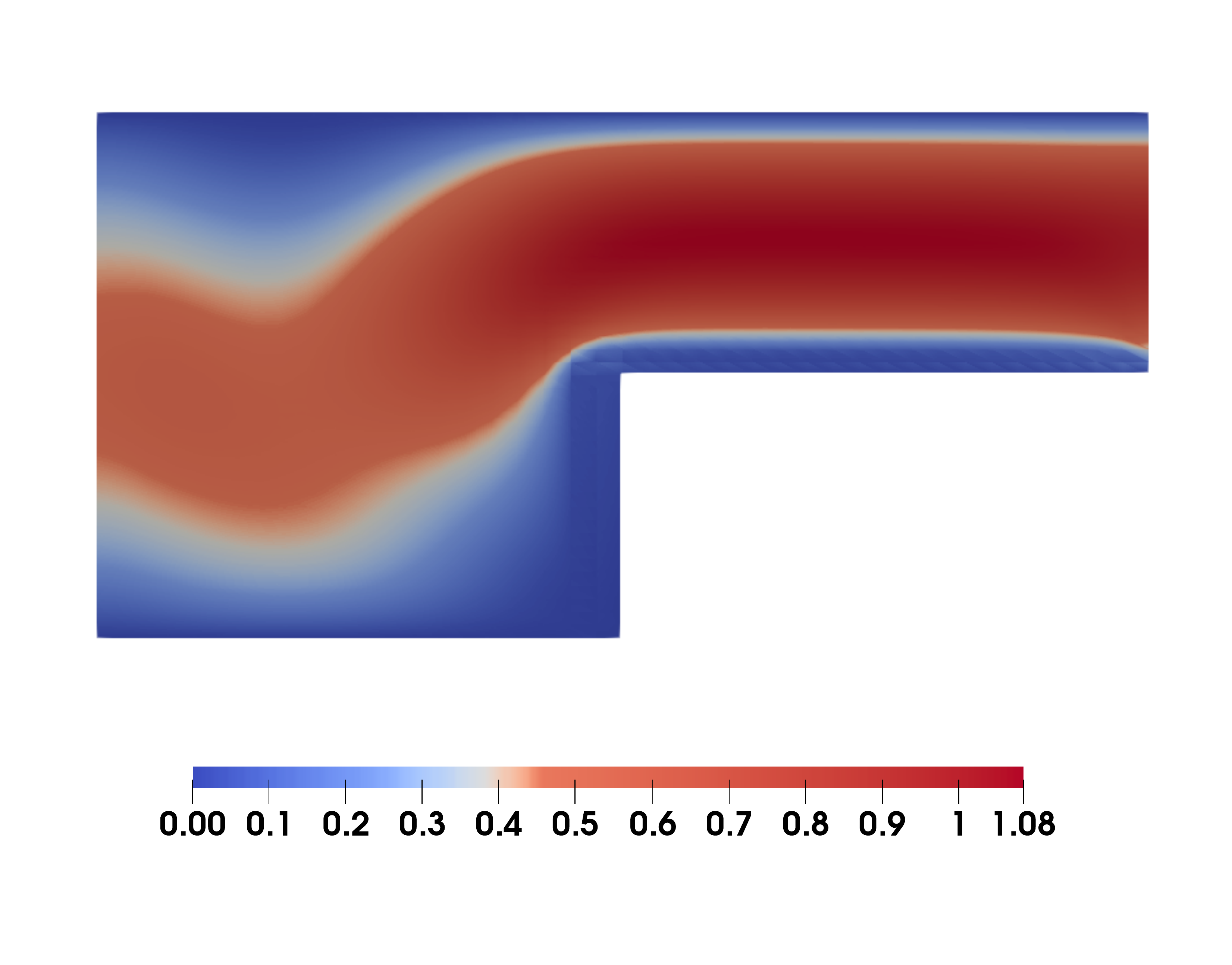}
	\end{subfigure}	\caption{ 
		Example 2: The magnitude of the approximate velocity field using Method C1 ($P_2\times\RT_1\times Q_1$), $\Lambda=0$ in the left panels and $\Lambda=2000$ in the right panels. 
		Top:  Results using the proposed method.
		Bottom: Results using standard approach. 
		\label{fig:presrob_RT1_unfitted} }            
\end{figure}

\section*{Acknowledgments}
The authors are grateful to the anonymous reviewers for carefully reading our manuscript and providing valuable comments that helped us to improve this work.

\addcontentsline{toc}{section}{Bibliography}
\bibliographystyle{spmpsci}
\bibliography{references.bib}

\end{document}